\documentclass[a4paper,11pt]{article}

\usepackage[english]{babel}
\usepackage[english]{translator}
\usepackage[T1]{fontenc}
\usepackage[utf8]{inputenc}

\usepackage{amsmath}
\usepackage{amssymb}
\usepackage{amsfonts}
\usepackage{amstext}
\usepackage{amsthm}
\usepackage{mathtools}
\usepackage{bbm}

\usepackage{graphicx}
\usepackage{xcolor}
\usepackage{psfrag}
\usepackage{pgfplots}
\usepackage{adjustbox}
\usepackage{subfig}
\usepackage{colortbl}

\usepackage{titlesec}
\usepackage{listings}
\usepackage{enumerate}
\usepackage{enumitem}
\usepackage[makeroom]{cancel}
\usepackage{todonotes}
\definecolor{asparagus}{rgb}{0.53, 0.66, 0.42}
\definecolor{ballblue}{rgb}{0.13, 0.67, 0.8}
\definecolor{cadmiumgreen}{rgb}{0.0, 0.42, 0.24}
\definecolor{cobalt}{rgb}{0.0, 0.28, 0.67}
\definecolor{darklavender}{rgb}{0.45, 0.31, 0.59}
\definecolor{green(pigment)}{rgb}{0.0, 0.65, 0.31}
\definecolor{blue(ncs)}{rgb}{0.0, 0.53, 0.74}
\definecolor{brandeisblue}{rgb}{0.0, 0.44, 1.0}
\definecolor{darkterracotta}{rgb}{0.8, 0.31, 0.36}
\definecolor{cobalt}{rgb}{0.0, 0.28, 0.67} 
\definecolor{ceruleanblue}{rgb}{0.16, 0.32, 0.75}
\definecolor{dodgerblue}{rgb}{0.12, 0.56, 1.0} 

\definecolor{color0}{rgb}{0.12156862745098,0.466666666666667,0.705882352941177}
\definecolor{color1}{rgb}{1,0.498039215686275,0.0549019607843137}
\definecolor{color2}{rgb}{0.172549019607843,0.627450980392157,0.172549019607843}
\definecolor{color3}{rgb}{0.83921568627451,0.152941176470588,0.156862745098039}
\definecolor{color4}{rgb}{0.580392156862745,0.403921568627451,0.741176470588235}
\definecolor{color5}{rgb}{0.549019607843137,0.337254901960784,0.294117647058824}
\definecolor{color6}{rgb}{0.890196078431372,0.466666666666667,0.76078431372549}

\definecolor{myBlue3}{RGB}{60,124,155} 
\definecolor{myForestGreen}{RGB}{34,139,34}

\usepackage{ulem}

\newcommand{\gradEh}{\dot{E}_h}
\newcommand{\ba}{\boldsymbol a}
\newcommand{\bb}{\boldsymbol b}
\newcommand{\bx}{\boldsymbol x}
\newcommand{\by}{\boldsymbol y}

\newcommand{\zf}[1]{z_{#1}} 

\usepackage{hyperref}

\usepackage{geometry}
\usepackage{tabularx}
\makeatletter
\def\hlinewd#1{%
\noalign{\ifnum0=`}\fi\hrule \@height #1 %
\futurelet\reserved@a\@xhline}
\makeatother

\hypersetup{
  pdffitwindow=false,
  pdfhighlight=/O,
  pdfnewwindow,
  colorlinks=true,
  citecolor=red,             
  linkcolor=blue,            
  menucolor=blue,            
  urlcolor=blue,             
  pdfpagemode=UseOutlines,
  bookmarksnumbered=true,
  linktocpage,
  pdfkeywords={},
  pdfcreator={pdflatex},
  pdfproducer={LaTeX with hyperref}
}

\newcommand{\C}{\mathbb{C}}               
\newcommand{\R}{\mathbb{R}}               

\renewcommand{\Re}{\mathrm{Re}\,}          

\renewcommand{\i}{\mathrm{i}} 
\newcommand{\ci}{\mathrm{i}} 
\renewcommand{\d}{\,\mathrm{d}}

\DeclareMathOperator{\diam}{diam}

\newcommand{\Ku}{K_{\tau,\kappa}(u)} 

\newcommand{\Lzero}{\mathcal{L}_{0,u}}
\newcommand{\Lone}{\mathcal{L}_1}

\newcommand{\uhp}{u_{h}} 

\newcommand{\orthiu}{(\ci u)^{\perp}}

\renewcommand{\div}{\mathrm{div} \,}		
\newcommand{\curl}{\mathrm{curl} \,}

\newcommand{\bfA}{\boldsymbol{A}}

\newcommand{\Ccoe}{\rho(\kappa)} 

\newcommand{\Rh}{R_{h}} 

\newcommand{\quotes}[1]{``#1''}

\newcommand{\GD}{\Gamma_{\rm D}}

\newcommand{\Einv}{E^{\prime\prime}(u)|_{\orthiu}^{-1}}

\theoremstyle{definition}
\newtheorem{definition}{Definition}[section]

\newtheorem{remark}[definition]{Remark}

\theoremstyle{plain}
\newtheorem{theorem}[definition]{Theorem}
\newtheorem{lemma}[definition]{Lemma}
\newtheorem{corollary}[definition]{Corollary}
\newtheorem{proposition}[definition]{Proposition}
\newtheorem{conclusion}[definition]{Conclusion}

\allowdisplaybreaks

\begin{document}

\begin{center}
{\Large The pollution effect for FEM approximations of the\\ Ginzburg--Landau equation\renewcommand{\thefootnote}{\fnsymbol{footnote}}\setcounter{footnote}{0}
 \hspace{-3pt}\footnote{P. Henning acknowledges the support by the Deutsche Forschungsgemeinschaft through the grant HE 2464/7-1.}}
\end{center}

\begin{center}
{\large Th\'eophile Chaumont-Frelet\footnote[1]{
Inria, Univ. Lille, CNRS, UMR 8524 -- Laboratoire Paul Painlev\'e, 59000 Lille, France, \\ e-mail: \textcolor{blue}{theophile.chaumont@inria.fr}.}},
{\large Patrick Henning\footnote[2]{Department of Mathematics, Ruhr University Bochum, DE-44801 Bochum, Germany, \\ e-mail: \textcolor{blue}{patrick.henning@rub.de}.}}\\[2em]
\end{center}

\begin{abstract}
In this paper, we investigate the approximation properties of solutions to the Ginzburg--Landau equation (GLE) in finite element spaces. Special attention is given to how the errors are influenced by coupling the mesh size $h$ and the polynomial degree $p$ of the finite element space to the size of the so-called Ginzburg-Landau material parameter $\kappa$. As observed in previous works, the finite element approximations to the GLE are suffering from a numerical pollution effect, that is, the best-approximation error in the finite element space converges under mild coupling conditions between $h$ and $\kappa$, whereas the actual finite element solutions possess poor accuracy in a large pre-asymptotic regime which depends on $\kappa$. In this paper, we provide a new error analysis that allows us to quantify the pre-asymptotic regime and the corresponding pollution effect in terms of explicit resolution conditions. In particular, we are able to prove that higher polynomial degrees reduce the pollution effect, i.e., the accuracy of the best-approximation is achieved under relaxed conditions for the mesh size. We provide both error estimates in the $H^1$- and the $L^2$-norm and we illustrate our findings with numerical examples.
\end{abstract}

\section{Introduction}
The Ginzburg-Landau equation (GLE) is a common model to describe superconductors  \cite{Ginzburg1955,Landau1965}, i.e., materials that allow to conduct electricity without electrical resistance. Mathematically, superconductors can be described by the order parameter $u : \Omega \rightarrow \C$ which is a complex function that acts on the body $\Omega \subset \R^d$ (for $d=2,3$) that is occupied by the superconducting material. From $u$, one can extract the density $|u|^2$ of superconducting charge carriers (so-called Cooper pairs). The density is a real-valued physical observable with $0 \le |u|^2 \le 1$, where $|u|^2=0$ represents the normal state (no superconducting charge carriers) and $|u|^2=1$ represents the perfectly superconducting state. The order parameter $u$ is given as a solution to the GLE which reads in its basic form, cf. \cite{DuGuPe},
\begin{align} 
\label{stationarGLE}	 (\tfrac{\i}{\kappa} \nabla  + \bfA )^{2} \hspace{1pt}u + (|u|^2-1) u &= 0  \qquad \mbox{in } \Omega,\\
\nonumber	 (\tfrac{\i}{\kappa} \nabla u + \bfA \, u) \cdot \mathbf{n} &=0 \qquad \mbox{on } \partial \Omega.
\end{align}
Here, $\i$ is the complex unit, $\kappa$ is a real-valued positive material parameter and $\bfA : \Omega \rightarrow \R^d$ is a vector potential for the magnetic field \,$\mathbf{H}=\curl \bfA$\, within the superconductor. Under the Coulomb gauge, the vector potential is chosen to satisfy $\div \bfA = 0$, together with the boundary condition $\bfA \cdot \mathbf{n} \vert_{\partial \Omega} = 0$. Equation \eqref{stationarGLE} is formally obtained as the first order optimality condition fulfilled by local and global minimizers of the Ginzburg--Landau energy.

The physically most interesting states of a superconductor are its mixed normal-superconducting states where the magnetic field penetrates the superconductor in \quotes{localized defects} (points for $d=2$, lines for $d=3$). These defects are vortices corresponding to circulating superconducting currents. The vortices can form complex structures and patterns known as Abrikosov vortex lattices \cite{Abr04}. The crucial parameter in this context is the material parameter $\kappa$. For fixed $\bfA \not= 0$ (and hence fixed magnetic field), the size of the individual vortices decreases as $\kappa$ increases, leading to a denser arrangement of vortices that has to be resolved numerically. At the same time, the diameter of the individual vortices is shrinking, cf. Figure \ref{fig:vortices_reference} in the numerical experiments for an illustration. Apparently, larger $\kappa$-values require finer computational meshes to resolve the vortices. Hence, $\kappa$ is related to some numerical resolution condition.

When solving the GLE \eqref{stationarGLE} in a finite element space $V_{h}$ of mesh size $h$ and polynomial degree $p$, we can ask: How do we have to select $h$ and $p$ relative to $\kappa$ to ensure meaningful numerical solutions? The question is easily answered for the best approximation, where we obtain with the stability estimate $\| u \|_{H^{s+1}(\Omega)} \lesssim \kappa^{s+1}$ for all sufficiently smooth solutions and a straightforward calculation that
\begin{align*}
\inf_{v_h \in V_{h}} \left(\|  u - v_h \|_{L^2(\Omega)} +\tfrac{1}{\kappa} \| \nabla u - \nabla v_h \|_{L^2(\Omega)}\right) \,\,\lesssim \,\, (h \kappa)^s,\qquad \mbox{for } s\le p.
\end{align*}
Hence, accurate numerical approximations require at least $h \kappa \lesssim 1$. However, in previous works \cite{BDH25,DoeHe24}, numerical experiments as well as analytical results for P1-FEM indicated that the practical resolution condition is much more severe when the actual finite element solution is considered and that there is a numerical pollution effect similar to what is known for Helmholtz equations with large wave number and related problems \cite{BaS97,bernkopf_chaumontfrelet_melenk_2025a,CFN20,lafontaine2022wavenumber,MeS10,Pet17}.

Error estimates on FEM approximations to the full Ginzburg--Landau equations (including an additional curl-curl problem for the vector potential) were first obtained in the seminal work by Du, Gunzburger and Peterson \cite{DuGuPe}. However, the parameter $\kappa$ was not yet traced in the estimates. The first attempt for $\kappa$-explicit error estimates for the GLE \eqref{stationarGLE} was given in \cite{DoeHe24}. In particular, it was proved that the finite element error behaves asymptotically as well as the $H^1$-best-approximation error. Later the findings were generalized to discrete multiscale spaces in \cite{BDH25} and \cite{DDH24}. Nevertheless, despite being $\kappa$-explicit in almost all arguments, a full tracing of the $\kappa$-dependency could not be achieved in the aforementioned works. This is due to the proof technique which exploits an absorption argument for higher order terms which in turn requires the mesh size to be \quotes{sufficiently small} without being able to quantify how small it actually needs to be. The reason is that the argument is based on abstract compactness results that just ensure convergence but no rates. Hence, a full $\kappa$-tracing is still open. Furthermore, the existing results on $\kappa$-dependencies are restricted to P1-FEM and it is unclear if potential resolution conditions can be relaxed in higher order FEM spaces.

It is worth to mention that there has been also significant work on FEM approximations for the time-dependent Ginzburg-Landau equation with vector potential, cf. \cite{Chen97,Du94b,Du97,GJX19,GaoS18,Li17,LiZ15,LiZ17,MaQ23}, however, there were no attempts to trace the dependency on $\kappa$ and to identify corresponding resolution conditions.\\[-0.5em]

In this work, we will give a significant extension of the previous results for the GLE \eqref{stationarGLE} by (i) providing a fully traced resolution condition under which the finite element error in $p$-FEM spaces behaves like the best-approximation error and (ii) proving that the resolution condition (for achieving the accuracy of the best-approximation) indeed relaxes with increasing polynomial degrees $p$. Furthermore (iii), compared to previous works, we obtain improved $L^2$-error estimates and our analysis is no longer restricted to global minimizers. 

To achieve this, we follow a new proof strategy that combines different techniques. On the one hand, we exploit that the GLE originates from an energy minimization problem, i.e. $u$ is a local minimizer of a (non-convex) energy functional $E(v)$. By the second order optimality condition, $E^{\prime\prime}(u)$ has a bounded inverse on some suitable subspace of $H^1(\Omega)$ and we can define the Galerkin-projection $\Rh$ associated with $E^{\prime\prime}(u)$ on the finite element space. We then split the error between an exact minimizer $u$ and a FEM minimizer $\uhp$ into $u-\uhp = (u-\Rh(u)) + (\Rh(u) - \uhp)$, investigating both error contributions individually.

For the first part, even though the operator $E^{\prime\prime}(u)$ was also used in previous works  \cite{DoeHe24,BDH25} the corresponding Galerkin-projection was avoided due to, at first glance, unfavourable $\kappa$-dependent stability bounds entering through $E^{\prime\prime}(u)$. We circumvent these issues by an abstract splitting technique proposed in \cite{CFN20} which allows us to decompose functions $E^{\prime\prime}(u)^{-1}g$ into a regular oscillatory contribution and a low-regularity contribution that behaves nicely for increasing $\kappa$-values. With this, we will be able to bound $u-\Rh(u)$ with optimal order estimates under suitable resolution conditions.

The second component of the error also requires dedicated new arguments.
To estimate $\Rh(u) - \uhp$, we establish an asymptotic superconvergence result with which we can bound $\Rh(u) - \uhp$ by $u-\uhp$ up to higher order contributions. The higher order terms are controlled by a Banach fixed-point argument which allows us to identify a resolution condition under which $\Rh(u) - \uhp$ becomes so small that it can be absorbed into $u-\uhp$ and ultimately yielding  $\|  u- \uhp \|_{H^1_{\kappa}(\Omega)} \lesssim  \| u - \Rh(u) \|_{H^1_{\kappa}(\Omega)}$ in a suitable $\kappa$-weighted $H^1$-norm.\\[0.5em]
{\it Overview.} In Section \ref{section:analytical-setting} we introduce the basic notation and the precise mathematical setting of the Ginzburg-Landau equation. Furthermore, we revisit some important results on the existence, local uniqueness, and regularity of corresponding solutions. In Section \ref{section:discrete-minimizers} we define the finite element approximations and we present our main result on quasi-optimality under a suitable resolution condition for the mesh size that relaxes with higher order finite element spaces. The proof of the main result is split into two major parts covering Section \ref{section-secE-shift-prop} and Section \ref{section:proofs:L2-H1-estimates}. Section \ref{section-secE-shift-prop} deals with the aforementioned splitting of solutions involving the operator $E^{\prime\prime}(u)$ and Section \ref{section:proofs:L2-H1-estimates} covers the core error estimates. In Section \ref{section:nonlinear-cg-method} we propose a nonlinear conjugate gradient method that is used for the numerical experiments presented afterwards in Section \ref{section:numerical-experiments}. In
Appendix \ref{appendix:section:quasi-interpolation}, we also present the construction and approximation properties of an Oswald-type quasi-interpolation operator adapted to curved meshes that is used as a tool in our error analysis.

\section{Analytical setting}
\label{section:analytical-setting}

\subsection{Sobolev spaces}

Throughout this manuscript, $\Omega \subset \mathbb{R}^d$ (for $d=2,3$)
denotes a fixed domain with a Lipschitz boundary $\partial \Omega$.
We assume for simplicity that the domain has a unit diameter,
which can always be achieved by rescaling.

For $2 \leq q \leq +\infty$, $L^q(\Omega)$ is the usual Lebesgue space of (complex-valued) integrable functions, and we denote by
$\|\cdot\|_{L^q(\Omega)}$ its usual norm. For $q=2$, we
equip $L^2(\Omega)$ with the real inner product 
$$ 
(v,w)_{L^2(\Omega)}:= \Re \Big( \int_{\Omega} v \, \overline{w} \d x \Big) 
\qquad \mbox{for } v,w \in L^2(\Omega),
$$
where $\overline{w}$ is the complex conjugate of $w$. This makes $L^2(\Omega)$
formally an $\R$-Hilbert space (consisting of complex-valued functions).
We also employ the same notation for vector fields, in which case
a dot product is used in the inner product.

For any integer $k \geq 1$, we employ the notation
$W^{k,q}(\Omega)$ for the usual Sobolev spaces of functions belonging to
$L^q(\Omega)$ together with all their partial derivatives of order less
than or equal to $k$. Classically, $|\cdot|_{W^{k,q}(\Omega)}$ is the semi-norm
of $W^{k,q}(\Omega)$ collecting the $L^q(\Omega)$ norms of the partial derivatives
of order $k$. We also employ the shorthand notation $H^1(\Omega) := W^{1,2}(\Omega)$
for the first-order Hilbert Sobolev space, which we equip with the real inner product given
by $(v,w)_{L^2(\Omega)} + (\nabla v,\nabla w)_{L^2(\Omega)}$ for $v,w \in H^1(\Omega)$.

Together with the standard Sobolev norm
\begin{equation*}
\|v\|_{W^{k,q}(\Omega)} := \sum_{\ell = 0}^k |v|_{W^{\ell,q}(\Omega)},
\end{equation*}
we will use the weighted version
\begin{equation*}
\|v\|_{W^{k,q}_\kappa(\Omega)} := \sum_{\ell = 0}^k \kappa^{-\ell} |v|_{W^{\ell,q}(\Omega)}.
\end{equation*}
The same convention applies to $\| v \|_{W^{k,\infty}_{\kappa}(\Omega)}$.

Finally, $W^{k,q}(\Omega,\mathbb{R}^d)$ denotes the space of
real-valued vector fields with every component belonging to $W^{k,q}(\Omega)$.

\subsection{PDE coefficients}

Throughout the manuscript we consider a vector field $\bfA \in L^\infty(\Omega,\mathbb{R}^d)$
representing a magnetic potential. We are free to choose a gauge condition, since only the
curl of $\bfA$ has physical relevance. We therefore assume without loss of generality that $\div \bfA = 0$ in $\Omega$, together with the boundary condition $\bfA \cdot \mathbf{n} = 0$ on $\partial \Omega$ (the Coulomb gauge with a homogeneous Neumann boundary condition, cf. \cite{DuGuPe}). 
We will sometimes require more
regularity on $\bfA$, and specifically, we demand that
$\bfA \in W^{p-1,\infty}(\Omega,\mathbb{R}^d)$.

The material parameter is constant and represented by the real number $\kappa$.
Since we are especially interested in the high $\kappa$ regime, we will assume without loss
of generality that $\kappa \geq 1$.

\begin{remark}[Piecewise smooth coefficient]
In principle, one could consider a piecewise smooth coefficient $\bfA$
without major modifications. For the sake of simplicity, we avoid doing
so here, but refer the reader to~\cite{CFN20} for details. Similarly,
(piecewise) smooth material coefficients $\kappa$ could be considered.
\end{remark}

\subsection{Hidden constant}
\label{subsection:hidden-constant}

Throughout this manuscript, we fix a polynomial degree $p \geq 1$, a constant $M > 0$ measuring
the shape-regularity of the finite element mesh, and an arbitrarily small parameter
$0 < \varepsilon \leq 1$, that is fixed throughout the manuscript. We will write $a \lesssim b$ if there is a constant $C>0$ solely depending
on $\Omega$, $\bfA$, $p$, $M$ and $\varepsilon$ such that $a \le C\, b$.
We also use the notation $a \gtrsim b$ if $a \ge C \, b$.
The hidden constant depends on $\bfA$ only through $\|\bfA\|_{W^{p-1,\infty}_{\kappa}(\Omega)}$,
and only through $\|\bfA\|_{L^\infty(\Omega)}$ in the results where we do not assume that
$\bfA$ is smooth. We note in particular that such a constant $C$ can be bounded independently
of $\kappa$ and the mesh size $h$. Physically, the dependency on $\bfA$ suggests that the amplitude
of the magnetic potential is moderate, but $\bfA$ is allowed to oscillate with frequency
at most $\kappa$.

At a limited number of places, if $\alpha,\beta,\gamma,\ldots$ are previously introduced symbols,
we will use the notation $a \lesssim_{\alpha,\beta,\gamma,\ldots} b$ to indicate that the constant
$C$ may also depend on $\alpha,\beta,\gamma,\ldots$.

\subsection{Functional inequalities}

In this section, we collect some useful inequalities linked
with Sobolev norms.

\begin{subequations}
\label{eq_product_rules}
We are first interested in the regularity of products of functions.
The following estimates are easily obtained by combining the product
rule for derivatives together with the H\"older inequality.
For all $m \geq 0$ and $s,t \in [2,+\infty]$, we have
\begin{equation}
\|uv\|_{W^{m,q}_\kappa(\Omega)}
\lesssim_{m,s,t}
\|u\|_{W^{m,s}_\kappa(\Omega)}
\|v\|_{W^{m,t}_\kappa(\Omega)}
\end{equation}
for all $u \in W^{m,s}(\Omega)$, $v \in W^{m,t}(\Omega)$ and
$1/q = 1/s+1/t$. Applying this inequality twice, for all $q \geq 1$,
we can further show that
\begin{equation}
\|uvw\|_{W^{m,q}_\kappa(\Omega)}
\lesssim_{m,q}
\|u\|_{W^{m,3q}_\kappa(\Omega)}
\|v\|_{W^{m,3q}_\kappa(\Omega)}
\|w\|_{W^{m,3q}_\kappa(\Omega)}
\end{equation}
for all $u,v,w \in W^{m,3q}(\Omega)$ and
\begin{equation}
\|uvw\|_{W^{m,q}_\kappa(\Omega)}
\lesssim_{m,q}
\|u\|_{W^{m,\infty}_\kappa(\Omega)}
\|v\|_{W^{m,\infty}_\kappa(\Omega)}
\|w\|_{W^{m,q}_\kappa(\Omega)}
\end{equation}
whenever $u,v \in W^{m,\infty}(\Omega)$ and $w \in W^{m,q}(\Omega)$.
\end{subequations}

We then turn to Sobolev inequalities for a single function.
By combining the fractional Sobolev embedding with a Banach
space real interpolation inequality (see e.g. \cite[Theorem 8.2]{nezza_palatucci_valdinoci_2012a},
\cite[Lemma 36.1]{tartar_2007a} and \cite[Lemma B.1]{mclean_2000a}), 
we have
\begin{equation}
\label{eq_injection_Linfty}
\|\phi\|_{L^\infty(\Omega)}
\lesssim_q
\|\phi\|_{L^q(\Omega)}^{1-s}\|\phi\|_{W^{1,q}(\Omega)}^{s}
\lesssim_q
\kappa^s \|\phi\|_{W^{1,q}_\kappa(\Omega)}
\end{equation}
for all $\phi \in W^{1,q}(\Omega)$ whenever $0 < s < 1$ and $sq > d$.
Similarly, we have the Gagliardo--Nirenberg inequalities
\begin{equation}
\label{eq_gagliardo_nirenberg_L4}
\|\phi\|_{L^4(\Omega)}
\lesssim
\|\phi\|_{L^2(\Omega)}^{1-d/4}\|\nabla \phi\|_{L^2(\Omega)}^{d/4}
\lesssim
\kappa^{d/4} \|\phi\|_{H^1_{\kappa}(\Omega)}
\end{equation}
for all $\phi \in H^1(\Omega)$, and
\begin{equation}
\label{eq_gagliardo_nirenberg_L2q}
\|\nabla \phi\|_{L^{2q}(\Omega)}
\lesssim
\|\phi\|_{L^\infty(\Omega)}^{1/2}\|\phi\|_{W^{2,q}(\Omega)}^{1/2}
\lesssim
\kappa \|\phi\|_{L^\infty(\Omega)}^{1/2}\|\phi\|_{W^{2,q}_{\kappa}(\Omega)}^{1/2}.
\end{equation}
for all $q \geq 1$ and $\phi \in W^{2,q}(\Omega)$, see e.g. \cite{BrezisMironescu18}.

\subsection{Minimization problem}

The (fixed-field) Ginzburg--Landau energy functional $E: H^1(\Omega) \rightarrow \R$ is defined by
\begin{equation} 
E(v) = \frac{1}{2} \int_{\Omega} \left|  \tfrac{\i}{\kappa} \nabla v + \bfA\, v \right|^2 + \frac{1}{2} (|v|^2-1)^2 \, \d x
\label{eq_GLenergy}
\end{equation}
and we seek the corresponding lowest energy states of the system, i.e., minimizers $u \in H^1(\Omega)$ with
\begin{align}
\label{minimizer-energy-def} 
E(u) \,\,= \min_{v \in H^1(\Omega)} E(v).
\end{align}
The existence of such global minimizers and their fundamental properties were first
established in \cite{DuGuPe}. In this work, we shall not only consider global minimizers, but also isolated local minimizers of $E$ (in a sense made precise below) and we will show that in a small neighborhood of any such isolated local minimizer of $E$
there is a corresponding discrete local minimizer in the finite element space, provided that the mesh is sufficiently fine. As a first step, we formulate the conditions fulfilled by minimizers and properly introduce what {\it isolation} means in this context.

\subsection{First and second order conditions for local and global minimizers}
If $u$ is a local minimizer of the energy functional $E$, it is classical that the first order condition for minimizers implies $E^{\prime}(u)=0$ and the second order condition implies a non-negative spectrum for $E^{\prime\prime}(u)$. Here, $E^{\prime}$ and $E^{\prime\prime}$ denote the first and second Fr\'echet derivatives of $E$ which can be computed as 
\begin{align}
\label{derivaitves-1}
 \langle E^\prime (v), w \rangle \,\,=\,\, ( \tfrac{\i}{\kappa} \nabla v + \bfA v , \tfrac{\i}{\kappa} \nabla w + \bfA w )_{L^2(\Omega)} + (\, (\vert v \vert^2 -1) v , w )_{L^2(\Omega)} 
\end{align}
and
\begin{align}
\label{derivaitves-2}
 \langle E^{\prime\prime} (v)z, w \rangle \,\,=\,\, ( \tfrac{\i}{\kappa} \nabla z+ \bfA z , \tfrac{\i}{\kappa} \nabla w + \bfA w  )_{L^2(\Omega)} + ( \, (2 \vert v \vert^2 \hspace{-2pt} -\hspace{-2pt}1) z  + v^2 \overline{z}  ,  w \,)_{L^2(\Omega)} 
\end{align}
for $v,w,z \in H^1(\Omega)$.

Exploiting the formula for $E^\prime$, the first order condition $E^{\prime}(u)=0$ is nothing but the Ginzburg--Landau equation in variational form, i.e., $u \in H^1(\Omega)$ solves
\begin{align*}
 ( \tfrac{\i}{\kappa} \nabla u + \bfA u , \tfrac{\i}{\kappa} \nabla v + \bfA v )_{L^2(\Omega)} + (\, (\vert u \vert^2 -1) u , v )_{L^2(\Omega)} 
 \,\,=\,\,0 \qquad \mbox{for all } v\in H^1(\Omega).
\end{align*}
With the Gauge condition $\div \bfA =0$ and $\bfA \cdot \mathbf{n}\vert_{\partial \Omega}  = 0$, the Ginzburg-Landau equation can be written in strong form as
\begin{align}
\label{GLE-strong-form}
 - \tfrac{1}{\kappa^2} \Delta u  + \tfrac{2\i}{\kappa} \bfA \cdot \nabla u  + |\bfA|^2 u + (|u|^2-1)u = 0
\end{align}
accompanied by the homogeneous Neumann boundary condition 
\begin{align*}
\nabla u \cdot \mathbf{n} = 0 \qquad \mbox{on } \partial \Omega.
\end{align*}

Turning now to the necessary second order condition for minimizers,
we recall that $E^{\prime\prime}(u)$ must have a nonnegative spectrum.
One might wonder if the spectrum is strictly positive, so that $E^{\prime\prime}(u)$
has a bounded inverse. However, this is unfortunately not the case and zero is in
fact an eigenvalue of $E^{\prime\prime}(u)$. This is caused by the invariance of
the energy under complex phase shifts or, in other words, the missing local uniqueness
of minimizers. This is easily seen by considering an arbitrary minimizer $u$ and the
corresponding class of gauge-transformed states $\exp(\omega \ci ) u$ for any
$\omega \in [-\pi,\pi)$. A simple calculation shows $E(\exp(\omega \ci ) u)=E(u)$
for all $\omega$ and hence, $\exp(\omega \ci ) u$ is another minimizer.
Practically, this is not an issue because all these minimizers share the same
density $|u|=|\exp(\omega \ci ) u|$ and can therefore be considered equivalent.
Nevertheless, this requires special attention in the error analysis.

One of the direct implications is the aforementioned singularity of $E^{\prime\prime}(u)$.
In fact, if we move from $u$ along the circle line parametrized by $\gamma : \omega \mapsto \exp(\omega \ci ) u$, then the energy is constant along that line and the corresponding directional derivatives must vanish. In $u$ (i.e. $\omega=0$), the corresponding tangential direction is $\gamma^{\prime}(0)= \ci u$ and we therefore expect $E^{\prime\prime}(u)\ci u = 0$. Indeed, using \eqref{derivaitves-2}, a direct calculation shows
\begin{eqnarray*}
 \langle E^{\prime\prime} (u)\ci u, v \rangle &=& ( \tfrac{\i}{\kappa} \nabla (\ci u)+ \bfA (\ci u) , \tfrac{\i}{\kappa} \nabla v + \bfA v  )_{L^2(\Omega)} + ( \, (\vert u \vert^2 \hspace{-2pt} -\hspace{-2pt}1) (\ci u)  + u^2 \overline{\ci u} +\vert u \vert^2 \ci u ,  v \,)_{L^2(\Omega)} \\
 &\overset{\eqref{derivaitves-1}}{=}& \underbrace{\langle E^{\prime}(\ci u) , v\rangle}_{=0} + 2 ( \underbrace{\Re( \overline{u} \ci u)}_{=0} u , v )_{L^2(\Omega)} \,\,\, = \,\,\, 0
\end{eqnarray*}
for any $v\in H^1(\Omega)$. A realistic sufficient second order condition for local minimizers must therefore exclude the orthogonal complement of the kernel, i.e., the $L^2$-orthogonal complement of $\ci u$ in $H^1(\Omega)$ given by
\begin{align*}
\orthiu := \{ \, v \in H^1(\Omega) \mid (\i u, v)_{L^2(\Omega)} =0 \, \}. 
\end{align*}
Note that $u \in \orthiu$ since $(\cdot,\cdot)_{L^2(\Omega)}$ only includes real parts. In summary, we can hence work with the following sufficient second order condition as suggested in previous works (cf. \cite{DoeHe24,BDH25}):
\begin{enumerate}[resume,label={(A\arabic*)}]
\item\label{A4}  The state $u\in H^1(\Omega)$ is a local minimizer of $E$ that satisfies the necessary first order condition $E^{\prime}(u)=0$ and the sufficient second order condition
$$
\langle E^{\prime\prime}(u) v , v \rangle >0 \qquad \mbox{for all } v \in \orthiu \setminus \{ 0 \}.
$$
\end{enumerate}
Condition \ref{A4} implies that zero is a simple eigenvalue of $E^{\prime\prime}(u)$ and that the remaining spectrum is positive. Hence, such a local minimizer is isolated from any other local minimizer of $E$.

\begin{remark}[Physically relevant states]
Not every critical point of the Ginzburg--Landau energy represents an observable physical state. Only energetically favorable solutions can be expected to be physically realizable. In particular, global minimizers of $E$ correspond to the most stable equilibrium configurations and are therefore observed in practice. Isolated local minimizers satisfying Assumption~\ref{A4} are typically also physically relevant as metastable states, since they generally possess a sufficiently large basin of attraction and are dynamically stable. In contrast, saddle points or highly unstable critical points do not usually appear in experiments, except possibly as transient states during dynamical evolution. Hence, by restricting our analysis to states fulfilling Assumption~\ref{A4} (i.e., isolated local minimizers), we focus precisely on the physically most relevant configurations.
\end{remark}

Exploiting \ref{A4}, it can be proved that $E^{\prime\prime}(u)$ induces a coercive operator:
\begin{lemma}
\label{coercivity_secE_u}
If~\ref{A4} holds true, then the constant
\begin{align}
\label{Ccoe_gtrsim_1}
\Ccoe \,\, := \,\, \max_{v \in \orthiu \setminus \{0\}} \frac{\|v\|_{H^1_{\kappa}(\Omega)}^2}{\langle E^{\prime\prime}(u) v , v \rangle} 
\,\, \gtrsim \,\, 1
\end{align}
is finite, and only depends on $u$ and $\kappa$.
\end{lemma}
A proof that $\rho(\kappa)$ is finite can be found in \cite[Proposition 2.2]{BDH25}.
Note that the explicit dependency of $\Ccoe$ on $\kappa$ is unknown, but numerical experiments
suggest a polynomial growth $\kappa^{\alpha}$ for some $\alpha>0$ (with typically
$\alpha = \mathcal{O}(1)$, cf. \cite{DoeHe24,BDH25}). The fact that $\rho(\kappa) \gtrsim 1$
is a direct consequence of the uniform continuity of $E^{\prime\prime}(u)$, see bound~\eqref{continuity_E} in Lemma~\ref{lemma:stability-garding} below.

As we will see, the operator $E^{\prime\prime}(u)$ plays a key role in the error analysis
which we will revisit in detail in Section \ref{section-secE-shift-prop}.

\begin{remark}\label{remark-gauge-choice}
We stress that for any $v \in H^1(\Omega)$, the condition $v \in \orthiu$
simply amounts to choosing a particular convention for the phase of $v$.
Indeed, if $v \in H^1(\Omega)$ is arbitrary, then for $\alpha \in \C$
\begin{equation*}
(\alpha v,\ci u)_{L^2(\Omega)} = \Re \int_\Omega \, \alpha\, v\,\overline{\ci u} = - \Re \left( \ci\alpha \int_\Omega v \overline{u} \right)
\end{equation*}
and we can always choose $\alpha \in \C$ with $|\alpha| = 1$ so that $\alpha v \in \orthiu$. To be precise, we can select $\alpha= \tfrac{\int_{\Omega} \overline{v} u}{\left|\int_{\Omega} \overline{v} u\right|}$ to obtain $(\alpha v,\ci u)_{L^2(\Omega)}=0$. 
Hence, requiring that $u_h \in \orthiu$ for a finite element solution $u_h$
is simply a convenient choice of phase convention to easily provide error estimates.
In practice, this condition does not need to be fulfilled, and the solution obtained
is accurate, up to a possible phase shift.
\end{remark}

\subsection{$\kappa$-dependent stability estimates}

Next, we shall establish $\kappa$-dependent stability estimates for any solution to the GLE $E^{\prime}(u)=0$.
The following $\kappa$-dependent $H^1$-bounds are established for global minimizers
in \cite{DuGuPe,DDH24,DoeHe24}. Here, we show that such bounds are actually valid for any
critical point. The proof only needs slight modifications and is reported here for completeness.

\begin{theorem} \label{estimates_u}
Assume that $u \in H^1(\Omega)$ is a critical point, i.e. $E^\prime(u) = 0$. Then, we have
	 \begin{align}
	 \label{stability-estimates-u}
		\vert u(x) \vert \,\leq\, 1 \quad \mbox{ a.e. in }  \Omega
		\qquad\hspace{10pt}
		\mbox{and}\hspace{10pt}
		\qquad  \tfrac{1}{\kappa} \| \nabla u \|_{L^2(\Omega)} \,\lesssim\, \| u\|_{L^2(\Omega)} \,\lesssim\, 1.
	\end{align}
\end{theorem}

\begin{proof}
The bound $0 \le |u| \le 1$ for any critical point with $E^{\prime}(u)=0$ can be proved analogously to the corresponding result in \cite[Proposition 3.11]{DuGuPe}. For the $H^1$-bounds we slightly deviate from \cite{DDH24} where the proof exploits $E(u) \le E(0)$, which is only guaranteed for minimizers. Instead, we directly use the Ginzburg--Landau equation. From $E^{\prime}(u)=0$ we obtain with \eqref{derivaitves-1} that
\begin{align*}
 \| \tfrac{\i}{\kappa} \nabla u + \bfA u \|_{L^2(\Omega)}^2 \,\,=\,\, (\, (1-\vert u \vert^2 ) u , u )_{L^2(\Omega)} \,\, \le \,\, \| u \|_{L^2(\Omega)}^2.
\end{align*}
Hence, we have the remaining two estimates as 
\begin{align*}
 \tfrac{1}{\kappa} \| \nabla u \|_{L^2(\Omega)} \,\,&\le\,\, 
  \| u \|_{L^2(\Omega)} + \| \bfA u \|_{L^2(\Omega)} 
 \,\, \le\,\,(1+ \|  \bfA  \|_{L^{\infty}(\Omega)} ) \, \| u \|_{L^2(\Omega)} 
 \,\,\le\,\, (1+ \|  \bfA  \|_{L^{\infty}(\Omega)} ) \sqrt{|\Omega|} .
\end{align*}
\end{proof}

By demanding more smoothness on $\bfA$ and $\Omega$, we can establish additional
bounds on critical points. Similar estimates were obtained previously in \cite{DDH24},
but they are extended here to higher order Sobolev spaces.

\begin{theorem}
\label{theorem_high_reg_u}
Assume that $\bfA \in W^{p-1,\infty}(\Omega)$ and that $\partial \Omega$
is of class $C^{p+1}$. Then, for all $u \in H^1(\Omega)$ such that $E^\prime(u) = 0$,
we have $u \in W^{p+1,4}(\Omega) \cap H^1(\Omega)$ with
\begin{equation}
\label{Wpq-estimates}
\|u\|_{W^{p+1,4}_\kappa(\Omega)} \lesssim 1
		\qquad\hspace{10pt}
		\mbox{and}\hspace{10pt}\qquad
\|u\|_{W^{p-1,\infty}_\kappa(\Omega)} \,\, \lesssim \,\, 
\kappa^{\varepsilon/(2p)},
\end{equation}
where we recall from Section \ref{subsection:hidden-constant}, that $\varepsilon$ is an arbitrarily small (but fixed) parameter. In addition, the estimate
\begin{align*}
\|u\|_{H^{p+1}_\kappa(\Omega)} \,\,\lesssim \,\, \kappa^{\varepsilon/2} \, \| u \|_{L^2(\Omega)}
\end{align*}
holds true.
\end{theorem}

\begin{proof}
We start by noting that $u$ solves the PDE
\begin{equation*}
-\Delta u
=
- 2\i\kappa \bfA \cdot \nabla u  - \kappa^2|\bfA|^2 u - \kappa^2 |u|^2u+ \kappa^2 u =: F
\end{equation*}
in $\Omega$, together with the Neumann boundary condition $\nabla u \cdot \mathbf{n} = 0$
on $\partial \Omega$. The remainder of the proof will therefore largely hinge on
the regularity shift
\begin{equation}
\label{prel_reg_shift}
\|u\|_{W^{\ell+1,q}(\Omega)} \lesssim_q \|F\|_{W^{\ell-1,q}(\Omega)} + \| u \|_{L^q(\Omega)}
\end{equation}
which holds for all $q \in [1,+\infty)$ and $1 \leq \ell \leq p$ provided
that the right-hand side is finite. Such regularity shift indeed does hold
true due to our assumption on $\partial \Omega$, see \cite[Theorem 15.2]{agmon_douglis_nirenberg_1959a} and \cite[Theorem on p. 143]{CantorMatovsky85}.
An easy consequence is that we also have
\begin{equation}
\label{tmp_reg_shift}
\|u\|_{W^{\ell+1,q}_\kappa(\Omega)}
\lesssim_{\ell,q} \kappa^{-2} \|F\|_{W^{\ell-1,q}_\kappa(\Omega)}
+
\| u \|_{W^{1,q}_{\kappa}(\Omega)}
\end{equation}
uniformly in $\kappa$. We first show that for any $q \in [1,+\infty)$, we have
\begin{equation}
\label{tmp_induction_regu}
\|u\|_{W^{\ell,q}_\kappa(\Omega)} \lesssim_{\ell,q} 1
\end{equation}
for all $\ell \leq p+1$ by induction on $\ell$.

Let us observe that \eqref{tmp_induction_regu} holds true for $\ell = 0$
due to \eqref{stability-estimates-u}.

For the case $\ell = 1$,
we proceed by induction on $q$. We first see that
\begin{equation*}
\|F\|_{L^2(\Omega)}
\lesssim
\kappa \|\bfA\|_{L^\infty(\Omega)}\|\nabla u\|_{L^2(\Omega)}
+
\kappa^2\|\bfA\|_{L^\infty(\Omega)}^2 \|u\|_{L^2(\Omega)}
+
\kappa^2 (1+\|u\|_{L^\infty(\Omega)}^2)\|u\|_{L^2(\Omega)}
\lesssim
\kappa^2
\end{equation*}
due to \eqref{stability-estimates-u}.
This shows that the result holds true for $q = 2$, and in fact
for any $1 \leq q \leq 2$. Furthermore, thanks to the regularity
shift \eqref{prel_reg_shift}, we have $u \in W^{2,2}(\Omega)$ with
$\| u \|_{W^{2,2}(\Omega)} \lesssim \kappa^2$. Let us now assume that
the result has been established for some $q \geq 2$. Proceeding as above, we have
\begin{equation*}
\|F\|_{L^{2q}(\Omega)}\,
\lesssim\,
\kappa  \|\nabla u\|_{L^{2q}(\Omega)} + \kappa^2
\end{equation*}
using the $L^\infty(\Omega)$ bound on $u$.
We now apply the Gagliardo-Nirenberg inequality in \eqref{eq_gagliardo_nirenberg_L2q},
to show that
\begin{equation*}
\|F\|_{L^{2q}(\Omega)}
\lesssim_q
\kappa \|u\|_{L^\infty(\Omega)}^{1/2}\|u\|_{W^{2,q}(\Omega)}^{1/2} + \kappa^2
\lesssim_q
\kappa^2
\end{equation*}
using again the $L^\infty(\Omega)$ bound on $u$ and the induction hypothesis $\|u\|_{W^{2,q}(\Omega)} \lesssim_q \kappa^2$ 
in combination with the regularity shift \eqref{prel_reg_shift}. This now shows
the result for all exponents less than $2q$, which concludes the proof by induction on $q$
for the case $\ell=1$.

We now assume that \eqref{tmp_induction_regu} holds true up to some $\ell \leq p$ and for any $q \in [1,+\infty)$. We show that it still holds true for $\ell+1$ and again any $q \in [1,+\infty)$. 
In view of \eqref{tmp_reg_shift}, we simply need to show that
\begin{equation*}
\|F\|_{W^{\ell-1,q}_{\kappa} (\Omega)} \lesssim_{\ell,q} \kappa^2.
\end{equation*}
Using the product rules in \eqref{eq_product_rules} and the induction assumption, we have
\begin{align*}
&\|F\|_{W_\kappa^{\ell-1,q}(\Omega)}
\\
&\lesssim_{\ell,q}
\kappa \|\bfA \cdot \nabla u\|_{W^{\ell-1,q}_\kappa(\Omega)}
+
\kappa^2 \||\bfA|^2 u\|_{W^{\ell-1,q}_\kappa(\Omega)}
+
\kappa^2 \||u|^2 u\|_{W^{\ell-1,q}_\kappa(\Omega)} + \kappa^2 \|u\|_{W^{\ell-1,q}_{\kappa}(\Omega)}
\\
&\lesssim_{\ell,q}
\kappa \|\bfA\|_{W^{\ell-1,\infty}_\kappa(\Omega)}\|\nabla u\|_{W^{\ell-1,q}_\kappa(\Omega)}
+
 \kappa^2  \left(  (\|\bfA\|_{W^{\ell-1,\infty}_\kappa(\Omega)}^2 + 1) \|u\|_{W^{\ell-1,q}_\kappa(\Omega)}
+
 \|u\|_{W^{\ell-1,3q}_\kappa(\Omega)}^3 \right)
\\
&\lesssim_{\ell,q}
\kappa^2 \|\bfA\|_{W^{\ell-1,\infty}_\kappa(\Omega)}\|u\|_{W^{\ell,q}_\kappa(\Omega)}
+
 \kappa^2  \left(   (\|\bfA\|_{W^{\ell-1,\infty}_\kappa(\Omega)}^2 + 1) \|u\|_{W^{\ell-1,q}_\kappa(\Omega)}
+
\|u\|_{W^{\ell-1,3q}_\kappa(\Omega)}^3 \right)
\\
&
\lesssim_{\ell,q}
\kappa^2,
\end{align*}
from which the desired bound on $F$ follows. The regularity shift in \eqref{tmp_reg_shift}
finishes the argument.

With \eqref{tmp_induction_regu} established,
the first estimate in \eqref{Wpq-estimates} readily follows by taking $q = 4$.
For the second estimate in \eqref{Wpq-estimates}, we simply employ the Sobolev embedding
in \eqref{eq_injection_Linfty}, and write that
\begin{equation*}
\|\phi\|_{L^\infty(\Omega)}
\lesssim
\kappa^{\varepsilon/2p}
\|\phi\|_{W^{1,q_\varepsilon}_{\kappa}(\Omega)}
\end{equation*}
for all $\phi \in W^{1,q_\varepsilon}(\Omega)$ with $q_\varepsilon := 2p(d+1)/\varepsilon$,
$s_\varepsilon =\varepsilon/(2p)$ so that $s_\varepsilon q_\varepsilon > d$.
A direct consequence is that
\begin{equation*}
\|u\|_{W^{p-1,\infty}_\kappa(\Omega)}
\lesssim
\kappa^{\varepsilon/2p} \|u\|_{W^{p,q_\varepsilon}_\kappa(\Omega)}
\lesssim
\kappa^{\varepsilon/2p}
\end{equation*}
due to \eqref{tmp_induction_regu}.

Similarly to what we did above, we will show that
\begin{equation*}
\|u\|_{H^{\ell+1}_\kappa(\Omega)} \,\,\, \lesssim_\ell \,\,\, \kappa^{\varepsilon \, \ell/(2p)} \|u\|_{L^2(\Omega)}
\end{equation*}
for $0 \leq \ell \leq p$ by induction on $\ell$. The result then follows with $\ell=p$. 
For $\ell = 0$, the result was established in Theorem \ref{estimates_u}.
Assuming the result for some $1 \leq \ell \leq p-1$, we are going to show
that it is still valid for $\ell+1$. 
Following the proof of the first two estimates above,
it suffices to show that
\begin{equation*}
\|F\|_{H^{\ell-1}_\kappa(\Omega)}
\lesssim
\kappa^2\kappa^{ \varepsilon \, \ell/(2p)}\|u\|_{L^2(\Omega)}.
\end{equation*}
We easily do so using the product rules in \eqref{eq_product_rules} and the induction
hypothesis as follows:
\begin{eqnarray*}
\lefteqn{\|F\|_{H^{\ell-1}_\kappa(\Omega)}
\,\,\,\lesssim\,\,\,
\kappa \|\bfA \cdot \nabla u\|_{H^{\ell-1}_\kappa(\Omega)}
+
\kappa^2 \||\bfA|^2 u\|_{H^{\ell-1}_\kappa(\Omega)} + \kappa^2 \|u\|_{H^{\ell-1}_\kappa(\Omega)}
+
\kappa^2\||u|^2 u\|_{H^{\ell-1}_\kappa(\Omega)} }
\\
&\lesssim&
\kappa \|\bfA\|_{W^{\ell-1,\infty}_\kappa(\Omega)}\|\nabla u\|_{H^{\ell-1}_\kappa(\Omega)}
+
\kappa^2 \|\bfA\|_{W^{\ell-1,\infty}_\kappa(\Omega)}^2\|u\|_{H^{\ell-1}_\kappa(\Omega)}
\\
&\enspace&\quad +
\kappa^2 \|u\|_{H^{\ell-1}_\kappa(\Omega)}
+
\kappa^2\|u\|_{W^{\ell-1,\infty}_\kappa(\Omega)}^2\|u\|_{H^{\ell-1}_\kappa(\Omega)}
\\
&\overset{\eqref{Wpq-estimates}}{\lesssim}&
\kappa^2 \|\bfA\|_{W^{\ell-1,\infty}_\kappa(\Omega)}\|u\|_{H^{\ell}_\kappa(\Omega)}
+
\kappa^2 (\|\bfA\|_{W^{\ell-1,\infty}_\kappa(\Omega)}^2 + 1) \|u\|_{H^{\ell-1}_\kappa(\Omega)}
+
\kappa^2 \kappa^{\varepsilon/p} \|u\|_{H^{\ell-1}_\kappa(\Omega)}
\\
&\lesssim&
\kappa^2 \|u\|_{H^{\ell}_{\kappa}(\Omega)} + \kappa^2 \kappa^{\varepsilon/p} \|u\|_{H^{\ell-1}_\kappa(\Omega)}
\,\,\,\lesssim\,\,\,
\kappa^2 ( \kappa^{(\ell-1) \varepsilon/(2p)}
+
\kappa^{\varepsilon/p} \kappa^{(\ell-2) \varepsilon/(2p)} ) \|u\|_{L^2(\Omega)} \\
&\lesssim&
\kappa^2\kappa^{\ell \varepsilon/(2p)}\|u\|_{L^2(\Omega)}.
\end{eqnarray*}
\end{proof}

\section{Discrete minimizers in finite element spaces and $H^1_{\kappa}$-optimality}
\label{section:discrete-minimizers}
In this section, we describe the finite element discretization of the Ginzburg-Landau minimization problem and present our main result regarding optimality of the corresponding discrete minimizers in the $H^1_{\kappa}$-norm.

\subsection{Finite element setting}
In the following we consider a shape-regular, conforming partition $\mathcal{T}_h$ of $\Omega$
into (curved) simplices $T \in \mathcal{T}_h$ of (maximum) mesh size
$h:=\max\{ \diam(T)\,|\,T\in \mathcal{T}_{h}\}$. Each mesh element $K \in \mathcal{T}_h$
is obtained from a fixed reference simplex $\widehat{K}$ through a bi-Lipschitz mapping
$\mathcal{F}_K: \widehat{K} \to K$.

The corresponding $\mathbb{P}_p$ Lagrange finite element space is given by
\begin{align}
\label{def-Vh}
V_h
:=
\{
v_h \in H^1(\Omega) \, | \,\, v_h \vert_{K} \circ \mathcal{F}_K
\in \mathbb{P}_p(\widehat{K}) \mbox{ for all } K \in \mathcal{T}_h
\},
\end{align}
where $\mathbb{P}_p(\widehat{K})$ is the set of polynomials of total degree
less than or equal to $p$. In the following, we assume
that there exists a constant $M > 0$ and a projection operator
$I_h: H^1(\Omega) \to V_{h}$ such that
\begin{subequations}
\label{eq_assumption_interpolation}
\begin{equation}
|I_h(v)|_{H^1(\Omega)} \leq M|v|_{H^1(\Omega)}
\end{equation}
for all $v \in H^1(\Omega)$, and
\begin{equation}
h^{-1} \|v-I_h(v)\|_{L^q(\Omega)} + |v-I_h(v)|_{W^{1,q}(\Omega)}
\leq
M h^\ell \| v \|_{W^{\ell+1,q}(\Omega)}
\end{equation}
for $0 \leq \ell \leq p$ and $2 \leq q \leq 4$.
\end{subequations}

The existence of such an interpolation operator is standard for meshes consisting of straight
elements (see e.g. \cite{ern_guermond_2017a}), but has not been addressed for curved elements to the best of our knowledge.
In Appendix~\ref{appendix:section:quasi-interpolation} below, we show that such a condition can
indeed be fulfilled with a constant $M$ independent of $h$ under natural regularity assumptions
on the element mappings $\mathcal{F}_K$. In turn, such mappings may be constructed under the
natural smoothness assumption that $\partial \Omega$ is of class
$C^{p+1}$, see~\cite{bernardi_1989a,chaumontfrelet_spence_2024a,lenoir_1986a}.

For simplicity, we will further assume that $h\kappa \leq 1$.
This assumption is only made to simplify expressions like
$h\kappa + 1 \lesssim 1$. In addition, it is in fact not restrictive,
since we will impose stronger conditions on the mesh size later on.

\subsection{Discrete minimizer}

In what follows, we either consider a discrete global minimizer $\uhp \in V_{h}$ with the property 
\begin{align}
\label{discrete-minimizer-energy-def} 
E(\uhp) \,\,= \min_{v_h \in V_{h}} E(v_h)
\end{align}
or a discrete local minimizer with $\langle E^{\prime}(\uhp), v_h \rangle = 0$ for all $v_h \in V_{h}$.
Global discrete minimizers always exist in finite-dimensional spaces $V_{h}$ (cf. \cite[Lemma 3.2]{BDH25}) which directly follows from the non-negativity of $E(v)$ and the inequality $\|v\|_{H^1_{\kappa}(\Omega)}^2 \lesssim 1 + E(v)$ for all $v \in H^1(\Omega)$. 

However, just like continuous minimizers in $H^1(\Omega)$, $\uhp$ cannot be unique and any complex phase shift of $\uhp$, i.e., $\exp(\ci \omega) \uhp$ for $\omega \in [-\pi,\pi)$, is a (local) minimizer on the same energy level. This has to be considered when comparing some $\uhp$ with an exact local minimizer $u$.

Assuming that $\bfA \in W^{p-1,\infty}(\Omega)$ and that $\partial \Omega$ is $C^{p+1}$ and
using the $\kappa$-dependent bounds in Theorems \ref{estimates_u} and
\ref{theorem_high_reg_u} together with \eqref{eq_assumption_interpolation},
we can conclude that the $H^1_{\kappa}$-best-approximation error for any critical point $u$ (and therefore also for any local or global minimizer) satisfies 
\begin{align*}
\min_{v_h \in V_{h} }  \| u - v_h \|_{H^1_{\kappa}(\Omega)}
\,\,\, \lesssim \,\,\, 
h^{\ell+1}   \|  u  \|_{H^{\ell+1}(\Omega)}
+
\kappa^{-1} h^{\ell}  \| u \|_{H^{\ell+1}(\Omega)}
 \,\,\, \lesssim \,\,\, 
(1+h\kappa) \, (h \kappa)^{\ell} 
\end{align*}
provided $u \in H^{\ell+1}(\Omega)$ and $\ell \le p$. This shows that the mesh resolution
condition $h \kappa \leq 1$ we imposed above is natural
for obtaining meaningful numerical approximations. However, as we will see later, there is
a numerical pollution effect and the condition $h \kappa \lesssim 1$ is no longer sufficient
for $\uhp$ to be a reasonable approximation of $u$ in the $L^2$- or $H^1_{\kappa}$-norm. 
On the contrary, the discrete energy level of global minimizers is not affected by the pollution
effect. This is seen by the following result.
\begin{proposition}[Error in the minimal energy]
\label{prop-energy-error}
Assume that $\bfA \in W^{p-1,\infty}(\Omega)$ and $\partial \Omega$ is $C^{p+1}$.
Then it holds
\begin{align*}
0 \,\, \le \,\, \min_{u_h \in V_h} E(\uhp) - \min_{u \in H^1(\Omega)} E(u ) \,\, \lesssim \,\, (h \kappa)^{2p}.
\end{align*}
\end{proposition}

\begin{proof}
Let $u$ and $u_h$ respectively denote any global minimizers in $H^1(\Omega)$ and $V_h$.
Since $u_h \in V_{h}$ is a discrete global minimizer, we have $E(u) \le E(\uhp) \le E(\hspace{1pt}I_h(u)\hspace{1pt})$, and hence 
\begin{align*}
0 \le E(\uhp) - E(u) \le  E(\hspace{1pt}I_h(u)\hspace{1pt}) - E(u).
\end{align*}
Analogously to the proof of \cite[Lemma 5.9]{DoeHe24} it holds
\begin{eqnarray*}
 E(\hspace{1pt}I_h(u)\hspace{1pt}) - E(u) 
&\lesssim& \| u - I_h(u) \|_{H^1_{\kappa}(\Omega)}^2 +  \| u - I_h(u) \|_{L^2(\Omega)}^2
+ \| u - I_h(u) \|_{L^4(\Omega)}^4.
\end{eqnarray*}
The interpolation error estimates for $I_h$ in \eqref{eq_assumption_interpolation}
together with the stability bounds for $u$ from Theorem \ref{theorem_high_reg_u}
finish the proof.
\end{proof}

\subsection{Main results and $H^1_{\kappa}$-optimality}
In this section, we present the main result of this paper, which is formulated in
Theorem \ref{theorem:main-result} below. Essentially, the theorem states that if
the polynomial degree $p$ and the mesh size $h$ of the finite element space $V_{h}$
are selected relative to $\kappa$ with the resolution condition
$\kappa^{d/2}\rho(\kappa)(h\kappa)^p \,\, \ll \,\, 1$ 
for a local minimizer $u \in H^1(\Omega)$ satisfying \ref{A4}, then
there is a discrete (local) minimizer $u_h \in V_{h}$, such that the error
$\| u - u_h \|_{H^1_{\kappa}(\Omega)}$ converges  to zero with optimal order
$\mathcal{O}((h\kappa)^p)$. Similarly, a corresponding $L^2$-error estimate holds.
\begin{theorem}
\label{theorem:main-result}
Assume that $\bfA \in W^{p-1,\infty}(\Omega)$ and $\partial \Omega$ is
$C^{p+1}$ and consider a (local) minimizer $u\in H^1(\Omega)$ of $E$
satisfying \ref{A4}. Then, there exists a constant $K_\star \gtrsim 1$ such that if
\begin{align}
\label{resolution-condition-h-kappa}
\kappa^{d/2}\rho(\kappa)(h\kappa)^p \,\, \leq \,\, K_\star,
\end{align}
then there exists a discrete (local) minimizer $\uhp \in V_{h}$ of $E$ with 
\begin{eqnarray*}
\langle E^{\prime}(\uhp) , v_h \rangle &=& 0 \hspace{93pt} \mbox{ for all } v_h \in V_{h}
\qquad
\mbox{and} \\
\langle E^{\prime\prime}(\uhp) v_h , v_h \rangle &\gtrsim& \rho(\kappa)^{-1} \| v_h \|_{H^1_{\kappa}(\Omega)}^2 \hspace{20pt} \mbox{ for all } v_h \in V_{h}\cap \orthiu
\end{eqnarray*}
such that
\begin{align}
\label{final-H1-estimate}
\| u - \uhp \|_{H^1_{\kappa}(\Omega)} \,\,\, \lesssim \,\,\, (h \kappa)^{p} 
\end{align}
and
\begin{align}
\label{final-L2-estimate}
\| u - \uhp \|_{L^2(\Omega)} \,\,\,\, \lesssim \,\,  \kappa^{\varepsilon} \, (h \kappa)^{p+1} \,\,
 + \,\, \kappa^{d/2} \rho(\kappa)\,(h \kappa)^{2p}.
\end{align}
\end{theorem}
The proof of the theorem takes place in several steps. The essential preparations are introduced in Section \ref{section-secE-shift-prop} and the main proofs follow in Section \ref{section:proofs:L2-H1-estimates}.

Before we start with the proofs, let us briefly discuss the main findings of Theorem \ref{theorem:main-result}. First of all, the theorem predicts a numerical pollution effect (later observed in our numerical experiments), where the mesh size $h$ needs to fulfill the stronger resolution condition $\kappa^{d/2}\rho(\kappa)(h\kappa)^p \lesssim 1$ before the numerical approximations $\uhp$ start to converge with the expected rate of $O((h\kappa)^p)$ in the $H^1_{\kappa}$-norm to the correct minimizer $u$. Since $\rho(\kappa)$ is a priori unknown, the precise necessary resolution for $h$ is also a priori unknown. However, the theorem also shows that whatever the growth of $\rho(\kappa)$ w.r.t. $\kappa$, the resolution condition improves with the polynomial degree $p$ of the finite element space $V_{h}$. For example, for $d=2$ and $p=1$, the resolution condition becomes $h \rho(\kappa) \kappa^2 \lesssim 1$, whereas it reduces significantly to $ h \rho(\kappa)^{1/2}  \kappa^{3/2} \lesssim  1$ for $p=2$. Hence, higher order finite element spaces can effectively reduce the pollution effect for the Ginzburg--Landau equation. It remains open if the factor $\kappa^{d/2}$ in \eqref{resolution-condition-h-kappa} is sharp or if it can be removed with different proof techniques. 

Note that in the previous work \cite{DoeHe24} error estimates for $p=1$ were derived under the assumption that $h$ is \textit{sufficiently small}, with at least 
$\rho(\kappa)(h\kappa)\lesssim 1$. At first glance, this looks like a sharper condition. However, the precise meaning of 
\textit{sufficiently small} could not be quantified because the proof relied on a 
compactness argument (passing to the limit as $h\to 0$ to conclude convergence 
up to phase shifts). In the present work we make this smallness condition 
explicit by showing that $\kappa^{d/2}\,\rho(\kappa)\,(h\kappa) \lesssim 1$ is sufficient. Thus, our results extend and strengthen those of \cite{DoeHe24} 
even in the case $p=1$.

Finally, Theorem \ref{theorem:main-result} also establishes an $L^2$-error estimate of optimal order in $h$. However, the estimate still appears suboptimal for $p=1$ with respect to its $\kappa$-dependency. In fact, for $p=1$, \eqref{final-L2-estimate} essentially reduces to
\begin{align*}
 \| u - \uhp \|_{L^2(\Omega)} \,\,\,\, \lesssim  \,\,\,\, (\kappa^{\varepsilon}  + \kappa^{d/2} \rho(\kappa))\, (h \kappa)^{2}  \,\,\,\, \lesssim  \,\,\,\,  \kappa^{d/2} \rho(\kappa) \, (h \kappa)^{2}, 
\end{align*}
where $\rho(\kappa)$ appears as a multiplicative factor. For larger $p$, the term $\,\, \kappa^{d/2} \rho(\kappa)\,(h \kappa)^{2p}$ is of higher order and the scaling with $\rho(\kappa)$ vanishes asymptotically in the same way as for the $H^1_{\kappa}$-estimate. To be precise, for $\kappa^{d/2} \rho(\kappa)\,(h \kappa)^{p-1}\lesssim 1$, estimate \eqref{final-L2-estimate} becomes asymptotically optimal with
\begin{eqnarray*}
 \| u - \uhp \|_{L^2(\Omega)} &\lesssim& \, \kappa^{\varepsilon} \, (h \kappa)^{p+1}.
\end{eqnarray*}
The remaining part of the paper is devoted to the proof of Theorem \ref{theorem:main-result}. Furthermore, we will present corresponding numerical experiments to illustrate the pollution effect and how it is reduced in higher order FE spaces.

\section{The operator $E^{\prime\prime}(u)$ and the shift property}
\label{section-secE-shift-prop}
As an essential preparation for the error analysis, this section is concerned with various auxiliary results regarding the properties of the operator $E^{\prime\prime}(u)$ and its inverse. For that, recall from Lemma \ref{coercivity_secE_u} that $E^{\prime\prime}(u)$ is coercive on $\orthiu$, however, with a coercivity constant $\rho(\kappa)^{-1}$ that scales unfavourably with $\kappa$. For that reason, we will mainly exploit a G{\aa}rding inequality in our analysis, where the arising constants do no longer involve $\kappa$. In the following, we collect all the necessary auxiliary results. We start with the uniform continuity of $E^{\prime\prime}(u)$ in the $H^1_{\kappa}$-norm and the aforementioned G{\aa}rding inequality.
\begin{lemma}\label{lemma:stability-garding}
Consider a critical point $u \in H^1(\Omega)$ with $E^{\prime}(u)=0$. Then
$E^{\prime\prime}(u)$ is continuous on $H^1(\Omega)$ w.r.t.\ $\| \cdot \|_{H^1_\kappa(\Omega)}$
with a continuity constant independent of $\kappa$, i.e.,
\begin{align}
\label{continuity_E}
\langle E^{\prime\prime}(u) v, w \rangle \,\, \lesssim \,\, \| v \|_{H^1_{\kappa}(\Omega)}  \| w \|_{H^1_{\kappa}(\Omega)}
\qquad
\mbox{for all } v,w \in H^1(\Omega)
\end{align}
and, for all $v\in H^1(\Omega)$, we have the G{\aa}rding inequality
\begin{align*}
\langle E^{\prime\prime} (w)v, v \rangle \,\,\ge \,\,
\tfrac{1}{2} \| v \|_{H^1_{\kappa}(\Omega)}^2
-
\left (\tfrac{3}{2} + \|\bfA\|_{L^{\infty}(\Omega)}^2\right ) \| v\|_{L^2(\Omega)}^2.
\end{align*}
\end{lemma}

\begin{proof}
The proof of the continuity is straightforward using $|u|\le 1$ a.e. and the G{\aa}rding inequality follows with $| \tfrac{\ci}{\kappa} \nabla v + \bfA v |^2 \ge \tfrac{1}{2} \tfrac{1}{\kappa^{2}} |\nabla v |^2 -  | \bfA v|^2$ as
\begin{eqnarray*}
 \langle E^{\prime\prime} (w)v, v \rangle &\overset{\eqref{derivaitves-2}}{=}& 
 \int_{\Omega} | \tfrac{\ci}{\kappa} \nabla v + \bfA v |^2  
 \,+\, (2 \vert w \vert^2 \hspace{-2pt} -\hspace{-2pt}1) |v|^2  + \Re( w^2 \overline{v}^2) \\
 &\ge&  \int_{\Omega}  \tfrac{1}{2} \tfrac{1}{\kappa^{2}} |\nabla v |^2 \,-\,  | \bfA v|^2 \,+\, ( \vert w \vert^2 \hspace{-2pt} -\hspace{-2pt}1) |v|^2 \\
  &\ge& \int_{\Omega}  \tfrac{1}{2} \tfrac{1}{\kappa^{2}} |\nabla v |^2 +  \tfrac{1}{2} | v |^2  
  \,+\, (  \vert w \vert^2 - | \bfA|^2  - \tfrac{3}{2}) |v|^2.
\end{eqnarray*}
\end{proof}

Assuming that $u$ is a local minimizer satisfying~\ref{A4},
the Lax-Milgram lemma ensures that for each $\psi \in (\orthiu)^\star$,
there exists a unique $\Einv \psi \in \orthiu$
such that
\begin{equation*}
\langle E^{\prime\prime}(u) (\Einv \psi),v \rangle
=
\langle \psi, v \rangle
\end{equation*}
for all $v \in \orthiu$.
Let us recall a well-posedness and regularity result formulated in
\cite[Corollary 2.9]{DoeHe24} which essentially follows from the coercivity
in Lemma \ref{coercivity_secE_u}. 
\begin{proposition}
Consider a local minimizer $u \in H^1(\Omega)$ of $E$ satisfying \ref{A4}.
Then, for each $f \in L^2(\Omega)$, it holds
\begin{align}
\label{secEest-H1}
\| \Einv f \|_{H^1_{\kappa}(\Omega)} \,\, \lesssim \,\,  \Ccoe \, \| f\|_{L^2(\Omega)}.
\end{align}
Furthermore, if $\partial \Omega$ is $C^{2}$ (or if $\Omega$ is convex), then we have
$\Einv f  \in H^2(\Omega)$ and it holds
\begin{align}
\label{secEest-H2}
\| \Einv f \|_{H^2_{\kappa}(\Omega)}  \,\, \lesssim \,\, \Ccoe \, \| f\|_{L^2(\Omega)}.
\end{align}
\end{proposition}
Following~\cite{CFN20}, the main goal of this section is to refine the above result by deriving an expansion of $\Einv f$ in terms of $\kappa$ of the form 
$$
\Einv f  \,\,=\,\, \sum_{j=0}^{\ell-1} \kappa^j z_j  \,\, + \,\, r_{\ell},
$$
where the regularity of $z_j$ increases with $j$, that is $z_j \in H^{j+2}(\Omega)$ with
$| z_j |_{H^{j+2}(\Omega)} \lesssim \kappa^{j+2} \|f\|_{L^2(\Omega)}$ independent
of $\Ccoe$, and where $r_{\ell} \in H^{\ell+2}(\Omega)$ is a rest term fulfilling
$| r_{\ell} |_{H^{\ell+2}(\Omega)} \lesssim \Ccoe \kappa^{\ell+2} \|f\|_{L^2(\Omega)}$. In a first step towards such an expansion, we split $E^{\prime\prime}(u)$ according to its powers in $\kappa$, where we obtain the following result.
\begin{lemma}\label{lemma-form-of-secE}
For $u \in H^1(\Omega)$, we have
\begin{align*}
\langle E^{\prime\prime}(u) v , w \rangle = \tfrac{1}{\kappa^2} ( \nabla v ,\nabla w )_{L^2(\Omega)}
+ \tfrac{1}{\kappa} (  \Lone(\nabla v) ,  w )_{L^2(\Omega)}
+ ( \hspace{1pt}\Lzero(v) , w \hspace{1pt} )_{L^2(\Omega)},
\end{align*}
for all $v,w \in H^1(\Omega)$, where 
\begin{align}
\label{def-Lu}
\Lone(\nabla v):= 2 \hspace{1pt}\ci \, \nabla v \cdot \bfA
\qquad
\mbox{and}
\qquad
 \Lzero(v)\,\,:=\,\, |\bfA|^2v + (|u|^2-1)v + 2 \hspace{1pt}\Re\hspace{-1pt}(\overline{u} v) u.
\end{align}
\end{lemma}

\begin{proof}
With \eqref{derivaitves-2} and $u^2 \overline{v} +\vert u \vert^2 v = 2\, \Re(\overline{u} v) u$ it holds
\begin{align*}
 \langle E^{\prime\prime} (u)v, w \rangle \,\,=\,\, ( \tfrac{\i}{\kappa} \nabla v+ \bfA v , \tfrac{\i}{\kappa} \nabla w + \bfA w  )_{L^2(\Omega)} + ( \, (\vert u \vert^2 \hspace{-2pt} -\hspace{-2pt}1) v  +  2\, \Re(\overline{u} v) u ,  w \,)_{L^2(\Omega)}. 
\end{align*}
It remains to check the first term, for which we obtain 
\begin{eqnarray*}
\lefteqn{  ( \tfrac{\i}{\kappa} \nabla v+ \bfA v , \tfrac{\i}{\kappa} \nabla w + \bfA w  )_{L^2(\Omega)}  \,\,=\,\,
\Re \int_{\Omega} ( \tfrac{\i}{\kappa} \nabla v+ \bfA v ) \cdot \overline{( \tfrac{\i}{\kappa} \nabla w + \bfA w  )} }\\
%
 %
  &=&  \tfrac{1}{\kappa^2} \Re \int_{\Omega}  \nabla v  \cdot \overline{\nabla w }
 - \Re \int_{\Omega} \tfrac{\i}{\kappa} \bfA v \cdot \overline{ \nabla w}
 + \Re \int_{\Omega} \tfrac{\i}{\kappa} \nabla v \cdot \bfA \, \overline{w }
 + \Re \int_{\Omega} |\bfA|^2 v  \overline{ w } \\
%
 %
 %
  &=&  \tfrac{1}{\kappa^2} \Re \int_{\Omega}  \nabla v  \cdot \overline{\nabla w }
 + 2 \, \Re \int_{\Omega} \tfrac{\i}{\kappa} \nabla v \cdot \bfA \, \overline{w }
 + \Re \int_{\Omega} |\bfA|^2 v  \overline{ w } \\
  &=& \Re \int_{\Omega}  \left( - \tfrac{1}{\kappa^2} \Delta v  + \tfrac{2\i}{\kappa} \nabla v \cdot \bfA + |\bfA|^2 v  \right) \overline{ w },
\end{eqnarray*}
where we used $\bfA \cdot \mathbf{n} \vert_{\partial \Omega} =0 $ and $\div \bfA = 0$ in the penultimate step.
\end{proof}

\begin{proposition}
\label{proposition-shift}
Consider a critical point $u \in H^1(\Omega)$ such that $E^{\prime}(u)=0$.
Then, for $1 \leq \ell \le p$, \,\, $\widetilde{f} \in H^{\ell-1}(\Omega)$\, and \,\,$\widetilde{g} \in H^{\ell}(\Omega)$\, we let $\widetilde{u} \in \orthiu$ denote the solution to
\begin{align}
\label{eqn-tilde-u-shift}
(\nabla \widetilde{u} , \nabla v )_{L^2(\Omega)} 
+
\kappa^2 ( \widetilde{u} , v )_{L^2(\Omega)} 
&=
- \kappa^2 (\Lzero (\widetilde{f}) -\widetilde{f} , v)_{L^2(\Omega)} - \kappa \,( \Lone(\nabla  \widetilde{g})  , v)_{L^2(\Omega)} 
\end{align}
for all $v\in \orthiu$ and where $\Lzero$ and $\Lone$ are given by \eqref{def-Lu}.
Then, $\widetilde{u} \in H^{\ell+1}(\Omega)$
with
\begin{equation}
\label{regularity-estimate-tilde-u}
\| \widetilde{u} \|_{H^{\ell+1}_\kappa(\Omega)} 
\lesssim
\kappa^{\varepsilon/p}
\left (
\|\widetilde{f}\|_{H^{\ell-1}_{\kappa}(\Omega)} + \|\widetilde{g}\|_{H^{\ell}_{\kappa}(\Omega)}
\right ).
\end{equation}
\end{proposition}

\begin{proof}
We first establish the following stability estimate
\begin{align}
\label{stability-estimate-tilde-u}
\|\widetilde{u}\|_{H^1_\kappa(\Omega)}
\lesssim
\| \widetilde{f} \|_{L^2(\Omega)} + \|\widetilde{g} \|_{H^1_{\kappa}(\Omega)}
\end{align}
by taking $v = \widetilde{u}$ in \eqref{eqn-tilde-u-shift}. Here, we used
$\|\Lzero(\widetilde{f})\|_{L^2(\Omega)} \lesssim \|\widetilde{f}\|_{L^2(\Omega)}$
due to the estimate $|u| \leq 1$.

Our goal is then to apply regularity shifts for the Neumann problem.
More precisely, we will use that if $\psi \in H^1(\Omega)$ satisfies
$\kappa^2 \psi-\Delta \psi \in H^{\ell-1}(\Omega)$
and $\nabla \psi \cdot \mathbf{n} = 0$ on $\partial \Omega$, then
\begin{equation}
\label{tmp_shift_RD}
\|\psi\|_{H^{\ell+1}_\kappa(\Omega)} \,\,\, \lesssim \,\,\, \kappa^{-2}\|\kappa^2\psi-\Delta \psi\|_{H^{\ell-1}_\kappa(\Omega)} 
\end{equation}
for all $1 \leq \ell \leq p$. The estimate is completely standard without the $\kappa^2 \psi$
term in the right-hand side (see e.g. \cite[Thrm.~IV.2.4]{Mikhailov}), and the present
version is easily obtained by inductive application of the standard shift result.

Let us now verify that $\tilde{u}$ solves such a Neumann problem. For shortness, we introduce
\begin{align*}
\widetilde{F}_{\kappa} := - \kappa^2 \Lzero (\widetilde{f}) + \kappa^2 \widetilde{f} - \kappa \, \Lone( \nabla  \widetilde{g})   \in L^2(\Omega),
\end{align*}
so that \eqref{eqn-tilde-u-shift} becomes
$\kappa^2(\widetilde{u},v) + (\nabla \widetilde{u} , \nabla v )_{L^2(\Omega)} = (\widetilde{F}_{\kappa} , v )_{L^2(\Omega)}$
for all $ v\in \orthiu$.
In order to transform this equation to a regular Neumann problem on $H^1(\Omega)$,
it is natural to decompose an arbitrary test function $v \in H^1(\Omega)$ uniquely
in $v = v^{\perp}+ \alpha_v (\ci u)$, where $ v^{\perp} \in \orthiu$ is the $L^2$-projection
of $v$ onto $\orthiu$ and $\alpha_v := (v,\ci u)_{L^2(\Omega)} \| u \|_{L^2(\Omega)}^{-2}$
(cf. \cite{DoeHe24}). This decomposition can be now used in
$( \nabla \widetilde{u} , \nabla v )_{L^2(\Omega)}$ to obtain 
\begin{align*}
(\nabla \widetilde{u} ,  \nabla v)_{L^2(\Omega)}
&= (\nabla \widetilde{u} ,  \nabla v^{\perp} )_{L^2(\Omega)} + \alpha_v (\nabla \widetilde{u} ,  \nabla (\ci u) )_{L^2(\Omega)} \\
&= ( - \kappa^2 \widetilde{u}  + \widetilde{F}_{\kappa} , v^{\perp} )_{L^2(\Omega)}  + \alpha_v (\nabla \widetilde{u} ,  \nabla (\ci u) )_{L^2(\Omega)} \\
&= (  \widetilde{F}_{\kappa}  - \kappa^2 \widetilde{u} , v )_{L^2(\Omega)}  -  \alpha_v ( \widetilde{F}_{\kappa} - \kappa^2 \widetilde{u} , \ci u)_{L^2(\Omega)} + \alpha_v (\nabla \widetilde{u} ,  \nabla (\ci u) )_{L^2(\Omega)} \\
&= \left(  \widetilde{F}_{\kappa}  - \kappa^2 \widetilde{u}  - \frac{( \widetilde{F}_{\kappa} , \ci u)_{L^2(\Omega)}}{\| u\|_{L^2(\Omega)}^2} \ci u + \frac{(\nabla \widetilde{u} ,  \nabla (\ci u) )_{L^2(\Omega)}}{\| u\|_{L^2(\Omega)}^2} \ci u  , v \right)_{L^2(\Omega)}.
\end{align*}
Hence, defining 
\begin{equation*}
F_{\kappa} := \widetilde{F}_{\kappa} + \alpha \ci u,
\end{equation*}
where
\begin{equation*}
\alpha := \frac{(\nabla \widetilde{u} ,  \nabla (\ci u) )_{L^2(\Omega)} - (\widetilde{F}_{\kappa} , \ci u)_{L^2(\Omega)} }{\| u \|_{L^2(\Omega)}^2},
\end{equation*} 
we verified that $\widetilde{u} \in \orthiu \subset H^1(\Omega)$ solves
\begin{align}
\label{tildeu-Neumann-problem}
\kappa^2(\widetilde{u},v)_{L^2(\Omega)} + (\nabla \widetilde{u} ,  \nabla v)_{L^2(\Omega)} = (F_{\kappa},v)_{L^2(\Omega)} \qquad \mbox{for all } v\in H^1(\Omega),
\end{align}
and it remains to estimate the Sobolev norms of $F_{\kappa}$.

On the one hand, using \eqref{stability-estimates-u}
and \eqref{stability-estimate-tilde-u}, we have
\begin{equation*}
|\alpha|
\leq
\frac{
\|\nabla \widetilde{u}\|_{L^2(\Omega)}\|\nabla u\|_{L^2(\Omega)}
+
\|\widetilde F_\kappa\|_{L^2(\Omega)} \|u\|_{L^2(\Omega)}
}{\|u\|_{L^2(\Omega)}^2}
\leq
 \frac{\kappa^2}{\|u\|_{L^2(\Omega)}} \left (
\| \widetilde{f} \|_{L^2(\Omega)} + \| \widetilde{g} \|_{H^1_\kappa(\Omega)}
\right ).
\end{equation*}
As a result, using again \eqref{stability-estimates-u}, it follows that
\begin{equation*}
\|\alpha \ci u\|_{H^1_\kappa(\Omega)}
\lesssim
\kappa^2
\frac{\|u\|_{H^1_\kappa(\Omega)}}{\|u\|_{L^2(\Omega)}}
\left (
\| \widetilde{f} \|_{L^2(\Omega)} + \| \widetilde{g} \|_{H^1_\kappa(\Omega)}
\right )
\lesssim
\kappa^2
\left (
\| \widetilde{f} \|_{L^2(\Omega)} + \| \widetilde{g} \|_{H^1_\kappa(\Omega)}
\right ).
\end{equation*}
On the other hand, using the product rules from \eqref{eq_product_rules}, we obtain
\begin{eqnarray*}
\| \widetilde{F}_{\kappa} \|_{H^{\ell-1}_{\kappa} (\Omega)}
&=&
\kappa^2\| \Lzero (\widetilde{f}) - \widetilde{f} + \kappa^{-1} \, \Lone( \nabla  \widetilde{g}) \|_{H^{\ell-1}_{\kappa}(\Omega)} \\
&\lesssim&
\kappa^2 \left (
(1+\| u \|_{W^{\ell-1,\infty}_{\kappa}(\Omega)}^2 ) \|\widetilde f\|_{H^{\ell-1}_\kappa(\Omega)} + \|\widetilde{g}\|_{H^\ell_\kappa(\Omega)}
\right ) \\
&\overset{\eqref{Wpq-estimates}}{\lesssim}&
\kappa^2 \kappa^{\varepsilon/p} \left ( \|\widetilde f\|_{H^{\ell-1}_\kappa(\Omega)} + \|\widetilde{g}\|_{H^\ell_\kappa(\Omega)}
\right ),
\end{eqnarray*}
which concludes the proof due to \eqref{tmp_shift_RD}.
\end{proof}

We are now ready to prove a splitting of $z$ into low-regularity contributions $z_j$ that do not depend on $\rho(\kappa)$ and a high-regularity remainder $r_{\ell}$ that scales with $\rho(\kappa)$.
Here, we closely follow the procedure introduced in \cite{CFN20}.

\begin{theorem}\label{theorem-splitting-z}
Assume that $\bfA \in W^{p-1,\infty}(\Omega)$ and $\partial \Omega$ is $C^{p+1}$,
and consider a local minimizer $u \in H^1(\Omega)$ of $E$ satisfying \ref{A4}. Then, for
all $f\in L^2(\Omega)$, we have
\begin{equation*}
\Einv f \,\,\,=\,\,\, \sum_{j=0}^{p-2} z_j \,+\, r_{p-1},
\end{equation*}
where, for $j = 0,\dots,p-2$ we define recursively $z_j \in \orthiu$ by
\begin{align*}
(\nabla z_0 , \nabla v )_{L^2(\Omega)} + \kappa^2 \, (z_0 , v )_{L^2(\Omega)} = \kappa^2 \, ( f , v)_{L^2(\Omega)} 
\qquad\mbox{for all } v\in \orthiu
\end{align*}
and
\begin{align*}
(\nabla z_1 , \nabla v )_{L^2(\Omega)} + \kappa^2 \, (z_1 , v )_{L^2(\Omega)} 
= - \kappa \, (\Lone(\nabla z_0) , v  )_{L^2(\Omega)} 
\qquad\mbox{for all } v\in \orthiu
\end{align*}
and
\begin{align*}
(\nabla z_j , \nabla v )_{L^2(\Omega)} + \kappa^2 \, (z_j , v )_{L^2(\Omega)} 
=  - \kappa^2 \, (\Lzero(z_{j-2}) -z_{j-2}, v  )_{L^2(\Omega)}  - \kappa \, (\Lone(\nabla z_{j-1}) , v  )_{L^2(\Omega)}
\end{align*}
for all $v\in \orthiu$. Then we have $z_j \in H^{j+2}(\Omega)$ and 
\begin{align*}
\| z_j \|_{H^{j+2}_{\kappa}(\Omega)} \,\, \lesssim \,\, \kappa^{\varepsilon/p} \, \| f\|_{L^2(\Omega)}
\end{align*}
and with 
\begin{align*}
r_{p-1} = z - \sum_{j=0}^{p -2} z_j, \qquad
\mbox{or equivalently } r_{p-1} = r_{p-2} - z_{p-2}
\end{align*}
we have $r_{p-1} \in H^{p+1}(\Omega)$ and
\begin{align*}
\| r_{p-1} \|_{H^{p+1}_{\kappa}(\Omega)} \,\, \lesssim \,\, \kappa^{\varepsilon} \, \rho(\kappa) \, \| f\|_{L^2(\Omega)}.
\end{align*}
\end{theorem}

\begin{proof}
First we note that $z_j \in H^{j+2}(\Omega)$ and the corresponding regularity estimate follows by a recursive application of Proposition \ref{proposition-shift}. 
Applying Lemma \ref{lemma-form-of-secE} to $z:=\Einv f$, i.e. the solution $z \in \orthiu$ to $\langle E^{\prime\prime}(u)z , v\rangle= ( f, v)_{L^2(\Omega)} $ for all $v\in \orthiu$, yields
\begin{align*}
 ( \nabla z ,\nabla v )_{L^2(\Omega)}
+  \kappa^2 ( z ,  v )_{L^2(\Omega)}
=  - \,\kappa (  \Lone(\nabla z) ,  v )_{L^2(\Omega)}
 - \kappa^2 ( \hspace{1pt}\Lzero(z) , v \hspace{1pt} )_{L^2(\Omega)} 
+ \kappa^2 ( f + z , v )_{L^2(\Omega)}.
\end{align*}
To prove the desired result for arbitrary $\ell \in \{ 1, \dots, p-1\}$, we plug $r_{\ell} = z - \sum\limits_{j=0}^{\ell -1} z_j$ into the above characterization of $z$. 
For $\ell=1$, i.e. $r_1= z - z_0$, we obtain
\begin{eqnarray*}
\lefteqn{ ( \nabla r_1 ,\nabla v )_{L^2(\Omega)}
+  \kappa^2 ( r_1 ,  v )_{L^2(\Omega)}
\,\,=\,\,  ( \nabla (z-z_0) ,\nabla v )_{L^2(\Omega)}
+  \kappa^2 ( z-z_0 ,  v )_{L^2(\Omega)} } \\
&=&  - \kappa^2 ( \hspace{1pt}\Lzero(z) - z, v \hspace{1pt} )_{L^2(\Omega)}  - \,\kappa (  \Lone(\nabla z) ,  v )_{L^2(\Omega)}.\hspace{100pt}
\end{eqnarray*}
For $r_2=r_1 -  z_1$ we obtain
\begin{eqnarray*}
 \lefteqn{ ( \nabla r_2 ,\nabla v )_{L^2(\Omega)}
+  \kappa^2 ( r_2 ,  v )_{L^2(\Omega)}
\,\,=\,\,
 ( \nabla r_1 - \nabla z_1 ,\nabla v )_{L^2(\Omega)}
+  \kappa^2 ( r_1 -  z_1 ,  v )_{L^2(\Omega)}  } \\
&=&
- \,\kappa \, (  \Lone(\nabla z) ,  v )_{L^2(\Omega)}
 - \kappa^2 ( \hspace{1pt}\Lzero(z) , v \hspace{1pt} )_{L^2(\Omega)} 
+ \kappa^2 (  z , v )_{L^2(\Omega)} 
+ \kappa \, (\Lone(\nabla z_0) , v  )_{L^2(\Omega)}  \\
&=&
- \,\kappa \, (  \Lone(\nabla r_1) ,  v )_{L^2(\Omega)}
 - \kappa^2 ( \hspace{1pt}\Lzero(z) - z , v \hspace{1pt} )_{L^2(\Omega)}.
\end{eqnarray*}
For $r_{\ell} =  r_{\ell-1} - z_{\ell-1}$ we obtain the equation inductively: Let $r_0:= z$ and let, for $\ell \ge 3$ hold 
\begin{eqnarray*}
  ( \nabla r_{\ell-1} ,\nabla v )_{L^2(\Omega)}
+  \kappa^2 ( r_{\ell-1} ,  v )_{L^2(\Omega)}
\,=\,
- \,\kappa \, (  \Lone(\nabla r_{\ell-2}) ,  v )_{L^2(\Omega)}
 - \kappa^2 ( \hspace{1pt}\Lzero(r_{\ell-3}) - r_{\ell-3} , v \hspace{1pt} )_{L^2(\Omega)}.
\end{eqnarray*}
Then we obtain
\begin{eqnarray*}
 \lefteqn{ ( \nabla r_{\ell} ,\nabla v )_{L^2(\Omega)}
+  \kappa^2 ( r_{\ell} ,  v )_{L^2(\Omega)}
\,\,=\,\,
 ( \nabla r_{\ell-1} - \nabla z_{\ell-1} ,\nabla v )_{L^2(\Omega)}
+  \kappa^2 ( r_{\ell-1} - z_{\ell-1} ,  v )_{L^2(\Omega)}  } \\
&=&
- \,\kappa \, (  \Lone(\nabla r_{\ell-2}) ,  v )_{L^2(\Omega)}
 - \kappa^2 ( \hspace{1pt}\Lzero(r_{\ell-3}) - r_{\ell-3} , v \hspace{1pt} )_{L^2(\Omega)} \\
 &\enspace&\quad
 + \kappa^2 \, (\Lzero(z_{\ell-3}) - z_{\ell-3}, v  )_{L^2(\Omega)}  +  \kappa \, (\Lone(\nabla z_{\ell-2}) , v  )_{L^2(\Omega)} \\
 &=&
- \,\kappa \, (  \Lone(\nabla r_{\ell-2} - \nabla z_{\ell-2} ) ,  v )_{L^2(\Omega)}
 - \kappa^2 ( \hspace{1pt}\Lzero(r_{\ell-3}-z_{\ell-3}) - (r_{\ell-3}-z_{\ell-3}) , v \hspace{1pt} )_{L^2(\Omega)}\\
  &=&
- \,\kappa \, (  \Lone(\nabla r_{\ell-1} ) ,  v )_{L^2(\Omega)}
 - \kappa^2 ( \hspace{1pt}\Lzero(r_{\ell-2}) - r_{\ell-2} , v \hspace{1pt} )_{L^2(\Omega)}.
\end{eqnarray*}
From the equation for $r_1$ we have
\begin{align*}
\| r_{1} \|_{H^2_{\kappa}(\Omega)} \,\,\lesssim\,\,  \| z \|_{H^1_{\kappa}(\Omega)}
\,\, \overset{\eqref{secEest-H1}}{\lesssim} \ \, \rho(\kappa) \, \| f\|_{L^2(\Omega)}
\end{align*}
and by applying Proposition \ref{proposition-shift} to the equation for $r_1$, we also obtain
\begin{align*}
\| r_{1} \|_{H^3_{\kappa}(\Omega)} & \,\,\overset{\eqref{regularity-estimate-tilde-u}}{\lesssim}\,\, 
\kappa^{\varepsilon/p}  \| z \|_{H^2_{\kappa}(\Omega)} 
\,\, \overset{\eqref{secEest-H2}}{\lesssim} \ \, \kappa^{\varepsilon/p} \, \rho(\kappa) \, \| f\|_{L^2(\Omega)}.
\end{align*}
For $r_2$ we obtain analogously with Proposition \ref{proposition-shift} 
\begin{align*}
\| r_{2} \|_{H^4_{\kappa}(\Omega)} &\,\,\overset{\eqref{regularity-estimate-tilde-u}}{\lesssim}\,\,  \kappa^{\varepsilon/p} \left(  \| z\|_{H^2_{\kappa}(\Omega)} + \| r_1 \|_{H^{3}_{\kappa}(\Omega)}
\right) \,\,\lesssim\,\, \kappa^{\varepsilon/p} \, (1 + \kappa^{\varepsilon/p}) \, \rho(\kappa) \, \| f\|_{L^2(\Omega)}.
\end{align*}
Repeating the argument recursively at most $p-1$ times and using $(\kappa^{\varepsilon/p})^{p-1}\le \kappa^{\varepsilon}$ proves 
\begin{align*}
\| r_{\ell} \|_{H^{\ell+2}_{\kappa}(\Omega)} &\,\,\lesssim\,\, \kappa^{\varepsilon} \rho(\kappa) \,\| f\|_{L^2(\Omega)}.
\end{align*}
The result is obtained for $\ell=p-1$.
\end{proof}

\section{Proofs of the $L^2$ and $H^1_{\kappa}$-error estimates}
\label{section:proofs:L2-H1-estimates}
With the considerations in Section \ref{section-secE-shift-prop} regarding the shift-property of the operator $E^{\prime\prime}(u)$, we are now prepared to turn towards the error estimates for the minimizers. This takes place in three steps. In the first step (Section \ref{subsection-Ritz-projection-secE}), we will introduce the Ritz projection $\Rh$ based on $E^{\prime\prime}(u)$ and quantify the corresponding projection error, as well as the error between $\Rh(u)$ and a discrete minimizer $\uhp$. In the second step (Section \ref{subsection-preliminary-error-bounds}), we will derive uniform $\kappa$-bounds for the distance between an exact local minimizer $u$ and the closest discrete local minimizer $\uhp$. These uniform bounds are suboptimal, but will be sufficient to get control over higher order error contributions that enter through the nonlinearity of the Ginzburg--Landau equation. Finally, in step three (Section \ref{subsection:proof:main-result}), we combine all our findings to derive the desired error estimates between $u$ and $\uhp$ in the $L^2$- and $H^1_{\kappa}$-norm.

\subsection{Approximation factor}

We are now ready to quantify the worst-case best-approximation error $\eta_{h}$ associated
with solutions to problems involving the invertible operator $E^{\prime\prime}(u)\vert_{\orthiu}$.
Specifically, we define the worst-case best-approximation error by
\begin{align}
\label{def-eta-h-p}
\eta_{h}
:=
\sup_{\substack{f \in L^2(\Omega) \\ \|f\|_{L^2(\Omega)} = 1}}
\min_{v_h \in V_{h} \cap \orthiu}
\| \Einv f  -  v_h \|_{H^1_{\kappa}(\Omega)}.
\end{align}

\begin{theorem}\label{theorem:eta-hp-estimate}
Let $u \in H^1(\Omega)$ denote a local minimizer of $E$ such that \ref{A4} is fulfilled.
Then,
\begin{equation*}
\lim_{h \to 0} \eta_h = 0.
\end{equation*}
In addition, if $\bfA \in W^{p-1,\infty}(\Omega)$ and $\partial \Omega$ is $C^{p+1}$, then
it holds that
\begin{align}
\label{est-eta-h-p}
\eta_{h} \,\,\lesssim\,\,
\kappa^{\varepsilon} \left( \, h \kappa + \rho(\kappa)(h \kappa)^{p} \, \right).
\end{align}
\end{theorem}
\begin{proof}
The fact that $\eta_h \to 0$ as $h \to 0$ without further regularity assumption
is a simple consequence of the fact that the injection $\Einv(L^2(\Omega)) \hookrightarrow H^1(\Omega)$
is compact. Indeed $\Einv$ is continuous from $(H^1(\Omega))^\prime \to H^1(\Omega)$ and the
injection $L^2(\Omega) \hookrightarrow (H^1(\Omega))^\prime$ is compact.

We then focus on \eqref{est-eta-h-p}.
For $f \in L^2(\Omega)$, we consider $\zf{} := \Einv f  \in \orthiu$.
To exploit Theorem \ref{theorem-splitting-z} we let $\zf{j} \in \orthiu$
denote the solutions to
\begin{eqnarray*}
(\nabla \zf{0} , \nabla v )_{L^2(\Omega)} + \kappa^2 \, (\zf{0}, v )_{L^2(\Omega)} &=& \kappa^2 \, ( f , v)_{L^2(\Omega)} \\[0.3em]
(\nabla \zf{1} , \nabla v )_{L^2(\Omega)} + \kappa^2 \, (\zf{1}, v )_{L^2(\Omega)} 
&=& - \kappa \, (\Lone(\nabla \zf{ 0 }) , v  )_{L^2(\Omega)}  \\[0.3em]
(\nabla \zf{j} , \nabla v )_{L^2(\Omega)} + \kappa^2 \, (\zf{j} , v )_{L^2(\Omega)} 
&=&  - \kappa^2 \, (\Lzero( \zf{j-2}) - \zf{j-2}, v  )_{L^2(\Omega)}  - \kappa \, (\Lone(\nabla \zf{j-1}) , v  )_{L^2(\Omega)}
\end{eqnarray*}
for all $v\in \orthiu$. Setting
\begin{align*}
r_{p-1} := \zf{} - \sum_{j=0}^{p -2} \zf{j},
\end{align*}
Theorem \ref{theorem-splitting-z} yields $\zf{j} \in H^{j+2}(\Omega)$ and $r_{p-1} \in H^{p+1}(\Omega)$ with 
\begin{align*}
\|  \zf{j} \|_{H^{j+2}_{\kappa}(\Omega)}  \,\, \lesssim\,\,
\kappa^{\varepsilon/p} \, \| f \|_{L^2(\Omega)}
\,\,\lesssim\,\, \kappa^\varepsilon \|f\|_{L^2(\Omega)}
\qquad
\mbox{and}
\qquad
\| r_{p-1} \|_{H^{p+1}_{\kappa}(\Omega)} \,\,\lesssim \,\,
\kappa^{\varepsilon} \, \rho(\kappa) \, \| f \|_{L^2(\Omega)}.
\end{align*}
We now construct an interpolation operator $I_h^{\perp} : \orthiu \rightarrow V_{h} \cap \orthiu$ by setting
\begin{align*}
I_h^{\perp} (v) := I_h (v) - \frac{(I_h(v) , \ci u )_{L^2(\Omega)} }{ (I_h(\ci u) , \ci u )_{L^2(\Omega)} } \, I_h(\ci u)
\end{align*}
for all $v \in \orthiu$. Since $h \kappa \leq 1$ by assumption, for every $v \in \orthiu$,
we have
\begin{eqnarray*}
\| v - I_h^{\perp} (v) \|_{H^1_{\kappa}(\Omega)} 
&\lesssim& \| v -I_h (v) \|_{H^1_{\kappa}(\Omega)}.
\end{eqnarray*}
The argument is e.g. elaborated in \cite[Proof Lem.~4.6]{BDH25} and exploits $\tfrac{1}{\kappa} \| \nabla u \|_{L^2(\Omega)} \lesssim \| u\|_{L^2(\Omega)}$.

With this, we obtain
\begin{eqnarray*}
\lefteqn{
\inf_{v_h \in V_{h} \cap \orthiu}  \|\zf{}  -  v_h\|_{H^1_{\kappa}(\Omega)} 
  \,\,\leq\,\,
\|\zf{} - I_h^\perp(\zf{})\|_{H^1_{\kappa}(\Omega)} 
  \,\,\lesssim\,\,
\| \zf{} -  I_h(\zf{}) \|_{H^1_{\kappa}(\Omega)} } \\
&=& \| r_{p-1} + \sum\limits_{j=0}^{p -2} \zf{j}  - I_h(r_{p-1}) - \sum\limits_{j=0}^{p -2} I_h(\zf{j})  \|_{H^1_{\kappa}(\Omega)} \\
&\le& \| r_{p-1} - I_h(r_{p-1})\|_{H^1_{\kappa}(\Omega)} + \sum\limits_{j=0}^{p -2}  \| \zf{j}  -  I_h(\zf{j})  \|_{H^1_{\kappa}(\Omega)} \\
&\lesssim&  h^{p} \kappa^{p} \| r_{p-1} \|_{H^{p+1}_{\kappa}(\Omega)} + \sum\limits_{j=0}^{p-2} h^{j+1} \kappa^{j+1} \| \zf{j} \|_{H^{j+2}_{\kappa}(\Omega)} \\
&\lesssim&
(h \kappa)^{p} \kappa^{\varepsilon} \, \rho(\kappa) \, \| f \|_{L^2(\Omega)} + \sum\limits_{j=0}^{p-2} h^{j+1}  \kappa^{j+1} \kappa^{\varepsilon}\, \| f \|_{L^2(\Omega)}
\\
&\lesssim&
\kappa^{\varepsilon}
\left( h \, \kappa + (h \, \kappa)^{p}  \, \rho(\kappa)  \right) \, \|f\|_{L^2(\Omega)}.
\end{eqnarray*}
This finishes the proof.
\end{proof}

\subsection{Estimates for the $E^{\prime\prime} (u)$-Ritz projection}
\label{subsection-Ritz-projection-secE}
The Ritz projection associated with the elliptic operator
$E^{\prime\prime} (u) \vert_{\orthiu}$ will play a central role in
what follows. For any $v \in \orthiu$, we define $\Rh(v) \in V_{h} \cap \orthiu$
by requiring that
\begin{align}
\label{eq_def_Rh}
\langle E^{\prime\prime} (u) \Rh(v) , v_h \rangle = \langle E^{\prime\prime} (u) \,v , v_h \rangle  \qquad \mbox{for all } v_h \in V_{h} \cap \orthiu.
\end{align}
Note that $\Rh$ is well-defined by the coercivity of $E^{\prime\prime}(u)$ on $\orthiu$.
The remainder of this section is dedicated to corresponding error estimates that
are important for our analysis.

We start by showing an error estimate for the Ritz projection $R_h$.
Since this projection corresponds to a coercive bilinear form with a ``small''
coercivity constant, naive arguments lead to error estimates with unfavorable
$\kappa$ dependence. Since on the other hand, this bilinear form satisfies a G\aa rding
inequality with favorable constants, we can employ the standard ``Schatz argument''
\cite{Sch74} to obtain sharp estimate for sufficiently fine meshes.
 
\begin{lemma}\label{lemma:estimate-ritz-projection}
Consider a local minimizer $u \in H^1(\Omega)$ of $E$ satisfying \ref{A4}.
Recall $\eta_{h}$ from \eqref{def-eta-h-p}. There exists $\eta_\star \gtrsim 1$
(i.e. independent of $\kappa$ and $h$) such that if
$\eta_{h} \leq \eta_\star$
then, for any $v \in \orthiu$, it holds
\begin{align*}
\| v - \Rh(v) \|_{H^1_{\kappa}(\Omega)} \,\, \lesssim \,\, \inf_{v_h \in V_{h} }  \| v - v_h \|_{H^1_{\kappa}(\Omega)}.
\end{align*}
and
\begin{align*}
\| v - \Rh(v) \|_{L^2(\Omega)} &\lesssim \, \eta_{h} \,
\| v - \Rh(v)  \|_{H^1_{\kappa}(\Omega)}.
\end{align*}
\end{lemma}

\begin{proof}
First, the G{\aa}rding inequality in Lemma \ref{lemma:stability-garding} yields
\begin{align*}
\|  v - \Rh(v)  \|_{H^1_{\kappa}(\Omega)}^2
\,\,\lesssim\,\,
\langle E^{\prime\prime} (u) ( v - \Rh(v) ) ,  v - \Rh(v)  \rangle
+ \|  v - \Rh(v)  \|_{L^2(\Omega)}^2.
\end{align*}
If $\xi \in \orthiu$ solves
\begin{align*}
\langle E^{\prime\prime} (u) \xi , \phi \rangle  = (  \Rh(v) - v  , \phi )_{L^2(\Omega)} \qquad \mbox{for all } \phi \in \orthiu,
\end{align*}
then we obtain for arbitrary $\xi_h, v_h \in V_{h} \cap \orthiu$ that
\begin{eqnarray*}
\lefteqn{
\| \Rh(v)-v \|_{H^1_{\kappa}(\Omega)}^2 \,\,\lesssim\,\, \langle E^{\prime\prime} (u) (\Rh(v)-v +\xi ) , \Rh(v)-v \rangle }\\
&=&  \langle E^{\prime\prime} (u) (v_h -v ) , \Rh(v)-v \rangle + \langle E^{\prime\prime} (u) (\xi -\xi_h ) , \Rh(v)-v \rangle \\
&\lesssim& \left( \|   v_h -v \|_{H^1_{\kappa}(\Omega)} +  \|  \xi -\xi_h \|_{H^1_{\kappa}(\Omega)} \right)  \|  \Rh(v)-v  \|_{H^1_{\kappa}(\Omega)}.
\end{eqnarray*}
The definition of $\eta_h$ in \eqref{def-eta-h-p} yields
\begin{eqnarray*}
\| \Rh(v)-v \|_{H^1_{\kappa}(\Omega)} 
&\lesssim& \inf_{v_h \in V_{h} \cap \orthiu}  \|   v_h -v \|_{H^1_{\kappa}(\Omega)} +  \eta_{h} \, \|  \Rh(v)-v  \|_{L^2(\Omega)}.
\end{eqnarray*} 
Since $ \|  \Rh(v)-v  \|_{L^2(\Omega)} \le \| \Rh(v)-v \|_{H^1_{\kappa}(\Omega)}$ and since (analogously as in the proof of  Theorem \ref{theorem:eta-hp-estimate})
\begin{align*}
\inf_{v_h \in V_{h} \cap \orthiu}  \|   v_h -v \|_{H^1_{\kappa}(\Omega)} 
\lesssim \inf_{v_h \in V_{h} }  \|   v_h -v \|_{H^1_{\kappa}(\Omega)},
\end{align*}
we obtain
\begin{eqnarray*}
\left(1-  \frac{1}{2}\frac{\eta_{h}}{\eta_\star} \right) \| \Rh(v)-v \|_{H^1_{\kappa}(\Omega)} 
&\lesssim& \inf_{v_h \in V_{h} }  \|   v_h -v \|_{H^1_{\kappa}(\Omega)}
\end{eqnarray*} 
for some generic $\eta_\star \gtrsim 1$.
For the $L^2$-estimates we write
\begin{eqnarray*}
\lefteqn{
\| \Rh(v) - v \|_{L^2(\Omega)}^2 \,\,\,=\,\,\,  \langle E^{\prime\prime}(u) (\xi ) , \Rh(v) - v \rangle }\\
&=& \inf_{\xi_h \in V_{h} \cap \orthiu} \langle E^{\prime\prime}(u) (\xi -\xi_h) , \Rh(v) - v \rangle 
\,\,\,\lesssim\,\,\,  \eta_{h} \, \|  \Rh(v)-v  \|_{L^2(\Omega)} \| \Rh(v)-v  \|_{H^1_{\kappa}(\Omega)}.
\end{eqnarray*}
\end{proof}
Next, we establish an estimate for the error between a discrete minimizer $u_h$ and the Ritz projection $\Rh(u)$ of an exact minimizer $u$. 
\begin{lemma}
\label{lemma:H1-est-Rh}
Assume that $u \in H^1(\Omega)$ is a (local) minimizer of $E$ 
satisfying \ref{A4} and let $u_h \in V_h \cap \orthiu$ be an arbitrary discrete critical point such that
$\langle E^{\prime}(u_h),v_h\rangle = 0 $ for all $v_h \in V_{h}$ (cf. Remark \ref{remark-gauge-choice}). 
It holds
\begin{eqnarray*}
\lefteqn{\| \Rh(u)-u_h \|_{H^1_{\kappa}(\Omega)}}
\\
&\lesssim&
\eta_{h} \| \Rh(u)-u_h \|_{L^2(\Omega)}
+
\rho(\kappa) \left (
\| u-u_h\|_{L^4(\Omega)}^2
+
\kappa^{d/4} \| u-u_h\|_{L^4(\Omega)}^3
\right ).
\end{eqnarray*}
\end{lemma}

\begin{proof}
We start as in the proof of Lemma \ref{lemma:estimate-ritz-projection} to estimate
$\Rh(u)-u_h$. Lemma \ref{lemma:stability-garding} yields
\begin{align*}
\|  \Rh(u)-u_h  \|_{H^1_{\kappa}(\Omega)}^2
\,\,\lesssim\,\,
\langle E^{\prime\prime} (u) ( \Rh(u)-u_h  ) ,  \Rh(u)-u_h \rangle
\,+\,
\|  \Rh(u) - u_h  \|_{L^2(\Omega)}^2
\end{align*}
and accordingly we let $\xi \in \orthiu$ denote the unique solution to
\begin{align*}
\langle E^{\prime\prime} (u) \xi , v \rangle
=
(  \Rh(u)-u_h  , v )_{L^2(\Omega)} \qquad \mbox{for all } v \in \orthiu.
\end{align*}
Recalling that $\langle E^{\prime\prime} (u) \, \cdot , \cdot \rangle$ is symmetric, we conclude for arbitrary $\xi_h \in V_{h} \cap \orthiu$ that
\begin{eqnarray*}
\lefteqn{
\| \Rh(u)-u_h \|_{H^1_{\kappa}(\Omega)}^2 \,\,\lesssim\,\, \langle E^{\prime\prime} (u) (\Rh(u)-u_h+\xi) , \Rh(u)-u_h \rangle }\\
&=&  \langle E^{\prime\prime} (u) \xi , \Rh(u)-u_h \rangle + \langle E^{\prime\prime} (u) (\Rh(u)-u_h) , \Rh(u)-u_h \rangle \\
&=&   \underbrace{\langle E^{\prime\prime} (u) (\xi  - \xi_h) , \Rh(u)-u_h \rangle}_{=:\,\mbox{I}}  + \underbrace{ \langle E^{\prime\prime} (u) (\Rh(u)-u_h + \xi_h) , u-u_h \rangle}_{=:\,\mbox{II}}.
\end{eqnarray*}
We now estimate the two terms separately. For that, let us select $\xi_h \in V_{h} \cap \orthiu$ through the $H^1_{\kappa}$-orthogonal projection, i.e.,
\begin{align*}
\| \xi - \xi_h \|_{H^1_{\kappa}(\Omega)} \,\, = \,\, \inf_{ v_h  \in V_{h} \cap \orthiu} \| \xi - v_h \|_{H^1_{\kappa}(\Omega)}.
\end{align*}
{\it Term I.}  
The definition of $\eta_h$ in \eqref{def-eta-h-p} yields directly
\begin{align*}
|\langle E^{\prime\prime} (u) (\xi  - \xi_h) , \Rh(u)-u_h \rangle| \,\, &\lesssim \,\, 
\| \xi  - \xi_h \|_{H^1_{\kappa}(\Omega)} \|  \Rh(u)-u_h \|_{H^1_{\kappa}(\Omega)} \\
\,\,&\le\,\, \eta_{h}  \|  \Rh(u)-u_h \|_{H^1_{\kappa}(\Omega)} \, \| \Rh(u)-u_h \|_{L^2(\Omega)}.
\end{align*}
{\it Term II.} We start with rewriting
\begin{equation*}
\langle E^{\prime\prime} (u)v_h,( u-u_h) \rangle
=
\langle E^{\prime\prime} (u)( u-u_h) ,v_h \rangle
\end{equation*}
for  arbitrary $v_h \in V_{h} \cap \orthiu$. Using the continuous and the discrete Ginzburg-Landau equations we have
\begin{align*}
\langle E^{\prime}(u) , v_h \rangle \,\,=\,\, \langle E^{\prime}(u_h) , v_h \rangle  \,\, =\,\,  0 \quad \mbox{for all } v_h\in V_{h}.
\end{align*}
We therefore obtain from \eqref{derivaitves-1} and \eqref{derivaitves-2} (cf. \cite[Lem.~5.5]{DoeHe24})
\begin{eqnarray*}
\lefteqn{ \langle E^{\prime\prime} (u)( u-u_h) ,v_h \rangle \,\,\,=\,\,\, \langle E^{\prime\prime} (u) u -  E^{\prime\prime} (u_h) u_h  ,v_h \rangle  
+ \langle ( E^{\prime\prime} (u_h) - E^{\prime\prime} (u) ) u_h  ,v_h \rangle } \\
&=& 2 ( |u|^2 u , v_h )_{L^2(\Omega)}  +  ( |u_h|^2 u_h , v_h )_{L^2(\Omega)}  - (2|u|^2 u_h + u^2 \overline{u_h} , v_h )_{L^2(\Omega)}  \\
&=&  ( 2 u |u-u_h|^2 + (u-u_h)^2 \overline{u} - |u-u_h|^2 (u-u_h)  , v_h )_{L^2(\Omega)} .
\end{eqnarray*}
We conclude with $v_h = \Rh(u) - u_h + \xi_h \in V_{h} \cap \orthiu$ that
\begin{eqnarray*}
\lefteqn{ |\langle E^{\prime\prime} (u)( u-u_h) ,\Rh(u) - u_h + \xi_h \rangle| }\\
&\lesssim&   \| u-u_h\|_{L^4(\Omega)}^2 \, \| \Rh(u) - u_h + \xi_h \|_{L^2(\Omega)} + \| u-u_h\|_{L^4(\Omega)}^3  \, \| \Rh(u) - u_h + \xi_h \|_{L^4(\Omega)}.
\end{eqnarray*}
In order to estimate $\| \Rh(u) - u_h + \xi_h \|_{L^4(\Omega)}$
we exploit the Gagliardo--Nirenberg inequality for $d=2,3$,
which yields for any $v \in H^1(\Omega)$ that
\begin{align*}
\| v \|_{L^4(\Omega)} \,\lesssim\, \| \nabla v \|_{L^2(\Omega)}^{d/4}  \, \|  v \|_{L^2(\Omega)}^{1-d/4}
\qquad
\mbox{and hence}
\qquad
\| v \|_{L^4(\Omega)}^2 
\,\lesssim\, \kappa^{d/2} \| v \|_{H^1_{\kappa}(\Omega)}^2.
\end{align*}
We conclude
\begin{eqnarray*}
\lefteqn{ |\langle E^{\prime\prime} (u)( u-u_h) ,\Rh(u) - u_h + \xi_h \rangle| }\\
&\lesssim&  \left( \| u-u_h\|_{L^4(\Omega)}^2 + \kappa^{d/4} \| u-u_h\|_{L^4(\Omega)}^3 \right) \, \| \Rh(u) - u_h + \xi_h \|_{H^1_{\kappa}(\Omega)}
 \\
&\lesssim&  \left( \| u-u_h\|_{L^4(\Omega)}^2 + \kappa^{d/4} \| u-u_h\|_{L^4(\Omega)}^3 \right) \,
\left( \, \| \Rh(u) - u_h \|_{H^1_{\kappa}(\Omega)}  + \| \xi \|_{H^1_{\kappa}(\Omega)} \right) \\
&\lesssim&  \left( \| u-u_h\|_{L^4(\Omega)}^2 + \kappa^{d/4} \| u-u_h\|_{L^4(\Omega)}^3 \right) \,\, (1+ \rho(\kappa))\,
\| \Rh(u) - u_h \|_{H^1_{\kappa}(\Omega)} .
\end{eqnarray*}
The estimate $\Ccoe \gtrsim 1$ as per~\eqref{Ccoe_gtrsim_1}, finishes the proof.
\end{proof}
In order to replace one of the $L^4$-error contributions in Lemma \ref{lemma:H1-est-Rh} by a contribution in the $H^1_{\kappa}$-norm, we can again use the Gagliardo-Nirenberg interpolation inequality.
Note that it is important to only replace one $L^4$-contribution by an $H^1_{\kappa}$-term and to keep the remaining ones. Together with some $L^4$-bounds that follow later in Theorem \ref{theorem:existence-of-local-discrete-minimizer}, this will allow us to remove some artificial $\kappa$-dependencies which would otherwise remain. With this in mind, we formulate the following super-convergence result.
\begin{conclusion}[Asymptotic superconvergence]
\label{conclusion:estimate-uh-Rhu}
Assume that $u \in H^1(\Omega)$ is a (local) minimizer of $E$ 
and $u_h \in V_h \cap \orthiu$ is an arbitrary discrete critical point such that
$\langle E^{\prime}(u_h),v_h\rangle = 0 $ for all $v_h \in V_{h}$.
Then there exists $\eta_\star \gtrsim 1$ such that if $\eta_h \leq \eta_\star$ we have
\begin{eqnarray*}
  \| \Rh(u)-u_h \|_{H^1_{\kappa}(\Omega)} 
&\lesssim&
  \rho(\kappa) \left( \kappa^{d/4} \| u-u_h\|_{L^4(\Omega)} +  \kappa^{d/2} \| u-u_h\|_{L^4(\Omega)}^2 \right)  \,  \| u-u_h\|_{H^1_{\kappa}(\Omega)}.
\end{eqnarray*}

\end{conclusion}

\begin{proof}
By Lemma \ref{lemma:H1-est-Rh} we have
\begin{eqnarray*}
\lefteqn{ \| \Rh(u)-u_h \|_{H^1_{\kappa}(\Omega)}  }\\
&\leq & \frac{1}{2\eta_\star} \left (\eta_h  \| \Rh(u)-u_h \|_{L^2(\Omega)}
 + \rho(\kappa) \hspace{-2pt}\left(  \| u-u_h\|_{L^4(\Omega)}^2 + \kappa^{d/4} \| u-u_h\|_{L^4(\Omega)}^3 \right)
\right )
\end{eqnarray*}
for some $\eta_\star \gtrsim 1$. Due to our assumption on $\eta_{h}$,
we can absorb the $\| \Rh(u)-u_h \|_{L^2(\Omega)}$-contribution and the estimate reduces to
\begin{eqnarray}
\label{proof-conclusion-uh-Rhugh-est}
 \| \Rh(u)-u_h \|_{H^1_{\kappa}(\Omega)}  
&\lesssim&
\rho(\kappa) \hspace{-2pt}\left(  \| u-u_h\|_{L^4(\Omega)}^2 + \kappa^{d/4} \| u-u_h\|_{L^4(\Omega)}^3 \right).
\end{eqnarray}
Using again the Gagliardo--Nirenberg inequality $\| v \|_{L^4(\Omega)} \,\lesssim\, \kappa^{d/4} \| v \|_{H^1_{\kappa}(\Omega)}$ on one of the $L^4$-terms, the $H^1_{\kappa}$ estimate \eqref{proof-conclusion-uh-Rhugh-est} becomes
\begin{eqnarray*}
 \| \Rh(u)-u_h \|_{H^1_{\kappa}(\Omega)} 
&\lesssim&  \, \rho(\kappa) \left( \kappa^{d/4} \| u-u_h\|_{L^4(\Omega)} +  \kappa^{d/2} \| u-u_h\|_{L^4(\Omega)}^2 \right) \, \| u-u_h\|_{H^1_{\kappa}(\Omega)}.
\end{eqnarray*}
\end{proof}
Lemma \ref{lemma:estimate-ritz-projection} and Conclusion \ref{conclusion:estimate-uh-Rhu} finish the $H^1$-estimate, if we can show that the higher order contribution $ \rho(\kappa) \left( \kappa^{d/4} \| u-u_h\|_{L^4(\Omega)} +  \kappa^{d/2} \| u-u_h\|_{L^4(\Omega)}^2 \right)  \,  \| u-u_h\|_{H^1_{\kappa}(\Omega)}$ is negligible provided that $u$ and $u_h$ are related in a suitable way.

\begin{remark}
The above proof technique relies on a regularity splitting, and we have followed
the approach of~\cite{CFN20} to obtain such a splitting. Other approaches are possible,
and we refer in particular to~\cite{bernkopf_chaumontfrelet_melenk_2025a,MeS10}.
We emphasize, however, that all these works for the Helmholtz equation assume fixed coefficients,
whereas the coefficients appearing in $E''(u)$ depend on~$\kappa$ through the exact solution~$u$.
We have chosen to extend the methodology proposed in~\cite{CFN20} because it provided
the most natural framework for our setting, and have not investigated whether
the other techniques in~\cite{bernkopf_chaumontfrelet_melenk_2025a,MeS10}
can be similarly extended to oscillating coefficients.

We further note that, for the Helmholtz equation, so-called pre-asymptotic error
estimates can be obtained under weaker resolution conditions~\cite{chaumontfrelet_spence_2024a,galkowski2025sharp,li2025higher},
in regimes where the finite element solution is not (uniformly in~$\kappa$) quasi-optimal.
Such pre-asymptotic estimates rely on refined G{\aa}rding inequalities
(so far established only for fixed coefficients) and are outside the scope of the present work.
\end{remark}

\subsection{Preliminary error bounds}
\label{subsection-preliminary-error-bounds}
In this section we want to derive a preliminary (pessimistic) estimate for the error $ \| u-u_h\|_{H^1_{\kappa}(\Omega)}$. More precisely, we want to identify a resolution condition such that it holds 
\begin{align*}
\rho(\kappa) \left( \kappa^{d/4} \| u-u_h\|_{L^4(\Omega)} +  \kappa^{d/2} \| u-u_h\|_{L^4(\Omega)}^2 \right)  \,  \| u-u_h\|_{H^1_{\kappa}(\Omega)} 
\, \,\ll \, \,
\| u-u_h\|_{H^1_{\kappa}(\Omega)}.
\end{align*}
If such an estimate is available, we can later absorb the term in the final estimates. In order to obtain the desired bound, we take inspiration from \cite{PoRa94} where an abstract formalism for the solution of nonlinear differential equations is developed. However, we cannot directly use the findings and arguments of \cite{PoRa94} in our work, as this would result in far too pessimistic estimates (with respect to $\kappa$) to be applicable in our setting.

In the first step towards the desired estimates, we introduce an alternative characterization of a discrete local minimizer.
For this we will employ a linearization of $E'$ around $u$ which we denote by $\gradEh$. To be precise, the linearized approximation 
$\gradEh : H^1(\Omega) \to H^1(\Omega)^{\ast}$ 
is defined by
\begin{align*}
\langle \gradEh(v) , w \rangle := \langle E^{\prime}(v) , \Rh(w) \rangle + \langle E^{\prime\prime}(u) v,w-  \Rh(w) \rangle
\end{align*}
for all $v,w \in H^1(\Omega)$, where $\Rh$ is the Ritz projection introduced in~\eqref{eq_def_Rh}. 
Note that the dot in $\gradEh$ is purely notational and serves to indicate that it approximates the derivative $E^\prime$.
\begin{lemma}
\label{lemma:characterization:Ehprime}
Consider a (local) minimizer $u\in H^1(\Omega)$ of $E$ such that \ref{A4} holds. Then,
\begin{align*}
u_h \in \orthiu  \,\,\,\, \mbox{solves} \,\,\,\, \gradEh (u_h)= 0 
\end{align*}
if and only if
\begin{align*}
u_h \in  V_{h} \cap \orthiu \qquad 
\mbox{and} \qquad
 \langle E^{\prime}(u_h),v_h \rangle = 0 \quad \mbox{for all } v_h \in  V_{h} \cap \orthiu. 
\end{align*}
\end{lemma}
\begin{proof}
The direction \quotes{$\Leftarrow$} is obvious. Conversely, let  $u_h \in \orthiu$ solve $\gradEh (u_h)= 0$, then for all $w_h \in V_{h} \cap \orthiu$ we have $ \langle E^{\prime}(u_h) , w_h \rangle = 0$. This in turn implies $\langle E^{\prime\prime}(u) u_h,w-  \Rh(w) \rangle =0$ for all $w \in \orthiu$ and hence, since $\Rh$ is the $\langle E^{\prime\prime}(u)\cdot,\cdot\rangle$-orthogonal projection and \ref{A4} holds, we have $u_h \in V_{h} \cap \orthiu$.
\end{proof}
Our goal is to identify a neighborhood of $u$ such that, in that neighborhood, there is a unique solution $u_h \in \orthiu$ with $\gradEh (u_h)= 0$. The size of the neighborhood will allow us to give a pessimistic estimate for $\| u - u_h \|_{H^1_{\kappa}(\Omega)}$ provided that another resolution condition is fulfilled that remains to be identified.

Before we can prove the existence of a suitable neighborhood, we require two more auxiliary results.
The first one collects various properties of $\gradEh^{\hspace{1pt}\prime}$ and its inverse.

\begin{lemma}\label{lemma:properties_Eh_sec}
Consider a (local) minimizer $u \in H^1(\Omega)$ that fulfills \ref{A4}.
It holds that
\begin{align}
\label{eq_identity_Ehpp}
\langle \gradEh^{\hspace{1pt}\prime}(v) w , z \rangle = \langle E^{\prime\prime}(v) w , \Rh(z) \rangle + \langle E^{\prime\prime}(u) w ,z-  \Rh(z) \rangle
\end{align}
for all $v,w,z \in \orthiu$. In particular, we have
\begin{align*}
\langle \gradEh^{\hspace{1pt}\prime}(u) w , z \rangle = \langle E^{\prime\prime}(u) w , z \rangle
\end{align*}
and hence coercivity of $\langle \gradEh^{\hspace{1pt}\prime}(u) \,\cdot , \cdot \rangle$ on $\orthiu$ with
\begin{align*}
\langle \gradEh^{\hspace{1pt}\prime}(u) w , w \rangle \,\,\ge\,\, \rho(\kappa)^{-1} \| w \|_{H^1_{\kappa}(\Omega)}^2 \qquad
\mbox{for all } w\in \orthiu.
\end{align*}
Furthermore, we have 
\begin{align}
\label{id-u-rhu}
E^{\prime\prime}(u)\vert_{\orthiu}^{-1} \gradEh (u) \,\,=\,\, u - \Rh(u).
\end{align}
\end{lemma}

\begin{proof}
All properties follow by straightforward calculations. Only identity \eqref{id-u-rhu} requires verification. As the coercivity ensures that $\gradEh^{\hspace{1pt}\prime}(u)\vert_{\orthiu}^{-1}$ exists, we let $\phi_u := E^{\prime\prime}(u)\vert_{\orthiu}^{-1} \gradEh (u)  \in \orthiu$
to shorten notation. By definition, we have
\begin{align*}
\langle E^{\prime\prime}(u) \phi_u , v \rangle = \langle \gradEh (u) , v \rangle \qquad \mbox{for all } v\in \orthiu.
\end{align*}
Using the expression of $\gradEh (u)$, this yields
\begin{align*}
\langle E^{\prime\prime}(u) \phi_u , v \rangle &\,\,=\,\, \langle E^{\prime}(u) , \Rh(v) \rangle + \langle E^{\prime\prime}(u) u,v-  \Rh(v) \rangle \,\,=\,\,\langle E^{\prime\prime}(u) u,v-  \Rh(v) \rangle\\
 &\,\,=\,\, \langle E^{\prime\prime}(u) (u- \Rh(u)),v-  \Rh(v) \rangle \,\,=\,\, \langle E^{\prime\prime}(u) (u- \Rh(u)),v \rangle.
\end{align*}
Since this holds for arbitrary $v\in \orthiu$ and since both $\phi_u \in \orthiu$ and $u- \Rh(u)\in \orthiu$, we conclude $\phi_u=u- \Rh(u)$.
\end{proof}
Next, we quantify the continuity of $\gradEh^{\hspace{1pt}\prime}$.
\begin{lemma}
\label{lemma:continuity-secEh}
Let $u \in H^1(\Omega)$ denote a (local) minimizer of $E$ such that \ref{A4} holds.
If $\eta_h  \leq \eta_\star$ as per Lemma \ref{lemma:estimate-ritz-projection},
then for arbitrary $v,w,z \in \orthiu$ it holds
 \begin{equation}
\label{eq_cont_Eh_L4}
  \langle ( \gradEh^{\hspace{1pt}\prime}(u) - \gradEh^{\hspace{1pt}\prime}(z))  w , v \rangle 
 \lesssim \,  \left( \, \| u-z \|_{L^4(\Omega)} \, 
 +  \kappa^{d/4} \, \| u-z \|_{L^4(\Omega)}^2 \right) \| w \|_{L^4(\Omega)} \, \| v \|_{H^1_{\kappa}(\Omega)}
 \end{equation}
and
 \begin{equation}
\label{eq_cont_Eh_H1}
 \langle (\gradEh^{\hspace{1pt}\prime}(u) -  \gradEh^{\hspace{1pt}\prime}(z))  w , v \rangle 
 \lesssim \,  \left( \, \kappa^{d/4} \, \| u-z \|_{L^4(\Omega)} \, 
 +  \kappa^{d/2} \, \| u-z \|_{L^4(\Omega)}^2 \right) \| w \|_{H^1_{\kappa}(\Omega)} \, \| v \|_{H^1_{\kappa}(\Omega)}.
 \end{equation}
\end{lemma}

\begin{proof}
Let $v,w \in \orthiu$, according to \eqref{eq_identity_Ehpp} from Lemma \ref{lemma:properties_Eh_sec} it holds
\begin{align*}
\langle \gradEh^{\hspace{1pt}\prime}(z) w , v \rangle = \langle E^{\prime\prime}(z) w , \Rh(v) \rangle + \langle E^{\prime\prime}(u) w ,v-  \Rh(v) \rangle.
\end{align*}
Hence, for all $r \in \orthiu$, we have
\begin{eqnarray}
\label{Taylor-secEh}
\lefteqn{ \langle \gradEh^{\hspace{1pt}\prime\prime}(z) r , (w,v) \rangle \,\,\,=\,\,\, \langle E^{\prime\prime\prime}(z) r , ( w , \Rh(v)) \rangle } \\
\nonumber&=&  2 \, \int_{\Omega}  \Re\hspace{-1pt}( z \overline{r}) \, \Re\hspace{-1pt}( w \overline{\Rh(v)}) + \Re\hspace{-1pt}( w \overline{z}) \, \Re\hspace{-1pt}(r \overline{\Rh(v)}) + \Re\hspace{-1pt}( r \overline{w} ) \, \Re\hspace{-1pt}(z \overline{\Rh(v)}) \,\mbox{\normalfont d}x.
\end{eqnarray}
Note that $\langle\gradEh^{\hspace{1pt}\prime\prime}(z) r , (w,v) \rangle$ is symmetric w.r.t. to the second and third argument, i.e.
\begin{eqnarray*}
\langle \gradEh^{\hspace{1pt}\prime\prime}(z) r , (w,v) \rangle &=& \langle \gradEh^{\hspace{1pt}\prime\prime}(z) w , (r,v) \rangle.
\end{eqnarray*}
We shall exploit this in the following calculations. Also note that the third derivative $\gradEh^{\hspace{1pt}\prime\prime\prime}(z)$ is constant (i.e. independent of $z$) with
\begin{eqnarray*}
\langle \gradEh^{\hspace{1pt}\prime\prime\prime}(z) q , (r,w,v) \rangle &=& \langle E^{\prime\prime\prime}(q) r , ( w , \Rh(v)) \rangle.
\end{eqnarray*}
By Taylor-expanding $\gradEh^{\hspace{1pt}\prime}(z)$ in $u$ we therefore obtain the exact representation
\begin{eqnarray*}
\lefteqn{ \langle \gradEh^{\hspace{1pt}\prime}(z) w , v \rangle }\\ 
&=&  \langle \gradEh^{\hspace{1pt}\prime}(u) w , v \rangle +
 \langle \gradEh^{\hspace{1pt}\prime\prime}(u) (z-u),  (w , v ) \rangle
 + \tfrac{1}{2}  \langle \gradEh^{\hspace{1pt}\prime\prime\prime} (u) (z-u)  , (z-u , w , v) \rangle \\
 &=&  \langle \gradEh^{\hspace{1pt}\prime}(u) w , v \rangle +
 \langle E^{\prime\prime\prime}(u) (z-u),  (w , \Rh(v)) \rangle
 + \tfrac{1}{2}  \langle E^{\prime\prime\prime}(z-u) (z-u)  ,( w , \Rh(v) ) \rangle.
\end{eqnarray*}
With \eqref{Taylor-secEh} and $\| u \|_{L^{\infty}(\Omega)} \le 1$ we conclude
\begin{eqnarray*}
\lefteqn{ \langle \left( \gradEh^{\hspace{1pt}\prime}(u) -  \gradEh^{\hspace{1pt}\prime}(z) \right) w , v \rangle }\\ 
 &=&  \langle E^{\prime\prime\prime}(u) (u-z),  (w , \Rh(v)) \rangle
 - \tfrac{1}{2}  \langle E^{\prime\prime\prime}(z-u) (z-u)  ,( w , \Rh(v) ) \rangle \\
 &\le& 6  \| u -z  \|_{L^4(\Omega)} \| w  \|_{L^4(\Omega)} \| \Rh(v)  \|_{L^2(\Omega)}
 + 3 \| u -z  \|_{L^4(\Omega)}^2 \| w  \|_{L^4(\Omega)} \| \Rh(v)  \|_{L^4(\Omega)}.
\end{eqnarray*}
With the Gagliardo--Nirenberg estimate
$\| \Rh(v) \|_{L^4(\Omega)} \,\lesssim\, \kappa^{d/4} \| \Rh(v) \|_{H^1_{\kappa}(\Omega)}$ we further obtain
\begin{eqnarray*}
\lefteqn{ \langle \left( \gradEh^{\hspace{1pt}\prime}(u) - \gradEh^{\hspace{1pt}\prime}(z) \right) w , v \rangle }\\ 
 &\lesssim&  
 \| u -z  \|_{L^4(\Omega)}  \| w  \|_{L^4(\Omega)} \| \Rh(v)  \|_{L^2(\Omega)}
 +   
 \kappa^{d/4}
  \| u -z  \|_{L^4(\Omega)}^2   \| w \|_{L^4(\Omega)}   \| \Rh(v)  \|_{H^1_{\kappa}(\Omega)} \\
 &\le& \left( 
 \| u -z  \|_{L^4(\Omega)} + \kappa^{d/4}
  \| u -z  \|_{L^4(\Omega)}^2 \right) \| w  \|_{L^4(\Omega)}   \| \Rh(v)  \|_{H^1_{\kappa}(\Omega)}.
\end{eqnarray*}
The $H^1_{\kappa}$-stability of $\Rh$ for $\eta_{h} \leq \eta_\star$
according to Lemma \ref{lemma:estimate-ritz-projection}
then finishes the proof of \eqref{eq_cont_Eh_L4}. The second estimate
in \eqref{eq_cont_Eh_H1} then simply follows from \eqref{eq_cont_Eh_L4} by another
application of the Gagliardo--Nirenberg estimate on $w$.
\end{proof}
We are now ready to prove the existence of a unique zero of $\gradEh $
in a neighborhood of any local minimizer $u$ that satisfies \ref{A4}.

\begin{theorem}
\label{theorem:existence-of-local-discrete-minimizer}
Consider a local minimizer $u \in H^1(\Omega)$ of $E$ satisfying \ref{A4}.
Then, there exist constants $\tau_\star,c_\star \gtrsim 1$
such that for all $0 < \tau \leq \tau_\star$, if 
\begin{align}
\label{def-cstar}
\kappa^{d/2} \, \rho(\kappa) \, \|u-R_h(u)\|_{H^1_\kappa(\Omega)}
\,\,\le \,\,
c_\star \, \tau
\end{align}
then there exists a unique $u_h \in V_{h} \cap \orthiu$ with
\begin{align*}
\rho(\kappa) \, \| u_h - u \|_{H^1_{\kappa}(\Omega)} \,\,\le\,\,  \tau
\qquad
\mbox{and}
\qquad
\kappa^{d/4}\,\rho(\kappa) \, \| u_h - u \|_{L^4(\Omega)} \,\, \le \, \tau
\end{align*}
such that 
\begin{align*}
\langle E^{\prime}(u_h) , v_h \rangle = 0 \,\,\mbox{ for all } v_h \in V_{h}.
\end{align*}
Finally, it also holds
\begin{align}
\label{secEuh-coercive}
\langle E^{\prime\prime}(u_h) v_h , v_h \rangle \,\, \gtrsim \,\, \rho(\kappa)^{-1} \| v_h \|_{H^1_{\kappa}(\Omega)}^2
\qquad \mbox{for all } v_h \in V_{h} \cap \orthiu.
\end{align}
i.e., $u_h$ is a (local) minimizer of $E$ in $V_{h}$.
\end{theorem}

\begin{proof}
Let $0 < \tau \leq \tau_\star$ (with an upper bound on $\tau_\star \gtrsim 1$
to be specified later) and consider the following neighborhood of $u$:
\begin{align*}
\Ku \,:=\, \{
z \in \orthiu 
\, | \,
\rho(\kappa) \| u - z \|_{H^1_{\kappa}(\Omega)} \le \tau
\,\mbox{ and }\,
\kappa^{d/4} \rho(\kappa) \| u - z \|_{L^4(\Omega)} \le \tau \, \}.
\end{align*}
The set $\Ku$ is closed in $\orthiu$ because of the continuous embedding of $H^1(\Omega)$ into $L^4(\Omega)$
(here, $d\le 3$). We consider the Newton-type (fixed-point) function $F : \Ku \rightarrow \orthiu$ given by
\begin{align*}
F(z) := z - E^{\prime\prime}(u)_{\vert \orthiu}^{-1} \gradEh (z),
\end{align*}
which is motivated by the local quadratic convergence of Newton's method. 
For a sufficiently small $\tau_\star$ we want to prove that
$F$ has a unique fixed-point $u_h \in \Ku$ by using the Banach fixed-point
theorem. For that, we need to identify a resolution condition (depending on
$h$ and $\kappa$) such that $F$ is a contraction, i.e.
$\| F(w) - F(z)\|_{H^1_{\kappa}(\Omega)} \le L \|  w - z \|_{H^1_{\kappa}(\Omega)}$
on $\Ku$ for some $L<1$ and
$F(\,\Ku\,) \subset \Ku$. 

We start by deriving a condition for $\tau_\star$ such that the contraction property holds. Let $w,z \in \Ku$ be arbitrary. In this setting, we write 
$$
\gamma(s):=F(s\,w + (1-s)z) \qquad \mbox{and hence } F(w) - F(z) = \int_{0}^1 \gamma^{\prime}(s)\, \mbox{d}s.
$$
Using the fact that $F(z) = \gradEh^{\hspace{1pt}\prime}(u)_{\vert \orthiu}^{-1} (\gradEh^{\hspace{1pt}\prime}(u)z - \gradEh (z))$ we obtain
\begin{align*}
\gamma^{\prime}(s) &= \gradEh^{\hspace{1pt}\prime}(u)_{\vert \orthiu}^{-1} ( \gradEh^{\hspace{1pt}\prime} (u)(w-z) - \gradEh^{\hspace{1pt}\prime}(s\,w + (1-s)z)(w-z) ).
\end{align*}
Since the coercivity of $\gradEh^{\hspace{1pt}\prime}(u)$ with constant $\rho(\kappa)^{-1}$ implies for all $\mathcal{F} \in (\orthiu)^{\ast}$ that
\begin{align*}
\| \gradEh^{\hspace{1pt}\prime}(u)^{-1} \mathcal{F} \|_{H^1_{\kappa}(\Omega)} \le \rho(\kappa) \sup_{v \in \orthiu} \frac{ \langle \mathcal{F} , v \rangle }{\| v \|_{H^1_{\kappa}(\Omega)} },
\end{align*}
we conclude with Lemma \ref{lemma:continuity-secEh}
\begin{eqnarray*}
\lefteqn{ \| F(w) - F(z) \|_{H^1_{\kappa}(\Omega)} \,\,\le\,\, \rho(\kappa) \int_{0}^1 \sup_{v \in \orthiu} \frac{\langle ( \gradEh^{\hspace{1pt}\prime}(u) - \gradEh^{\hspace{1pt}\prime}(s\,w + (1-s)z)) \, (w-z)  , v \rangle }{\| v \|_{H^1_{\kappa}(\Omega)} } \,\mbox{d}s } \\
&\lesssim& \rho(\kappa) \, \sup_{s \in [0,1]} \left( \kappa^{d/4} \, \| u - s\,w - (1-s)z \|_{L^4(\Omega)} +  \kappa^{d/2}  \, \| u - s\,w - (1-s)z \|_{L^4(\Omega)}^2 \right) \, \|w-z  \|_{H^1_{\kappa}(\Omega)} \\
&\lesssim&
\rho(\kappa) \,
( \rho(\kappa)^{-1}\tau + \rho(\kappa)^{-2} \tau^2)
\, \|w-z  \|_{H^1_{\kappa}(\Omega)}
\,\,\,\lesssim\,\,\, (1+\tau^\star) \tau \, \|w-z\|_{H^1_\kappa(\Omega)}
\end{eqnarray*}
since $\rho(\kappa) \gtrsim 1$. We
conclude that for a sufficiently small $\tau_\star$ (independent of
$\kappa$ and $h$) it holds 
\begin{eqnarray*}
 \| F(w) - F(z) \|_{H^1_{\kappa}(\Omega)} 
 &\le&
\tfrac{1}{2} \|w-z  \|_{H^1_{\kappa}(\Omega)},
\end{eqnarray*}
i.e., $F$ is a contraction.

It remains to show $\rho(\kappa) \| u - F(z) \|_{H^1_\kappa(\Omega)}\le \tau$ and $\kappa^{d/4} \rho(\kappa) \| u - F(z) \|_{L^4(\Omega)}\le \tau$ if $z \in \Ku$. For this, consider $ u - F(z)$ for arbitrary $z \in \orthiu$. As before, we have
\begin{align*}
F(u) - F(z) \,\, = \,\, 
\int_{0}^1  \gradEh^{\hspace{1pt}\prime}(u)^{-1} ( \gradEh^{\hspace{1pt}\prime}(u)(u-z) - \gradEh^{\hspace{1pt}\prime}(s\,u + (1-s)z)(u-z) )\, \mbox{d}s.
\end{align*}
As \eqref{id-u-rhu} implies $u-F(u)= E^{\prime\prime}(u)\vert_{\orthiu}^{-1} \gradEh (u) = u - \Rh(u)$ we have altogether
\begin{eqnarray*}
u - F(z) &=& u - F(u) + F(u) - F(z) \\
&=& u - \Rh(u) \,+\, E^{\prime\prime}(u)\vert_{\orthiu}^{-1}  \int_{0}^1 \left( \gradEh^{\hspace{1pt}\prime}(u)  - \gradEh^{\hspace{1pt}\prime}(s\,u +(1-s)z) \right)(u-z) \, \mbox{d}s.
\end{eqnarray*}
Using the short-hand notation
\begin{equation*}
\varepsilon := \| u - \Rh(u) \|_{H^1_{\kappa}(\Omega)},
\end{equation*}
we obtain again with Lemma \ref{lemma:continuity-secEh} and $u  -(s\,u +(1-s)z) = (1-s)\,(u -z)$ that
\begin{eqnarray*}
\lefteqn{ \rho(\kappa) \| u - F(z) \|_{H^1_\kappa(\Omega)} } \\
&\le& \rho(\kappa) \| u - \Rh(u) \|_{H^1_\kappa(\Omega)}  +   \rho(\kappa)^{2} \sup_{v \in \orthiu} \frac{| 
 \int_{0}^1 \left( \langle \gradEh^{\hspace{1pt}\prime}(u)  - \gradEh^{\hspace{1pt}\prime}(s\,u +(1-s)z) \right)(u-z) , v \rangle \, \mbox{d}s |}{\| v \|_{H^1_{\kappa}(\Omega)}} \\
&\lesssim& \,
\rho(\kappa)\|u-R_h(u)\|_{H^1_\kappa(\Omega)} 
+ \rho(\kappa)^{2} \, \left( \kappa^{d/4} \, \| u -z \|_{L^4(\Omega)} \,
+ \kappa^{d/2} \, \| u - z \|_{L^4(\Omega)}^2 \,\right)  \, \|u-z  \|_{H^1_{\kappa}(\Omega)}  \\
&\lesssim& \,
\rho(\kappa)\varepsilon
+
(\tau + \rho(\kappa)^{-1} \tau^2) \tau
\,\,\,\lesssim\,\,\,
(c_\star + \tau_\star+\tau_\star^2) \tau
\end{eqnarray*}
where we used the condition involving $c_\star $ in \eqref{def-cstar}, i.e., that it holds $\varepsilon \leq c_\star \kappa^{-d/2} \rho(\kappa)^{-1} \tau \lesssim c_\star \tau$.
We thus obtain for all sufficiently small $c_\star$ and $\tau_\star$ that
\begin{equation*}
\rho(\kappa)\, \| u - F(z) \|_{H^1_\kappa(\Omega)} 
\,\,\le\,\, \tau.
\end{equation*}
It remains to show $\kappa^{d/4}\rho(\kappa)\| u - F(z) \|_{L^4(\Omega)}  \le \tau$.
For this, we use again the Gagliardo-Nirenberg inequality from \eqref{eq_gagliardo_nirenberg_L4}
together with the previous $H^1_\kappa$-estimate to obtain
\begin{eqnarray*}
\lefteqn{ \kappa^{d/4}\rho(\kappa) \| u - F(z) \|_{L^4(\Omega)} }\\
&\le&
\kappa^{d/4}\rho(\kappa) \, \| u - \Rh(u) \|_{L^4(\Omega)}
\\
&\enspace&+\,\,
\kappa^{d/4}\rho(\kappa) \, \| E^{\prime\prime}(u)\vert_{\orthiu}^{-1}  \int_{0}^1 \left( \gradEh^{\hspace{1pt}\prime}(u)  - \gradEh^{\hspace{1pt}\prime}(s\,u +(1-s)z) \right)(u-z) \, \mbox{d}s \|_{L^4(\Omega)} \\
&\lesssim&
\kappa^{d/2} \rho(\kappa) \, \varepsilon 
+ \kappa^{d/2} \,\rho(\kappa)^{2} \, 
\int_{0}^1  \sup_{v \in \orthiu} \frac{\langle \left( \gradEh^{\hspace{1pt}\prime}(u)  - \gradEh^{\hspace{1pt}\prime}(s\,u +(1-s)z) \right)(u-z) , v \rangle \, \mbox{d}s }{ \| v\|_{H^1_{\kappa}(\Omega)}}  \, \mbox{d}s
\\
&\overset{\eqref{eq_cont_Eh_L4}}{\lesssim}& 
 \kappa^{d/2} \rho(\kappa)\, \varepsilon
+ \rho(\kappa)^{2} \, \kappa^{d/2} \, 
\left( \, \| u -z \|_{L^4(\Omega)}^2
+ \kappa^{d/4} \, \| u - z \|_{L^4(\Omega)}^3 \right) \\
&\overset{\eqref{def-cstar}}{\lesssim}&
c_\star \tau + \tau^2 + \rho(\kappa)^{-1} \tau^3
\,\,\,\lesssim\,\,\,
(c_\star + \tau_\star + \tau_\star^2) \,\tau.
\end{eqnarray*}
We then have for all sufficiently small $c_\star$ and $\tau_{\star}$ that 
\begin{eqnarray*}
 \kappa^{d/4}\rho(\kappa) \, \| u - F(z) \|_{L^4(\Omega)} \,\, \le \,\, \tau.
 \end{eqnarray*}
Consequently, $F$ is a contraction on $\Ku$ with $F(\,\Ku\,) \subset \Ku$ and we have existence of a unique function $u_h \in \orthiu$ such that $\gradEh (u_h)=0$. By Lemma \ref{lemma:characterization:Ehprime} we conclude that $u_h \in  V_{h} \cap \orthiu$ fulfills  
\begin{align*}
 \langle E^{\prime}(u_h),v_h \rangle = 0 \quad \mbox{for all } v_h \in  V_{h} \cap \orthiu. 
\end{align*}
Before we can verify that the test functions can be arbitrary in $V_{h}$ (and do not have to belong to $\orthiu$), we first verify the coercivity of $\langle E^{\prime\prime}(u_h)\,\cdot,\cdot \rangle$, i.e., property \eqref{secEuh-coercive}. To get the desired estimate, we use one more time Lemma \ref{lemma:continuity-secEh} to obtain
 \begin{eqnarray*}
  | \langle ( \gradEh^{\hspace{1pt}\prime}(u) -  \gradEh^{\hspace{1pt}\prime}(u_h))  v_h , v_h \rangle | 
&\lesssim& (  \kappa^{d/4}  \| u - u_h \|_{L^4(\Omega)} + \kappa^{d/2} \| u - u_h \|_{L^4(\Omega)}^2 ) \| v_h \|_{H^1_{\kappa}(\Omega)}^2  \\
 &\lesssim&
\tau \, \rho(\kappa)^{-1} \, \| v_h \|_{H^1_{\kappa}(\Omega)}^2.
 \hspace{170pt}
 \end{eqnarray*}
 
Since $\langle \gradEh^{\hspace{1pt}\prime}(w)  v_h , v_h \rangle=\langle E^{\prime\prime}(w)  v_h , v_h \rangle$ for all $v_h \in V_{h} \cap \orthiu$ and arbitrary $w\in \orthiu$ (cf. Lemma \ref{lemma:properties_Eh_sec}), we
obtain for sufficiently small $\tau$  that
\begin{eqnarray*}
\langle E^{\prime\prime}(u_h) v_h , v_h \rangle &=& \langle E^{\prime\prime}(u) v_h , v_h \rangle + \langle (E^{\prime\prime}(u_h)-E^{\prime\prime}(u)) v_h , v_h \rangle \\
&\ge& \rho(\kappa)^{-1} \|v_h \|_{H^1_{\kappa}(\Omega)}^2 
- |\langle (E^{\prime\prime}(u_h)-E^{\prime\prime}(u)) v_h , v_h \rangle | \\
&\ge&   \rho(\kappa)^{-1}(1-C \tau)  \| v_h \|_{H^1_{\kappa}(\Omega)}^2 
\,\, \gtrsim\,\,   \rho(\kappa)^{-1}  \| v_h \|_{H^1_{\kappa}(\Omega)}^2.
\end{eqnarray*}
It remains to verify that it holds $E^{\prime}(u_h)= 0$ on $V_{h}$. We already know that $E^{\prime}(u_h)= 0$ on $V_{h} \cap \orthiu$ and (by the coercivity) that $E^{\prime\prime}(u_h) $ has a positive spectrum on $V_{h} \cap \orthiu$. Consequently, $u_h$ is the unique local minimizer of $E$ in $V_{h} \cap \orthiu$ with $\rho(\kappa) \| u - u_h \|_{H^1_{\kappa}(\Omega)} \le \tau$ and $\kappa^{d/4}\rho(\kappa)\| u - u_h \|_{L^4(\Omega)} \le \tau$. However, the constraint $v_h \in \orthiu$ cannot change the (locally) minimal energy value, because for any $v_h \in V_{h}$ there exists some $\omega \in [-\pi,\pi)$ such that $\exp(\ci \omega) v_h \in \orthiu$ and $E(\exp(\ci \omega)v_h)=E(v_h)$. For that reason, $u_h$ must be a local minimizer both in $V_{h} \cap \orthiu$ and $V_{h}$. By the first order condition for local minimizers we therefore have $\langle E^{\prime}(u_h) , v_h \rangle = 0$ for all $v_h \in V_{h}$.
\end{proof}

\subsection{Proof of Theorem \ref{theorem:main-result}}
\label{subsection:proof:main-result}

We are now ready to combine the results of the previous subsections to prove Theorem \ref{theorem:main-result}. We first establish the main result in a more abstract form.

\begin{theorem}
\label{theorem:abstract-result}
Consider a (local) minimizer $u \in H^1(\Omega)$ of $E$ satisfying \ref{A4}. 
There exist constants $r_\star \gtrsim 1$,\,\, $\eta_\star \gtrsim 1$\,\,
and \, $M_\star \lesssim 1$\,\, such that if $\eta_h \leq \eta_\star$ and
\begin{equation}
\label{condition-r-star}
\kappa^{d/2} \rho(\kappa) \, \min_{v_h \in V_h} \|u-v_h\|_{H^1_\kappa(\Omega)} \,\,\,\leq \,\,\, r_\star,
\end{equation}
then there exists a unique discrete minimizer $u_h \in V_h \cap \orthiu$ in the neighborhood
\begin{equation*}
K
:=
\left \{
v \in H^1(\Omega)
\,\; | \;\,
\rho(\kappa) \|u-v\|_{H^1_\kappa(\Omega)}
\leq
M_\star
\,\,\text{ and }\,\,
\kappa^{d/4} \rho(\kappa) \|u-v\|_{L^4(\Omega)}
\leq
M_\star
\right \},
\end{equation*}
and we have
\begin{equation}
\label{eq_quasi_optimality_H1}
\|u-u_h\|_{H^1_\kappa(\Omega)}
\,\,\,\lesssim\,\,\,
\min_{v_h \in V_h} \|u-v_h\|_{H^1_\kappa(\Omega)}.
\end{equation}
In addition, the $L^2$-error estimate holds true
\begin{equation}
\label{eq_error_estimate_L2_H1}
\|u-u_h\|_{L^2(\Omega)}
\,\,\,\lesssim\,\,\,
\eta_h \min_{v_h \in V_h} \|u-v_h\|_{H^1_\kappa(\Omega)}
 \,+ \,
\rho(\kappa)  \,\, \kappa^{d/2} \left (\min_{v_h \in V_h} \|u-v_h\|_{H^1_\kappa(\Omega)}\right )^2.
\end{equation}
\end{theorem}

\begin{proof}
We apply Theorem \ref{theorem:existence-of-local-discrete-minimizer}
with a small enough $\tau_\dagger \leq \tau_\star$ with $\tau_\dagger \gtrsim 1$ to be fixed.
Then, under the assumption that
\begin{equation}
\label{tmp_assumption_main_result}
\kappa^{d/2}  \, \rho(\kappa) \, \|u-R_h(u)\|_{H^1_\kappa(\Omega)} \, \le \, c_\star \tau_\dagger,
\end{equation}
there exists a unique local minimizer  $u_h \in V_{h} \cap \orthiu$ with $E^{\prime}(u_h) \vert_{V_{h}}=0$ and with
\begin{align}
\label{proof:main-result-step-1-eqn}
\kappa^{d/4} \rho(\kappa) \, \| u_h - u \|_{L^4(\Omega)} \,\,\le\,\, \tau_\dagger.
\end{align}
We split
\begin{align*}
\| u_h - u \|_{H^1_{\kappa}(\Omega)} \,\,\le\,\, \| u - \Rh(u) \|_{H^1_{\kappa}(\Omega)} + \| u_h - \Rh(u) \|_{H^1_{\kappa}(\Omega)},
\end{align*}
and work on the second term. Conclusion \ref{conclusion:estimate-uh-Rhu}
implies that for some constant $C>0$ it holds 
\begin{eqnarray*}
\| u_h - \Rh(u) \|_{H^1_{\kappa}(\Omega)} 
&\le& C\, 
\rho(\kappa) \left( \kappa^{d/4} \| u-u_h\|_{L^4(\Omega)} +  \kappa^{d/2}  \| u-u_h\|_{L^4(\Omega)}^2 \right)  \,  \| u-u_h\|_{H^1_{\kappa}(\Omega)} \\
&\overset{\eqref{proof:main-result-step-1-eqn}}{\le}& C\, 
\rho(\kappa) \left( \tau_\dagger \, \kappa^{d/4} \ \kappa^{-d/4} \rho(\kappa)^{-1} + \tau_\dagger^2\,  \kappa^{d/2}  \kappa^{-d/2} \ \rho(\kappa)^{-2}  \right)  \,  \| u-u_h\|_{H^1_{\kappa}(\Omega)} \\
&\le& 2\, C\, \tau_\dagger \, \| u-u_h\|_{H^1_{\kappa}(\Omega)}
\,\,\, \le \,\,\, \tfrac{1}{2} \|u-u_h\|_{H^1_\kappa(\Omega)} 
\end{eqnarray*}
for $\tau_\dagger = \min\{ \tfrac{1}{4C} , 1 , \tau_\star \}$.
Combining the previous two estimates yields 
$\| u_h - u \|_{H^1_{\kappa}(\Omega)} \,\,\le\,\, 2 \| u - \Rh(u) \|_{H^1_{\kappa}(\Omega)}$ whenever \eqref{tmp_assumption_main_result} is met for $\tau_\dagger$ as defined. Hence, it remains to verify that \eqref{tmp_assumption_main_result} is indeed fulfilled for a suitable $r_\star$.  Here we use that under the assumption
$\eta_h \leq \eta_\star$, we have proved in Lemma \ref{lemma:estimate-ritz-projection}
that there is a constant $C_\star>0$ such that
\begin{equation*}
\|u-R_h(u)\|_{H^1_\kappa(\Omega)}
\,\,\,\leq\,\,\,
C_\star \min_{v_h \in V_h} \|u-v_h\|_{H^1_\kappa(\Omega)}.
\end{equation*}
Consequently, \eqref{tmp_assumption_main_result} is satisfied if we select $r_{\star}$
in \eqref{condition-r-star} as $r_\star = c_\star \tau_\dagger/C_\star$.
In conclusion, a combination of the estimates yields \eqref{eq_quasi_optimality_H1}.
We can set $M_\star := \tau_\dagger$.

For the $L^2$-error estimate, to shorten notation, we introduce a minimizer $v_h^\star \in V_h$
such that
\begin{equation*}
\|u-v_h^\star\|_{H^1_\kappa(\Omega)}
=
\inf_{v_h \in V_h} \|u-v_h\|_{H^1_\kappa(\Omega)}.
\end{equation*}
Then, we first use again  Lemma \ref{lemma:estimate-ritz-projection} for $\eta_h \leq \eta_\star$ to obtain
\begin{equation*}
\| u - \Rh(u) \|_{L^2(\Omega)}
\,\,\,\lesssim\,\,\,
\eta_h \, \|u-v_h^\star\|_{H^1_\kappa(\Omega)}.
\end{equation*}
In total we obtain with the suboptimal estimate
$ \| u_h - \Rh(u) \|_{L^2(\Omega)} \lesssim \| u_h - \Rh(u) \|_{H^1_{\kappa}(\Omega)} $
and Lemma \ref{lemma:H1-est-Rh} that
\begin{eqnarray*}
\lefteqn{ \| u-u_h \|_{L^2(\Omega)} 
\,\,\,\lesssim\,\,\,
\eta_h \|u-v_h^\star\|_{H^1_\kappa(\Omega)}
+
\|u_h-R_h(u)\|_{H^1_\kappa(\Omega)} }
\\
&\lesssim&
\eta_h \|u-v_h^\star\|_{H^1_\kappa(\Omega)}
+
\eta_h \|R_h(u)-u_h\|_{L^2(\Omega)}
+
\rho(\kappa)\left (\|u-u_h\|_{L^4(\Omega)}^2
+
\kappa^{d/4}\|u-u_h\|_{L^4(\Omega)}^3
\right ).
\end{eqnarray*}
At that point, we split
\begin{equation*}
\|R_h(u)-u_h\|_{L^2(\Omega)}
\leq
\|u-u_h\|_{L^2(\Omega)}
+
\|u-R_h(u)\|_{L^2(\Omega)}
\lesssim
\|u-u_h\|_{L^2(\Omega)}
+
\|u-v_h^\star\|_{H^1_{\kappa}(\Omega)},
\end{equation*}
and for $\eta_\star$ sufficiently small, we have
\begin{eqnarray*}
\| u-u_h \|_{L^2(\Omega)} 
&\lesssim&
\eta_h \|u-v_h^\star\|_{H^1_\kappa(\Omega)}
+
\rho(\kappa)\left (\|u-u_h\|_{L^4(\Omega)}^2
+
\kappa^{d/4}\|u-u_h\|_{L^4(\Omega)}^3
\right ).
\end{eqnarray*}
For the last term, we employ the Gagliardo-Nirenberg estimate from \eqref{eq_gagliardo_nirenberg_L4}
which gives
\begin{eqnarray*}
\|u-u_h\|_{L^4(\Omega)}^2 + \kappa^{d/4}\|u-u_h\|_{L^4(\Omega)}^3
&\lesssim&
\kappa^{d/2}\|u-u_h\|_{H^1_\kappa(\Omega)}^2
+
\kappa^d \|u-u_h\|_{H^1_\kappa(\Omega)}^3
\\
&\lesssim&
\left (
1+\kappa^{d/2}\|u-u_h\|_{H^1_\kappa(\Omega)}
\right )
\kappa^{d/2}\|u-u_h\|_{H^1_\kappa(\Omega)}^2
\\
&\lesssim&
\left (
1+\kappa^{d/2}\|u-v_h^\star\|_{H^1_\kappa(\Omega)}
\right )
\kappa^{d/2}\|u-v_h^\star\|_{H^1_\kappa(\Omega)}^2
\\
&\lesssim&
\kappa^{d/2}\|u-v_h^\star\|_{H^1_\kappa(\Omega)}^2
\end{eqnarray*}
due to the resolution condition in the assumptions.
\end{proof} 

Now, Theorem \ref{theorem:main-result} simply follows from applying the explicit expressions
we previously obtained from $\eta_h$ and the best-approximation error.

\begin{proof}[Proof of Theorem \ref{theorem:main-result}]
From \eqref{est-eta-h-p} in Theorem \ref{theorem:eta-hp-estimate} we recall that
\begin{eqnarray}
\label{tmp_eta}
\eta_h &\lesssim& \kappa^\varepsilon \left (h\kappa + \rho(\kappa)(k\kappa)^p\right )
\end{eqnarray}
and
\begin{eqnarray}
\label{tmp_min}
\min_{v_h \in V_h} \|u-v_h\|_{H^1_\kappa(\Omega)} &\lesssim& (h\kappa)^p.
\end{eqnarray}
We now observe that for $K_\star$ small enough, the condition that
\begin{eqnarray*}
\kappa^{d/2} \rho(\kappa) (h\kappa)^p &\le& K_\star
\end{eqnarray*}
ensures that $\eta_h \leq \eta_\star$ as well as
\begin{eqnarray*}
\kappa^{d/2} \rho(\kappa) \min_{v_h \in V_h} \|u-v_h\|_{H^1_\kappa(\Omega)}
&\leq& r_\star
\end{eqnarray*}
for the constants $\eta_\star,r_\star \gtrsim 1$ of Theorem \ref{theorem:abstract-result}.
Here, we used the fact that $\kappa^\varepsilon \leq \kappa^{d/2}$.

We are therefore in the setting of Theorem \ref{theorem:abstract-result}, and the error estimates
\eqref{final-H1-estimate} and \eqref{final-L2-estimate} explicitly expressed
in terms of $h$ and $\kappa$ simply follow by plugging \eqref{tmp_eta} and
\eqref{tmp_min} in their abstract versions \eqref{eq_quasi_optimality_H1} and
\eqref{eq_error_estimate_L2_H1}.
\end{proof} 

\section{A nonlinear conjugate gradient method}
\label{section:nonlinear-cg-method}
In order to find minimizers of the energy $E$, we apply a nonlinear conjugate gradient (CG) method. 
We start from the general iterative update
\begin{align}
\label{general-iteration}
u^{n+1} = u^n + \tau_n d^n,
\end{align}
where $u^n \in H^1(\Omega)$ is the current approximation, $\tau_n>0$ is the step size (computed by a line search to guarantee sufficient energy decrease),
and $d^n \in H^1(\Omega)$ is a descent direction. 
In nonlinear CG methods, $d^n$ is recursively defined by
\begin{align}
\label{general-descent-direction}
d^n = - \nabla_X E(u^n) + \beta^n d^{n-1},
\end{align}
where $\beta^n$ is a dissipation parameter and $\nabla_X E$ denotes the Sobolev gradient
of $E$ with respect to a Hilbert space $X \subset H^1(\Omega)$
equipped with inner product $(\cdot,\cdot)_X$. Formally, for any $z \in X$, the Sobolev gradient $\nabla_X E(z) \in X$ is defined by
\begin{align*}
( \nabla_X E(z) , w )_X = \langle  E^{\prime}(z) , w \rangle 
\qquad \text{for all } w\in X.
\end{align*} 
Sobolev gradients thus represent the direction of steepest descent with respect to the metric $\|\cdot\|_X$. 
For an introduction we refer to Neuberger \cite{Neu97}, and exemplarily to 
\cite{AHYY26,APS22,CLLZ24,DaK10,DaP17,GaPe01,HeP20,NoPr23} for applications.
\medskip

In our experiments, we employ an \textit{energy-adaptive} inner product that depends on the current iterate $z$.
The goal is to choose $(\cdot,\cdot)_{X,z}$ such that the Sobolev gradient becomes as close as possible to the identity mapping, resulting in steeper descent directions.
Specifically, for $v,w \in H^1(\Omega)$ we set
\begin{align}
\label{def-au-vw}
(v,w)_{X,z} := ( \tfrac{\ci}{\kappa} \nabla v + \bfA v , \tfrac{\ci}{\kappa} \nabla w + \bfA w )_{L^2(\Omega)} 
+ \bigl((|z|^2 + |\bfA|^2) v , w \bigr)_{L^2(\Omega)}.
\end{align}
provided that $(z,\bfA)$ are such that $(\cdot,\cdot)_{X,z}$ is an inner product, the corresponding Sobolev gradient $\nabla_{X,z} E(z) \in H^1(\Omega)$ satisfies
\begin{align*}
( \nabla_{X,z} E(z) , w )_{X,z} = \langle E^{\prime}(z) , w \rangle 
\qquad\text{for all } w\in H^1(\Omega).
\end{align*}
Since 
$\langle E^{\prime}(z) , w \rangle = (z,w)_{X,z} -  \bigl((1 + |\bfA|^2) z , w \bigr)_{L^2(\Omega)}$,
it follows that 
$\delta_z := z - \nabla_{X,z} E(z) \in H^1(\Omega)$ solves the elliptic problem
\begin{align*}
( \delta_z , v )_{X,z} =  \bigl((1 + |\bfA|^2) z , v \bigr)_{L^2(\Omega)} 
\qquad \text{for all } v\in H^1(\Omega).
\end{align*}
Hence, the Sobolev gradient in the $(\cdot,\cdot)_{X,z}$-metric is given by
$\nabla_{X,z} E(z) = z - \delta_z$. Sometimes $\nabla_{X,z} E(z)$ is also called a \quotes{preconditioned gradient}. 
Inserting this expression into \eqref{general-iteration}--\eqref{general-descent-direction} and using the 
Polak--Ribi\'ere formula \cite{polak1969} for $\beta^n$ yields the following algorithm.

\begin{definition}[Energy-adaptive nonlinear CG method]\label{def-nonlinear-cg}
Given $u^0 \in H^1(\Omega)$ with $u^0 \neq 0$ and $d^0 := 0$, perform for $n=0,1,2,\dots$ the following steps:
\begin{enumerate}
\item \textbf{Compute $\delta_{u^n}$:} With $( \cdot , \cdot )_{X,u^n}$ as in \eqref{def-au-vw}, solve for $\delta_{u^n} \in H^1(\Omega)$ such that
\[
( \delta_{u^n} , v )_{X,u^n} = \bigl((1 + |\bfA|^2) u^n , v \bigr)_{L^2(\Omega)} 
\qquad \text{for all } v \in H^1(\Omega).
\]

\item \textbf{Update descent direction:} 
\[
d^n := \delta_{u^n} - u^n + \beta^{n} d^{n-1}, \qquad
\beta^{n} := \max \!\left\{ 0,\ 
\frac{ ( \delta_{u^n} - u^n ,\, 
\delta_{u^n} - u^n - (\delta_{u^{n-1}} - u^{n-1}) )_{X,u^{n}} }
{ ( \delta_{u^{n-1}} - u^{n-1} ,\, 
\delta_{u^{n-1}} - u^{n-1} )_{X,u^{n-1}} }
\right\}.
\]

\item \textbf{Line search:} Determine
\[
\tau_n := \underset{\tau > 0}{\arg\min}\; E(u^n + \tau d^n).
\]
Note that $E(u^n + \tau d^n)$ is a quartic polynomial in $\tau$ so that an exact line search can be performed efficiently.

\item \textbf{Update iterate:} Set
\[
u^{n+1} := u^n + \tau_n d^n.
\]
\end{enumerate}
\end{definition}
\noindent
The dominant computational cost per iteration is solving the linear elliptic problem for $\delta_{u^n}$.
The line search is inexpensive since the coefficients of the quartic polynomial can be obtained by matrix--vector products.

Suitable generic starting values for the gradient method (that we also used in our experiments) are for example $u^0(x)= \tfrac{4}{5} + \ci \hspace{1pt}\tfrac{3}{5}$ (hence $|u^0|=1$), $u^0(x_1,x_2)=\ci + x_1 - \frac{1}{2}$ or  $u^0(x_1,x_2) = (x_1+\ci x_2) \exp( - \tfrac{|x|^2}{2})$.

\section{Numerical experiments}
\label{section:numerical-experiments}
In this section we present numerical experiments to support our theoretical results for local minimizers that are presumably also global minimizers. For that we consider a setting previously studied in \cite{BDH25,DoeHe24} where the computational domain is given by $\Omega = [0,1]^2 \subset \R^2$ and the  divergence-free magnetic potential (with vanishing normal trace) by
$$
\bfA(x,y):= \sqrt{2} \begin{pmatrix} \sin (\pi x) \cos (\pi y)\\ -\cos (\pi x) \sin (\pi y) \end{pmatrix}.
$$
The Ginzburg-Landau parameter is kept variable, where we consider the cases $\kappa = 8,16,24,32,40$. Note that $\kappa=8,24$ were previously studied in \cite{DoeHe24} (for P1-FEM) and $\kappa=8,16,32$ were studied in \cite{BDH25} (for P1-FEM and LOD spaces). The case $\kappa=40$ in this setting is not yet documented in the literature. Our reference minimizers $u_{\text{\tiny ref}}$ for the different $\kappa$ values were computed with the nonlinear conjugate gradient method described in Definition \ref{def-nonlinear-cg}, together with P2 finite elements on a uniform triangular mesh with mesh size $h_{\text{\tiny ref}} = 2^{-10}$.

As a tolerance for the nonlinear conjugate gradient method we used that the difference in two
successive energy approximations should be smaller than $10^{-15}$. Furthermore, after a converged
state $\widetilde u$ was found, we computed $E^{\prime\prime}(\widetilde u)$ and checked
that the spectrum was indeed non-negative (which we recall as the necessary condition for
minimizers). If there was a negative eigenvalue, we took one step (of optimal step length) in
the direction of the corresponding eigenfunction, noting that it must be a descent direction.
From the obtained result, we then restarted the conjugate gradient method. This avoids that the
iterative solver gets stuck in a saddle point. It is worth mentioning that such a step was indeed
necessary for most $\kappa$ values.

The densities $|u_{\text{\tiny ref}}|$ of the corresponding minimizers that we found are depicted in Figure \ref{fig:vortices_reference} and the corresponding energy values in Table \ref{tab:energy_reference_and_ev}. In all our test cases, we found lower energy states than the ones documented in \cite{BDH25,DoeHe24}. We also note that our results are consistent with the findings of \cite{BDH25,DoeHe24} as we could identify the states from the aforementioned references (for $\kappa=8,16,24,32$) as local minimizers of $E$. Hence, there is no contradiction to previous results. For $\kappa=40$ we do not have any reference comparison. We verified that the computed states are indeed (local) minimizers, by checking the spectrum of $E^{\prime\prime}(u_{\text{\tiny ref}})$ and verifying that the smallest eigenvalue is zero (up to numerical precision) and all remaining eigenvalues are positive. For brevity we will write in the following $u$ instead of $u_{\mbox{\tiny ref}}$.

\begin{figure}[h]
\centering
\begin{minipage}{0.3\textwidth}
\includegraphics[scale=0.25]{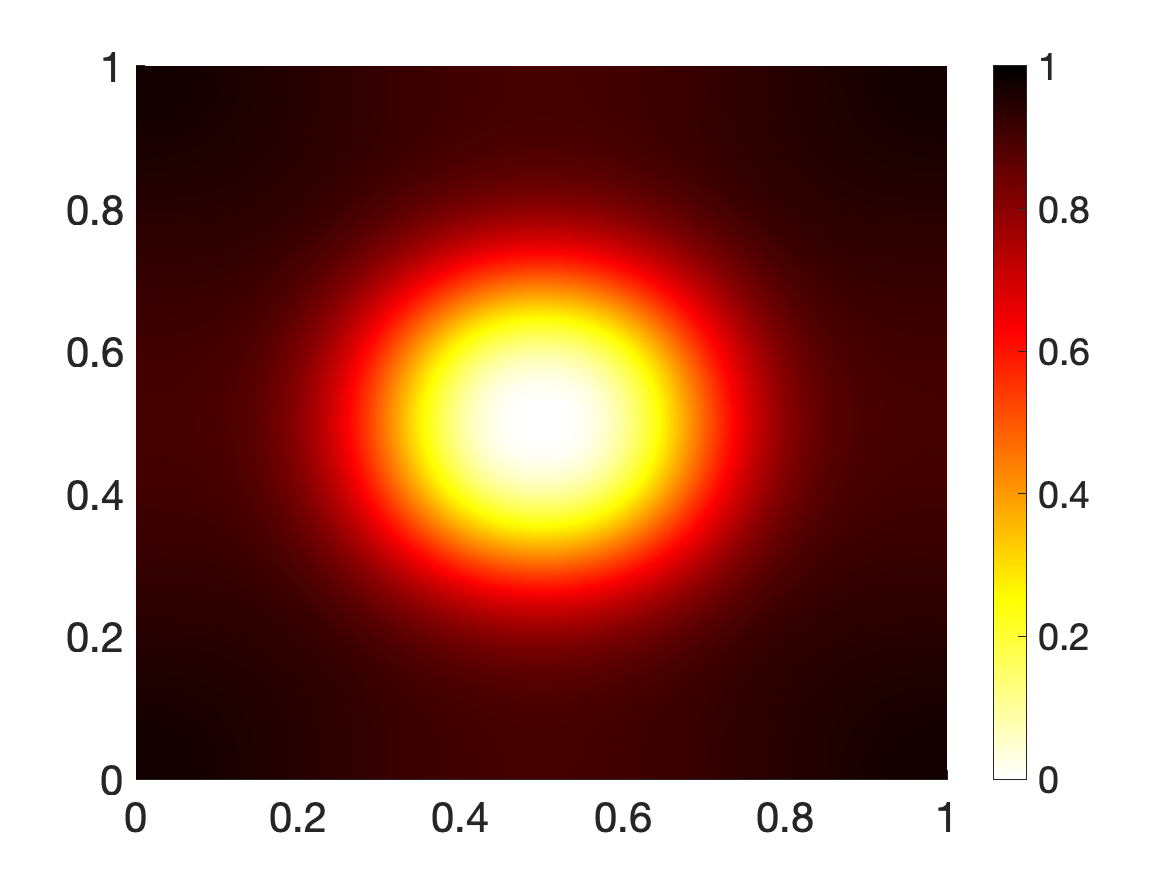}
\end{minipage}
\begin{minipage}{0.3\textwidth}
\includegraphics[scale=0.25]{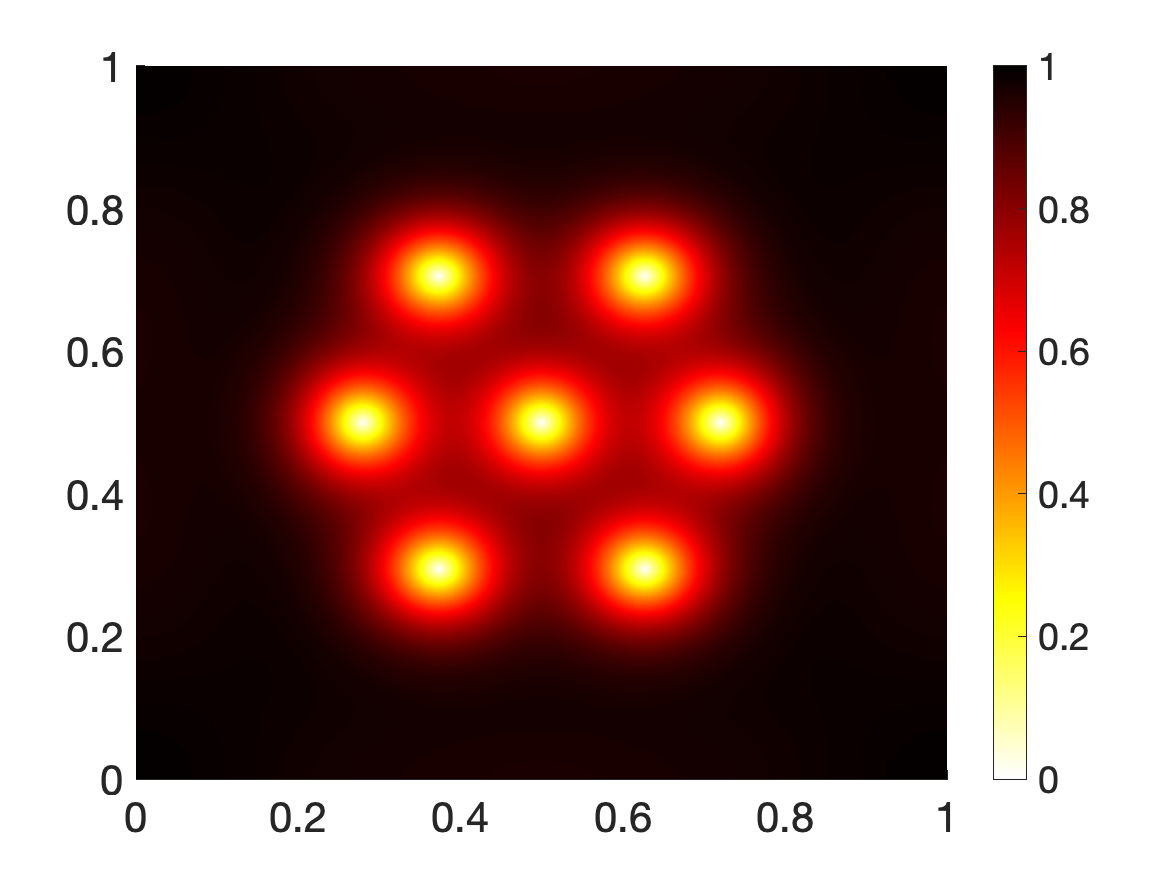}
\end{minipage}
\begin{minipage}{0.3\textwidth}
\includegraphics[scale=0.25]{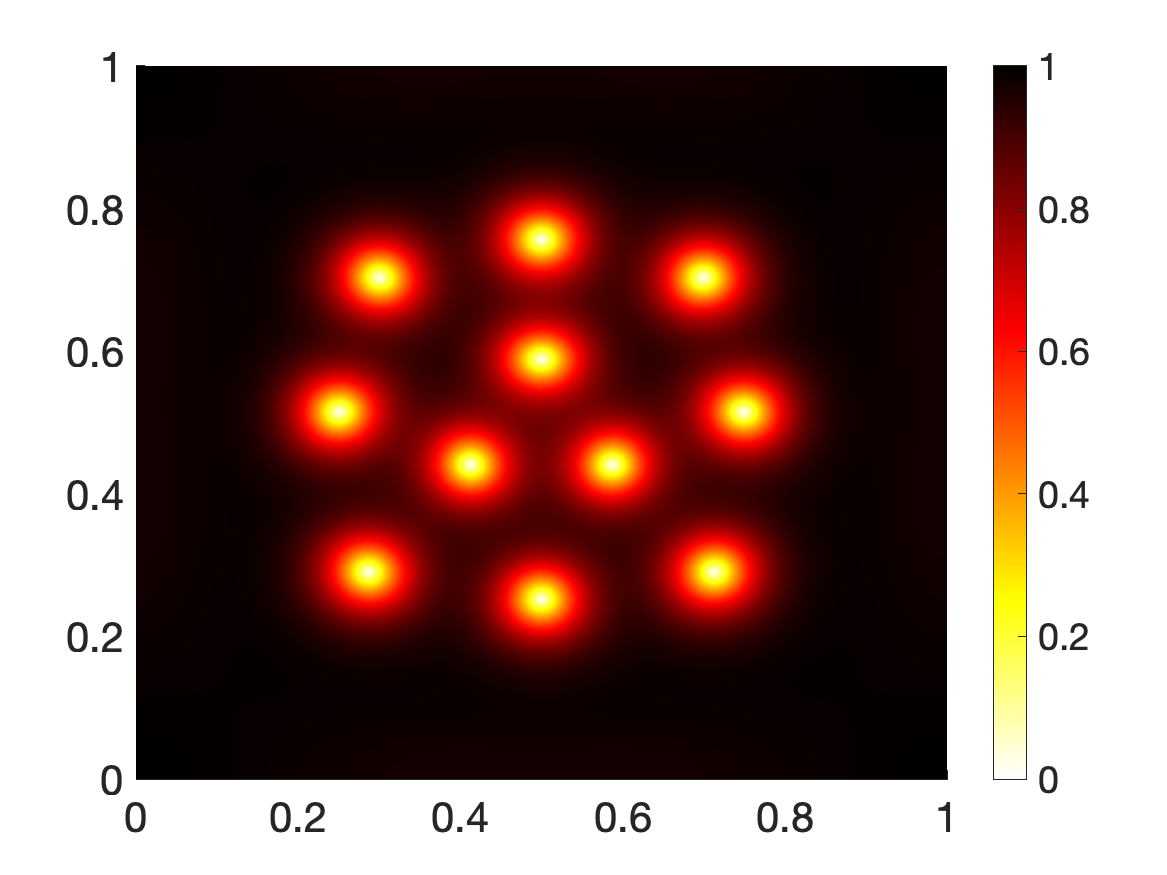}
\end{minipage} \\
\begin{minipage}{0.3\textwidth}
\includegraphics[scale=0.25]{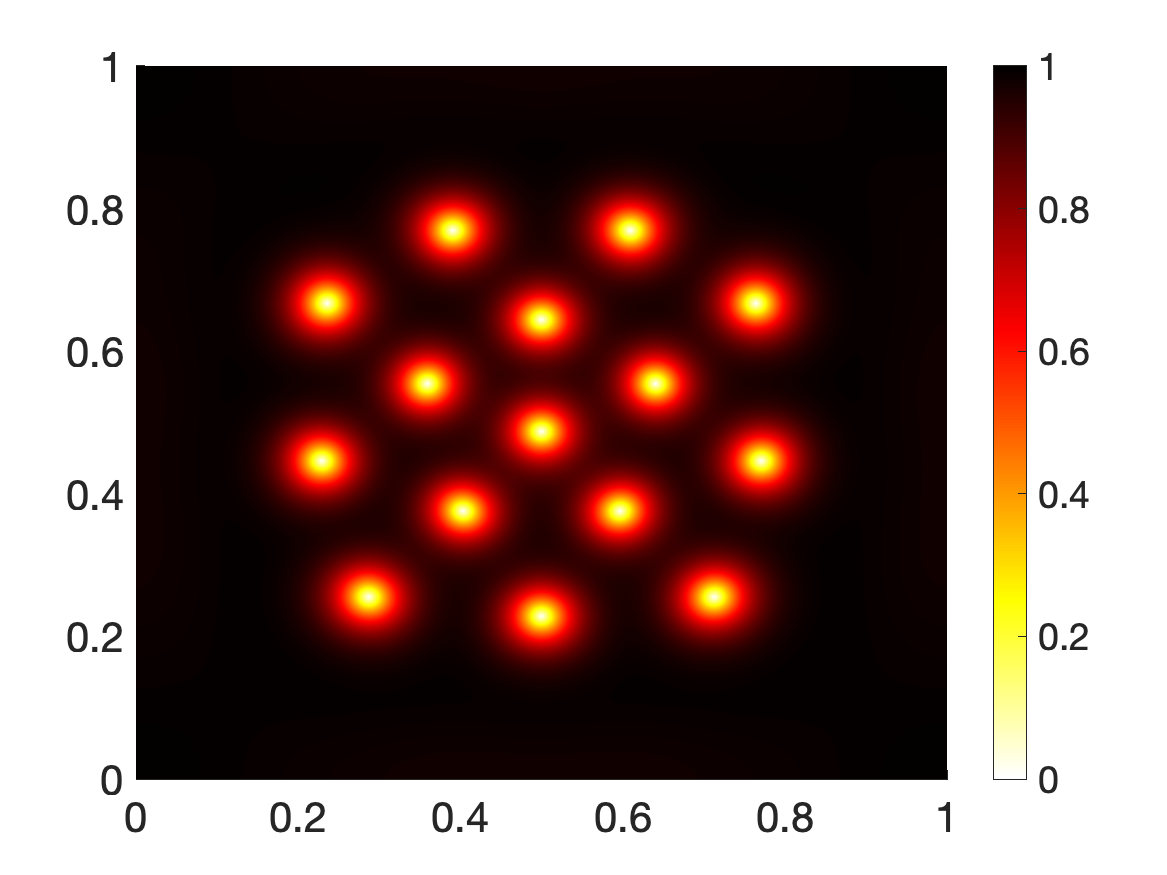}
\end{minipage}
\begin{minipage}{0.3\textwidth}
\includegraphics[scale=0.25]{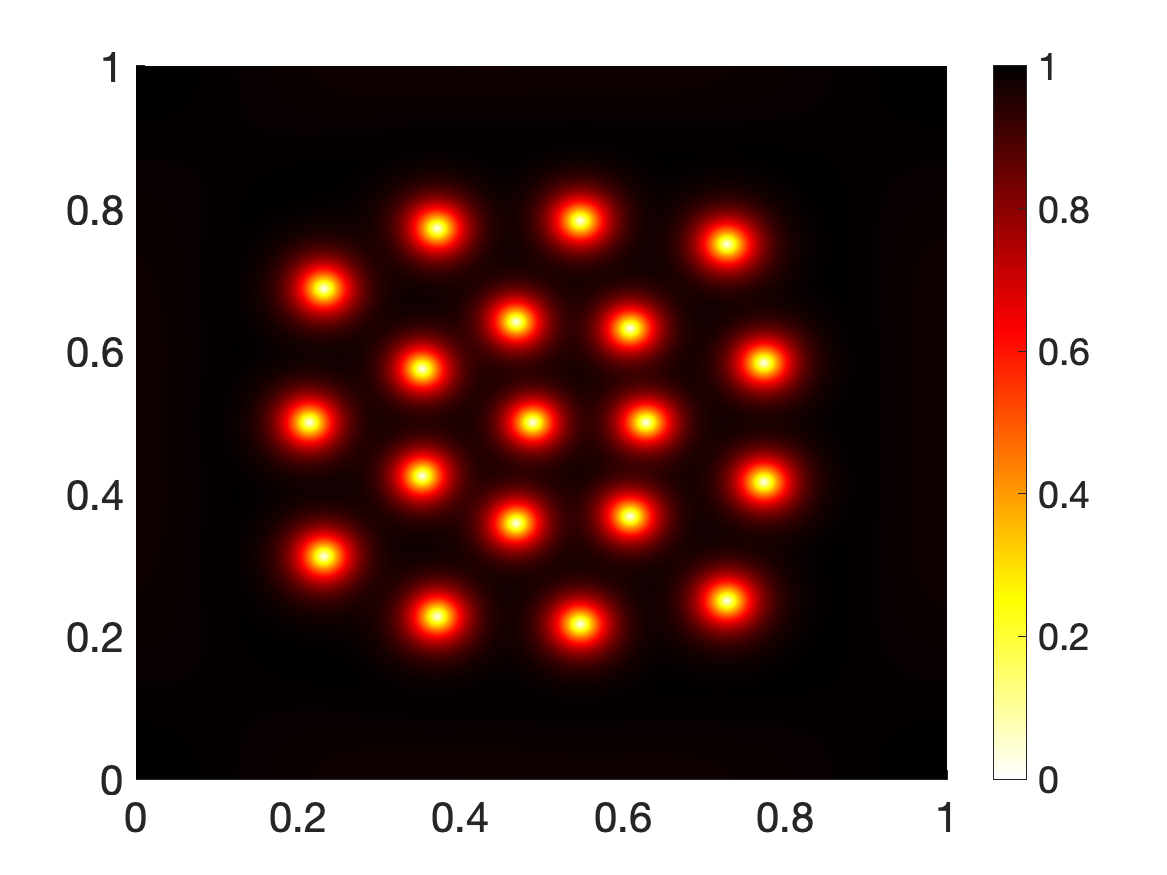}
\end{minipage}
	\caption{Vortex patterns of the reference solution $|u_{{\text{\tiny ref}}}|$ for $\kappa=8,16,24,32,40$ (from the upper left to the lower right picture).}
	\label{fig:vortices_reference}
\end{figure}

\renewcommand*{\arraystretch}{1.2}
\begin{table}[h]
	\centering
	\begin{tabular}[h]{ |c|c|c|c|c|c| }
	\hline
		$\kappa$ & 8 & 16 & 24 & 32 & 40 \\
		\hlinewd{1.2pt}
		$E(u_{\text{\tiny ref}})$
		& 1.164614e-01 
		& 8.314057e-02 
		& 6.599864e-02 
		& 5.530416e-02 
		& 4.789870e-02 \\ 
		\hline
	\end{tabular}
	\caption{Approximative minimal energy values $E(u_{\text{\tiny ref}})$ for the reference solution $u_{\text{\tiny ref}}$ for the values $\kappa=8,16,24,32,40$.}
	\label{tab:energy_reference_and_ev}
\end{table}

When comparing a reference minimizer $u$ with a discrete FE mininizer $u_h$ (for given mesh size $h$, polynomial order $p$ and fixed $\kappa$), the two states are ensured to be in phase so that the errors are not polluted by potential phase differences. For that, we compute the phase factors
\begin{align*}
\alpha_h := \frac{ \int_{\Omega} u\hspace{1pt} \overline{u_h} \hspace{1pt}\mbox{d}x }{ \left| \int_{\Omega} u \hspace{1pt} \overline{u_h} \hspace{1pt} \mbox{d}x \right|}, 
\end{align*}
which guarantee, by $|\alpha_h|=1$, that $|\alpha_h u_h | = |u_h|$ and $\int_{\Omega} u\hspace{1pt} \overline{\alpha_h u_h} \hspace{1pt}\mbox{d}x = |\alpha_h|^2=1$. Hence, $u$ and $\alpha_h u_h$ are in the same phase (that is $\alpha_h u_h \in \orthiu$) and we can compute the relevant errors $u - \alpha_h u_h$ as covered by our theory. For readability, we shall only write $u-u_h$ when presenting the results.

\subsection{Comparison of energy errors}

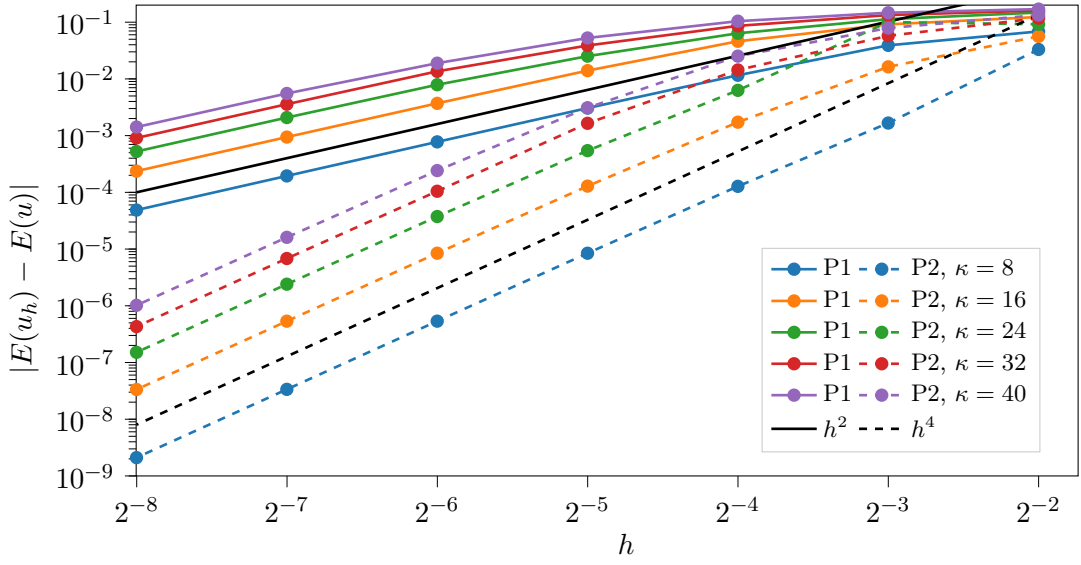
\begin{figure}[h]
	
	\centering{
\begin{tikzpicture}
   \begin{axis}[
        height=0.5\textwidth,
        width=0.9\textwidth,
	legend cell align={left},
	legend style={
	fill opacity=0.8,
	draw opacity=1,
	text opacity=1,
	at={(0.96,0.05)},
	anchor=south east,
	draw=white!80!black
	},
	log basis x={2},
	log basis y={10},
	tick align=outside,
	tick pos=left,
	x grid style={white!69.0196078431373!black},
	xmin=3.9e-03, xmax=3e-01,
	xmode=log,
	xtick style={color=black},
	xlabel={$h$},
	xlabel style={xshift=0.25cm, yshift=-0.1cm},
	y grid style={white!69.0196078431373!black},
	ymin=1e-09, ymax=2e-01,
	ymode=log,
	ytick style={color=black},
	ylabel={$|E(u_h) - E(u)|$},
	ylabel style={xshift=-0.5cm, yshift=0.25cm},
	legend columns=2,
	]
	\addplot [line width = 1, color0, mark=*, mark size=2, mark options={solid}]
	table {%
		2.500000000000000e-01     0.0686490469328192
 		1.250000000000000e-01     0.0390568536726402
		6.250000000000000e-02     0.0115935115276782
		3.125000000000000e-02     0.0030589454612752
		1.562500000000000e-02     0.0007760148732792
		7.812500000000000e-03     0.0001947321524892
                 3.906250000000000e-03     0.0000487289537062
	};
	\addlegendentry{\footnotesize P1}
	%
	\addplot [line width = 1, color0, dashed, mark=*, mark size=2, mark options={solid}]
	table {%
		2.500000000000000e-01     0.0330548049450732
 		1.250000000000000e-01     0.0016663780173582
		6.250000000000000e-02     0.0001276114785202
		3.125000000000000e-02     0.0000084454661782
		1.562500000000000e-02     0.0000005358387542
		7.812500000000000e-03     0.0000000336207712
                 3.906250000000000e-03     0.0000000021033002
	};
	\addlegendentry{\footnotesize P2, $\kappa = 8$}
	%
	\addplot [line width = 1, color1, mark=*, mark size=2, mark options={solid}]
	table {%
		2.500000000000000e-01     0.123487891940397133
 		1.250000000000000e-01     0.091431466872641133
		6.250000000000000e-02     0.046092562550846133
		3.125000000000000e-02     0.013994594743694133
		1.562500000000000e-02     0.003708252201141133
		7.812500000000000e-03     0.000942376944406133
                 3.906250000000000e-03     2.366028427345790e-04
	};
	\addlegendentry{\footnotesize P1}
	%
	\addplot [line width = 1, color1, dashed, mark=*, mark size=2, mark options={solid}]
	table {%
		2.500000000000000e-01     0.056249625494452133
 		1.250000000000000e-01     0.016322288116402133
		6.250000000000000e-02     0.001725626502027133
		3.125000000000000e-02     0.000128779066437133
		1.562500000000000e-02     0.000008454490098133
		7.812500000000000e-03     0.000000535256177133
                 3.906250000000000e-03     3.343599802585473e-08
	};
	\addlegendentry{\footnotesize P2, $\kappa=16$}
	%
	\addplot [line width = 1, color2, mark=*, mark size=2, mark options={solid}]
	table {%
		2.500000000000000e-01     0.147859004095639867  
 		1.250000000000000e-01     0.113058867802146867
		6.250000000000000e-02     0.063733685644205867
		3.125000000000000e-02     0.025184378558693867
		1.562500000000000e-02     0.007894355473875867 
		7.812500000000000e-03     0.002082006363042867
                 3.906250000000000e-03     0.000525210008332867 
	};
	\addlegendentry{\footnotesize P1}
	%
	\addplot [line width = 1, color2, dashed, mark=*, mark size=2, mark options={solid}]
	table {%
		2.500000000000000e-01     0.091876115118804867   
 		1.250000000000000e-01     0.100579938879802
		6.250000000000000e-02     0.006287765003780867
		3.125000000000000e-02     0.000543596270658867
		1.562500000000000e-02     0.000037443019637867 
		7.812500000000000e-03     0.000002401906390867
                 3.906250000000000e-03     0.000000151146060333
	};
	\addlegendentry{\footnotesize P2, $\kappa = 24$}
	%
	\addplot [line width = 1, color3, mark=*, mark size=2, mark options={solid}]
	table {%
		2.500000000000000e-01     0.160581419121754
 		1.250000000000000e-01     0.133429085941114
		6.250000000000000e-02     0.086200356430980    
		3.125000000000000e-02     0.038809463628050
		1.562500000000000e-02     0.013595166837361
		7.812500000000000e-03     0.003581780118932
                 3.906250000000000e-03     9.084448101686626e-04 
	};
	\addlegendentry{\footnotesize P1}
	%
	\addplot [line width = 1, color3, dashed, mark=*, mark size=2, mark options={solid}]
	table {%
		2.500000000000000e-01     0.117789581613689
 		1.250000000000000e-01     0.057478428962837
		6.250000000000000e-02     0.014362045342274
		3.125000000000000e-02     0.001648686947102
		1.562500000000000e-02     1.045590640125493e-04
		7.812500000000000e-03     6.816630762546128e-06
                 3.906250000000000e-03     4.292253561188897e-07 
	};
	\addlegendentry{\footnotesize P2, $\kappa = 32$}
	%
	\addplot [line width = 1, color4, mark=*, mark size=2, mark options={solid}]
	table {%
		2.500000000000000e-01     0.170116295547624
 		1.250000000000000e-01     0.146680750718399
		6.250000000000000e-02     0.104181336473454
		3.125000000000000e-02     0.052683385138651  
		1.562500000000000e-02     0.018940380684403
		7.812500000000000e-03     0.005535594196253
                 3.906250000000000e-03     0.001416179842608 
	};
	\addlegendentry{\footnotesize P1}
	%
	\addplot [line width = 1, color4, dashed, mark=*, mark size=2, mark options={solid}]
	table {%
		2.500000000000000e-01     0.133941226940954
 		1.250000000000000e-01     0.078852627056083
		6.250000000000000e-02     0.025210345838567
		3.125000000000000e-02     0.003093192809959
		1.562500000000000e-02     0.000241787990613
		7.812500000000000e-03     1.609798882507008e-05 
                 3.906250000000000e-03     1.019299650240058e-06 
	};
	\addlegendentry{\footnotesize P2, $\kappa = 40$}
	\addplot [line width = 1, black]
	table {%
		2.500000000000000e-01     0.4096
 		1.250000000000000e-01     0.1024
		6.250000000000000e-02     0.0256
		3.125000000000000e-02     0.0064
		1.562500000000000e-02     0.0016
		7.812500000000000e-03     0.0004
                 3.906250000000000e-03     0.0001
	};
	\addlegendentry{\footnotesize $h^2$}
	\addplot [line width = 1, black, dashed]
	table {%
		2.500000000000000e-01     0.134217728
 		1.250000000000000e-01     0.008388608
		6.250000000000000e-02     0.000524288
		3.125000000000000e-02     0.000032768
		1.562500000000000e-02     0.000002048
		7.812500000000000e-03     0.000000128
                 3.906250000000000e-03     0.000000008
	};
	\addlegendentry{\footnotesize $h^4$}
    \end{axis}
\end{tikzpicture}
	}
	\caption{Error in energy $|E(u)-E(u_h)|$ for $\kappa = 8,16,24,32,40$ between the reference minimal energy $E(u)$ and the minimal energy $E(u_h)$ in the finite element spaces for P1 (solid lines) and P2 (dashed lines) respectively.  
	}
	\label{fig:energy-error}
\end{figure}

We start with comparing the minimal energy level $E(u)$ with the discrete minimal energy levels $E(u_h)$ for P1 and P2 respectively. The corresponding results for $\kappa = 8,16,24,32,40$ are depicted in Figure \ref{fig:energy-error}. As predicted by Proposition \ref{prop-energy-error}, we can clearly see the second order convergence in $h$ for the P1 energy and fourth order convergence for the P2 energy. Furthermore, we can confirm that the resolution condition $h \kappa \lesssim 1$ is sufficient to observe the correct asymptotic convergence rates. As expected, the pre-asymptotic regime is still barely visible for $\kappa=8$ and it becomes slightly more pronounced with growing $\kappa$. For fixed $\kappa$, the asymptotic convergence regime starts essentially simultaneously for the P1 and P2 energy approximation. All in all, the experiments fully confirm the behavior of the energy error as stated in Proposition \ref{prop-energy-error}.
\subsection{Comparison of $L^2$- and $H^1_{\kappa}$-errors}
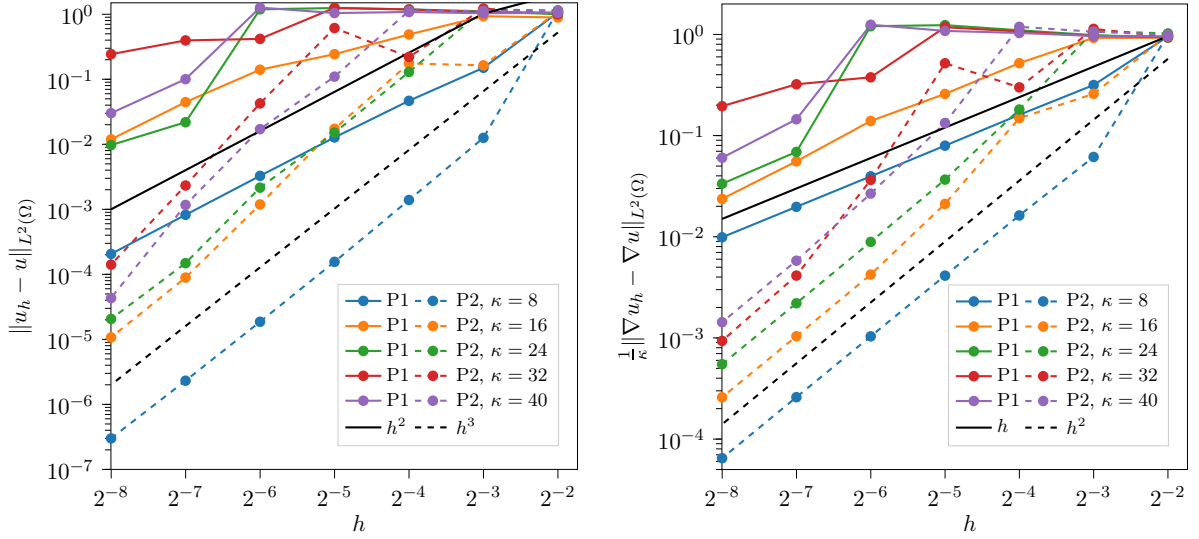
\begin{figure}[h]
	
	\begin{adjustbox}{width=0.5\linewidth}
\begin{tikzpicture}
   \begin{axis}[
        height=0.6\textwidth,
        width=0.6\textwidth,
	legend cell align={left},
	legend style={
	fill opacity=0.8,
	draw opacity=1,
	text opacity=1,
	at={(0.96,0.05)},
	anchor=south east,
	draw=white!80!black
	},
	log basis x={2},
	log basis y={10},
	tick align=outside,
	tick pos=left,
	x grid style={white!69.0196078431373!black},
	xmin=3.9e-03, xmax=3e-01,
	xmode=log,
	xtick style={color=black},
	xlabel={$h$},
	xlabel style={xshift=0.25cm, yshift=-0.1cm},
	y grid style={white!69.0196078431373!black},
	ymin=1e-07, ymax=1.5e+00,
	ymode=log,
	ytick style={color=black},
	ylabel={$\| u_h - u \|_{L^2(\Omega)}$},
	ylabel style={xshift=-0.5cm, yshift=0.25cm},
	legend columns=2,
	]
	\addplot [line width = 1, color0, mark=*, mark size=2, mark options={solid}]
	table {%
		2.500000000000000e-01     1.064764895563741
 		1.250000000000000e-01     0.150053830549995  
		6.250000000000000e-02     0.046639595574248
		3.125000000000000e-02     0.012736655685185 
		1.562500000000000e-02     0.003268538352030
		7.812500000000000e-03     8.224150953367682e-04
                 3.906250000000000e-03     2.060865457999999e-04
	};
	\addlegendentry{\footnotesize P1}
	%
	\addplot [line width = 1, color0, dashed, mark=*, mark size=2, mark options={solid}]
	table {%
		2.500000000000000e-01     1.128946663226182
 		1.250000000000000e-01     0.012572360996283
		6.250000000000000e-02     0.001389934373887
		3.125000000000000e-02     1.555442389674585e-04
		1.562500000000000e-02     1.865388762688436e-05
		7.812500000000000e-03     2.312902285107528e-06
                 3.906250000000000e-03     3.002797215051112e-07
	};
	\addlegendentry{\footnotesize P2, $\kappa = 8$}
	%
	\addplot [line width = 1, color1, mark=*, mark size=2, mark options={solid}]
	table {%
		2.500000000000000e-01     0.892740218075323
 		1.250000000000000e-01     0.932617449168569
		6.250000000000000e-02     0.488706888357458
		3.125000000000000e-02     0.243191716757446
		1.562500000000000e-02     0.140351845072075
		7.812500000000000e-03     0.044485542265553
                 3.906250000000000e-03     0.011909715312847
	};
	\addlegendentry{\footnotesize P1}
	%
	\addplot [line width = 1, color1, dashed, mark=*, mark size=2, mark options={solid}]
	table {%
		2.500000000000000e-01     1.093144319531300
 		1.250000000000000e-01     0.163850375908307
		6.250000000000000e-02     0.175114020363319
		3.125000000000000e-02     0.017281482855678
		1.562500000000000e-02     0.001188976980469
		7.812500000000000e-03     8.916024433462845e-05
                 3.906250000000000e-03     1.072862545144650e-05
	};
	\addlegendentry{\footnotesize P2, $\kappa=16$}
	%
	\addplot [line width = 1, color2, mark=*, mark size=2, mark options={solid}]
	table {%
		2.500000000000000e-01     1.000670541817422
 		1.250000000000000e-01     1.106592950587941
		6.250000000000000e-02     1.186677815395274
		3.125000000000000e-02     1.246725055610351
		1.562500000000000e-02     1.190326833565436
		7.812500000000000e-03     0.021850861140545
                 3.906250000000000e-03     0.009724102800119
	};
	\addlegendentry{\footnotesize P1}
	%
	\addplot [line width = 1, color2, dashed, mark=*, mark size=2, mark options={solid}]
	table {%
		2.500000000000000e-01     1.145234301880612
 		1.250000000000000e-01     1.180746801504670
		6.250000000000000e-02     0.129427288646972
		3.125000000000000e-02     0.015081395829735
		1.562500000000000e-02     0.002153382843506
		7.812500000000000e-03     1.486272229490621e-04
                 3.906250000000000e-03     2.061751405311361e-05
	};
	\addlegendentry{\footnotesize P2, $\kappa = 24$}
	%
	\addplot [line width = 1, color3, mark=*, mark size=2, mark options={solid}]
	table {%
		2.500000000000000e-01     1.058157776773067
 		1.250000000000000e-01     1.075251899056151
		6.250000000000000e-02     1.184885879499068
		3.125000000000000e-02     1.252953535138133
		1.562500000000000e-02     0.418134624053885 
		7.812500000000000e-03     0.395906576923468
                 3.906250000000000e-03     0.243331041755834
	};
	\addlegendentry{\footnotesize P1}
	%
	\addplot [line width = 1, color3, dashed, mark=*, mark size=2, mark options={solid}]
	table {%
		2.500000000000000e-01     1.005440845455322
 		1.250000000000000e-01     1.225719850732240
		6.250000000000000e-02     0.220339616360728
		3.125000000000000e-02     0.613430431015591
		1.562500000000000e-02     0.042568953214988 
		7.812500000000000e-03     0.002322784157689
                 3.906250000000000e-03     1.409927232990894e-04
	};
	\addlegendentry{\footnotesize P2, $\kappa = 32$}
	%
	\addplot [line width = 1, color4, mark=*, mark size=2, mark options={solid}]
	table {%
		2.500000000000000e-01     1.052298514405925
 		1.250000000000000e-01     1.036249721524094
		6.250000000000000e-02     1.082847814767705
		3.125000000000000e-02     1.043921143797822
		1.562500000000000e-02     1.266171720625219 
		7.812500000000000e-03     0.100428132420012
                 3.906250000000000e-03     0.030124310012315
	};
	\addlegendentry{\footnotesize P1}
	%
	\addplot [line width = 1, color4, dashed, mark=*, mark size=2, mark options={solid}]
	table {%
		2.500000000000000e-01     1.134804204290670
 		1.250000000000000e-01     1.127614987300715
		6.250000000000000e-02     1.127787095478574
		3.125000000000000e-02     0.109731527992190
		1.562500000000000e-02     0.017121091437065 
		7.812500000000000e-03     0.001166844714540
                 3.906250000000000e-03     4.316707992032286e-05
	};
	\addlegendentry{\footnotesize P2, $\kappa = 40$}
	\addplot [line width = 1, black]
	table {%
		2.500000000000000e-01     2.048
 		1.250000000000000e-01     1.024
		6.250000000000000e-02     0.256
		3.125000000000000e-02     0.064
		1.562500000000000e-02     0.016
		7.812500000000000e-03     0.004
                 3.906250000000000e-03     0.001
	};
	\addlegendentry{\footnotesize $h^2$}
	%
	\addplot [line width = 1, black, dashed]
	table {%
		2.500000000000000e-01    0.524288
 		1.250000000000000e-01    0.065536
		6.250000000000000e-02    0.008192
		3.125000000000000e-02    0.001024
		1.562500000000000e-02    0.000128
		7.812500000000000e-03    0.000016
                 3.906250000000000e-03    2e-06
	};
	\addlegendentry{\footnotesize $h^3$}
    \end{axis}
\end{tikzpicture}	
	\end{adjustbox}
	\hspace{0.1cm}
	\begin{adjustbox}{width=0.5\linewidth}
\begin{tikzpicture}
   \begin{axis}[
        height=0.6\textwidth,
        width=0.6\textwidth,
	legend cell align={left},
	legend style={
	fill opacity=0.8,
	draw opacity=1,
	text opacity=1,
	at={(0.96,0.05)},
	anchor=south east,
	draw=white!80!black
	},
	log basis x={2},
	log basis y={10},
	tick align=outside,
	tick pos=left,
	x grid style={white!69.0196078431373!black},
	xmin=3.9e-03, xmax=3e-01,
	xmode=log,
	xtick style={color=black},
	xlabel={$h$},
	xlabel style={xshift=0.25cm, yshift=-0.1cm},
	y grid style={white!69.0196078431373!black},
	ymin=5e-05, ymax=2e+00,
	ymode=log,
	ytick style={color=black},
	ylabel={$\tfrac{1}{\kappa} \| \nabla u_h - \nabla u\|_{L^2(\Omega)}$},
	ylabel style={xshift=-0.5cm, yshift=0.25cm},
	legend columns=2,
	]
	\addplot [line width = 1, color0, mark=*, mark size=2, mark options={solid}]
	table {%
		2.500000000000000e-01     0.954571782073502
 		1.250000000000000e-01     0.316034705183352
		6.250000000000000e-02     0.160451816313347
		3.125000000000000e-02     0.079514387841140
		1.562500000000000e-02     0.039569856401940
		7.812500000000000e-03     0.019756980646044
                 3.906250000000000e-03     0.009874874405080
	};
	\addlegendentry{\footnotesize P1}
	%
	\addplot [line width = 1, color0, dashed, mark=*, mark size=2, mark options={solid}]
	table {%
		2.500000000000000e-01     0.993676334498462
 		1.250000000000000e-01     0.061496426721215
		6.250000000000000e-02     0.016240581511332
		3.125000000000000e-02     0.004126520207998
		1.562500000000000e-02     0.001036250133490
		7.812500000000000e-03     2.593531171514743e-04
                 3.906250000000000e-03     6.473834157063776e-05
	};
	\addlegendentry{\footnotesize P2, $\kappa = 8$}
	%
	\addplot [line width = 1, color1, mark=*, mark size=2, mark options={solid}]
	table {%
		2.500000000000000e-01     0.925138866062612
 		1.250000000000000e-01     0.924462082391346
		6.250000000000000e-02     0.520088773918308
		3.125000000000000e-02     0.258311402708613
		1.562500000000000e-02     0.139473492216770
		7.812500000000000e-03     0.055634316913861
                 3.906250000000000e-03     0.023663351850619
	};
	\addlegendentry{\footnotesize P1}
	%
	\addplot [line width = 1, color1, dashed, mark=*, mark size=2, mark options={solid}]
	table {%
		2.500000000000000e-01     0.996156011926759
 		1.250000000000000e-01     0.258595910173735
		6.250000000000000e-02     0.148648416080996
		3.125000000000000e-02     0.021130424947878
		1.562500000000000e-02     0.004232577998943
		7.812500000000000e-03     0.001038024069348
                 3.906250000000000e-03     2.588036207124510e-04
	};
	\addlegendentry{\footnotesize P2, $\kappa=16$}
	%
	\addplot [line width = 1, color2, mark=*, mark size=2, mark options={solid}]
	table {%
		2.500000000000000e-01     0.933086664329161
 		1.250000000000000e-01     0.984244837207600
		6.250000000000000e-02     1.099296392981997
		3.125000000000000e-02     1.235478104319384
		1.562500000000000e-02     1.205324335586780
		7.812500000000000e-03     0.068977367734170
                 3.906250000000000e-03     0.033426383500327
	};
	\addlegendentry{\footnotesize P1}
	%
	\addplot [line width = 1, color2, dashed, mark=*, mark size=2, mark options={solid}]
	table {%
		2.500000000000000e-01     1.022784814107536
 		1.250000000000000e-01     1.081854721766629
		6.250000000000000e-02     0.181011116739935
		3.125000000000000e-02     0.036701196260005
		1.562500000000000e-02     0.008888630637450
		7.812500000000000e-03     0.002200250049548
                 3.906250000000000e-03     5.492949380883303e-04
	};
	\addlegendentry{\footnotesize P2, $\kappa = 24$}
	%
	\addplot [line width = 1, color3, mark=*, mark size=2, mark options={solid}]
	table {%
		2.500000000000000e-01     0.949657026411681
 		1.250000000000000e-01     0.965282011872039
		6.250000000000000e-02     1.069267152932629
		3.125000000000000e-02     1.178450088413673
		1.562500000000000e-02     0.376764231371264
		7.812500000000000e-03     0.321647260333476
                 3.906250000000000e-03     0.195195059768914
	};
	\addlegendentry{\footnotesize P1}
	%
	\addplot [line width = 1, color3, dashed, mark=*, mark size=2, mark options={solid}]
	table {%
		2.500000000000000e-01     0.944018132726887
 		1.250000000000000e-01     1.135175200943707
		6.250000000000000e-02     0.299757648993629
		3.125000000000000e-02     0.518802201395388
		1.562500000000000e-02     0.036478703085470
		7.812500000000000e-03     0.004133884911989
                 3.906250000000000e-03    9.342467185912502e-04
	};
	\addlegendentry{\footnotesize P2, $\kappa = 32$}
	%
	\addplot [line width = 1, color4, mark=*, mark size=2, mark options={solid}]
	table {%
		2.500000000000000e-01     0.951849338650917
 		1.250000000000000e-01     0.966260652265844
		6.250000000000000e-02     1.033527521614415
		3.125000000000000e-02     1.083665263574138
		1.562500000000000e-02     1.246237282635738
		7.812500000000000e-03     0.145026092049195
                 3.906250000000000e-03     0.060359293290927
	};
	\addlegendentry{\footnotesize P1}
	%
	\addplot [line width = 1, color4, dashed, mark=*, mark size=2, mark options={solid}]
	table {%
		2.500000000000000e-01     0.989079937764740
 		1.250000000000000e-01     1.063451646184564
		6.250000000000000e-02     1.188506461824488
		3.125000000000000e-02     0.133079856945677
		1.562500000000000e-02     0.026710145111903
		7.812500000000000e-03     0.005799591088139
                 3.906250000000000e-03     0.001431157456014
	};
	\addlegendentry{\footnotesize P2, $\kappa = 40$}
	%
	\addplot [line width = 1, black]
	table {%
		2.500000000000000e-01     0.96
 		1.250000000000000e-01     0.48
		6.250000000000000e-02     0.24
		3.125000000000000e-02     0.12
		1.562500000000000e-02     0.06
		7.812500000000000e-03     0.03
                 3.906250000000000e-03     0.015
	};
	\addlegendentry{\footnotesize $h$}
	\addplot [line width = 1, black, dashed]
	table {%
		2.500000000000000e-01     0.57344
 		1.250000000000000e-01     0.14336
		6.250000000000000e-02     0.03584
		3.125000000000000e-02     0.00896
		1.562500000000000e-02     0.00224
		7.812500000000000e-03     0.00056
                 3.906250000000000e-03     1.4e-04
	};
	\addlegendentry{\footnotesize $h^2$}
    \end{axis}
\end{tikzpicture}
	\end{adjustbox}
	\caption{
	Comparison of the errors $\| u - u_{h} \|_{L^2(\Omega)}$ and $\tfrac{1}{\kappa}\| \nabla u - \nabla u_{h} \|_{L^2(\Omega)}$ for the discrete minimizers $u_h$ in the FEM spaces $V_{h,1}$ (P1, solid lines) and $V_{h,2}$ (P2, dashed lines), where $\kappa = 8,16,24,32,40$.}
	\label{fig:L2-H1-errors}
\end{figure}

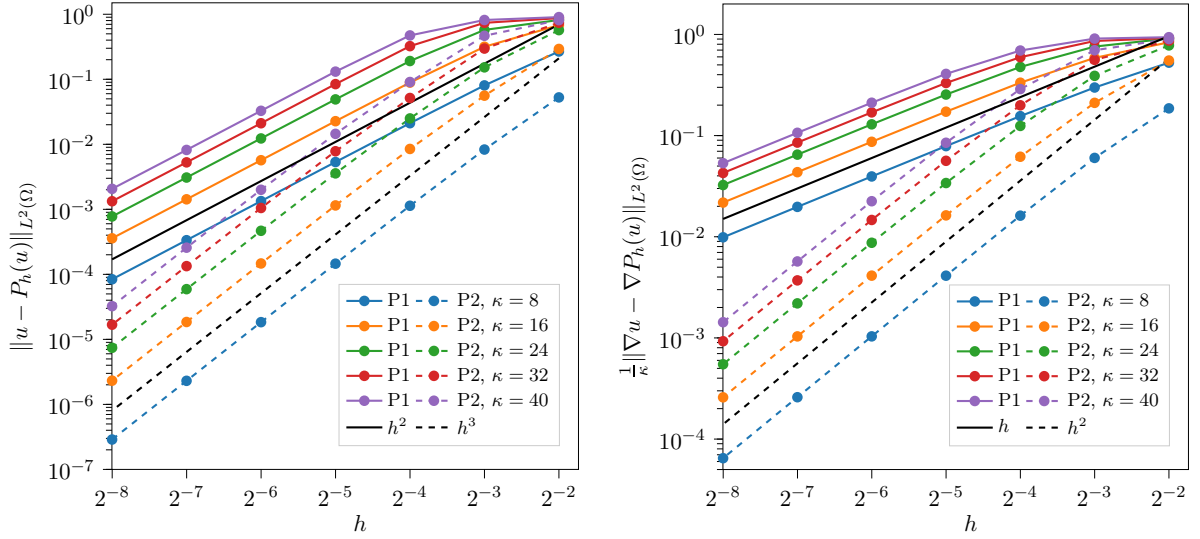
\begin{figure}[h]
	\begin{adjustbox}{width=0.5\linewidth}
\begin{tikzpicture}
   \begin{axis}[
        height=0.6\textwidth,
        width=0.6\textwidth,
	legend cell align={left},
	legend style={
	fill opacity=0.8,
	draw opacity=1,
	text opacity=1,
	at={(0.96,0.05)},
	anchor=south east,
	draw=white!80!black
	},
	log basis x={2},
	log basis y={10},
	tick align=outside,
	tick pos=left,
	x grid style={white!69.0196078431373!black},
	xmin=3.9e-03, xmax=3e-01,
	xmode=log,
	xtick style={color=black},
	xlabel={$h$},
	xlabel style={xshift=0.25cm, yshift=-0.1cm},
	y grid style={white!69.0196078431373!black},
	ymin=1e-07, ymax=1.5e+00,
	ymode=log,
	ytick style={color=black},
	ylabel={$\|  u - P_h(u)  \|_{L^2(\Omega)}$},
	ylabel style={xshift=-0.5cm, yshift=0.25cm},
	legend columns=2,
	]
	\addplot [line width = 1, color0, mark=*, mark size=2, mark options={solid}]
	table {%
		2.500000000000000e-01     0.267613271960889
 		1.250000000000000e-01     0.080478547856377
		6.250000000000000e-02     0.021109660980896
		3.125000000000000e-02     0.005337359020718
		1.562500000000000e-02     0.001338170016526
		7.812500000000000e-03     3.349665939707012e-04
                 3.906250000000000e-03     8.397166591310588e-05
	};
	\addlegendentry{\footnotesize P1}
	%
	\addplot [line width = 1, color0, dashed, mark=*, mark size=2, mark options={solid}]
	table {%
		2.500000000000000e-01     0.052619567451180
 		1.250000000000000e-01     0.008306294097603
		6.250000000000000e-02     0.001133971225346
		3.125000000000000e-02     1.459332124684916e-04
		1.562500000000000e-02     1.839060189494005e-05
		7.812500000000000e-03     2.303852435684596e-06
                 3.906250000000000e-03     2.881209507290185e-07
	};
	\addlegendentry{\footnotesize P2, $\kappa = 8$}
	%
	\addplot [line width = 1, color1, mark=*, mark size=2, mark options={solid}]
	table {%
		2.500000000000000e-01     0.679922318443323
 		1.250000000000000e-01     0.316164191577368
		6.250000000000000e-02     0.088959867014587
		3.125000000000000e-02     0.022669116870144
		1.562500000000000e-02     0.005689921136596
		7.812500000000000e-03     0.001425474575499
                 3.906250000000000e-03     3.583030730269583e-04
	};
	\addlegendentry{\footnotesize P1}
	%
	\addplot [line width = 1, color1, dashed, mark=*, mark size=2, mark options={solid}]
	table {%
		2.500000000000000e-01     0.292624311220371
 		1.250000000000000e-01     0.055984420107598
		6.250000000000000e-02     0.008483407387361
		3.125000000000000e-02     0.001147515130218
		1.562500000000000e-02     1.469877525256945e-04
		7.812500000000000e-03     1.849726305521159e-05
                 3.906250000000000e-03     2.316114092025825e-06
	};
	\addlegendentry{\footnotesize P2, $\kappa=16$}
	%
	\addplot [line width = 1, color2, mark=*, mark size=2, mark options={solid}]
	table {%
		2.500000000000000e-01     0.816162921894385
 		1.250000000000000e-01     0.574478142494099
		6.250000000000000e-02     0.190840464694219
		3.125000000000000e-02     0.049046297203338
		1.562500000000000e-02     0.012302615560233
		7.812500000000000e-03     0.003082603607512
                 3.906250000000000e-03     7.764723805213879e-04
	};
	\addlegendentry{\footnotesize P1}
	%
	\addplot [line width = 1, color2, dashed, mark=*, mark size=2, mark options={solid}]
	table {%
		2.500000000000000e-01     0.569849237036193
 		1.250000000000000e-01     0.152545503264869
		6.250000000000000e-02     0.024995287939261
		3.125000000000000e-02     0.003572402728049
		1.562500000000000e-02     4.673969596210482e-04
		7.812500000000000e-03     5.918470777177497e-05
                 3.906250000000000e-03     7.423069473790016e-06
	};
	\addlegendentry{\footnotesize P2, $\kappa = 24$}
	%
	\addplot [line width = 1, color3, mark=*, mark size=2, mark options={solid}]
	table {%
		2.500000000000000e-01     0.871417397649451
 		1.250000000000000e-01     0.738234576883724
		6.250000000000000e-02     0.323245325866659
		3.125000000000000e-02     0.084293481115998
		1.562500000000000e-02     0.021071906534248
		7.812500000000000e-03     0.005275185884030
                 3.906250000000000e-03     0.001330473955117
	};
	\addlegendentry{\footnotesize P1}
	%
	\addplot [line width = 1, color3, dashed, mark=*, mark size=2, mark options={solid}]
	table {%
		2.500000000000000e-01     0.740295204344505
 		1.250000000000000e-01     0.296700974576537
		6.250000000000000e-02     0.051534627337121
		3.125000000000000e-02     0.007823128315731
		1.562500000000000e-02     0.001048455045913
		7.812500000000000e-03     1.337498478118587e-04
                 3.906250000000000e-03     1.680896318282327e-05
	};
	\addlegendentry{\footnotesize P2, $\kappa = 32$}
	%
	\addplot [line width = 1, color4, mark=*, mark size=2, mark options={solid}]
	table {%
		2.500000000000000e-01     0.903828694452771
 		1.250000000000000e-01     0.815638132058588
		6.250000000000000e-02     0.472502239945422
		3.125000000000000e-02     0.131118352918641
		1.562500000000000e-02     0.032740389529451
		7.812500000000000e-03     0.008186818429141
                 3.906250000000000e-03     0.002065129626249
	};
	\addlegendentry{\footnotesize P1}
	%
	\addplot [line width = 1, color4, dashed, mark=*, mark size=2, mark options={solid}]
	table {%
		2.500000000000000e-01     0.812415059895561
 		1.250000000000000e-01     0.466711027340966
		6.250000000000000e-02     0.090722945834030
		3.125000000000000e-02     0.014476830747287
		1.562500000000000e-02     0.001997437590813
		7.812500000000000e-03     2.572266722888303e-04
                 3.906250000000000e-03     3.241080473958940e-05
	};
	\addlegendentry{\footnotesize P2, $\kappa = 40$}
	\addplot [line width = 1, black]
	table {%
		2.500000000000000e-01     0.69632
 		1.250000000000000e-01     0.17408
		6.250000000000000e-02     0.04352
		3.125000000000000e-02     0.01088
		1.562500000000000e-02     0.00272
		7.812500000000000e-03     0.00068
                 3.906250000000000e-03     0.00017
	};
	\addlegendentry{\footnotesize $h^2$}
	%
	\addplot [line width = 1, black, dashed]
	table {%
		2.500000000000000e-01    0.2097152
 		1.250000000000000e-01    0.0262144
		6.250000000000000e-02    0.0032768
		3.125000000000000e-02    0.0004096
		1.562500000000000e-02    0.0000512
		7.812500000000000e-03    0.0000064
                 3.906250000000000e-03    0.8e-06
	};
	\addlegendentry{\footnotesize $h^3$}
    \end{axis}
\end{tikzpicture}	
	\end{adjustbox}
	\hspace{0.1cm}
	\begin{adjustbox}{width=0.5\linewidth}
\begin{tikzpicture}
   \begin{axis}[
        height=0.6\textwidth,
        width=0.6\textwidth,
	legend cell align={left},
	legend style={
	fill opacity=0.8,
	draw opacity=1,
	text opacity=1,
	at={(0.96,0.05)},
	anchor=south east,
	draw=white!80!black
	},
	log basis x={2},
	log basis y={10},
	tick align=outside,
	tick pos=left,
	x grid style={white!69.0196078431373!black},
	xmin=3.9e-03, xmax=3e-01,
	xmode=log,
	xtick style={color=black},
	xlabel={$h$},
	xlabel style={xshift=0.25cm, yshift=-0.1cm},
	y grid style={white!69.0196078431373!black},
	ymin=5e-05, ymax=2e+00,
	ymode=log,
	ytick style={color=black},
	ylabel={$\tfrac{1}{\kappa} \| \nabla u - \nabla P_h(u) \|_{L^2(\Omega)}$},
	ylabel style={xshift=-0.5cm, yshift=0.25cm},
	legend columns=2,
	]
	\addplot [line width = 1, color0, mark=*, mark size=2, mark options={solid}]
	table {%
		2.500000000000000e-01     0.525972208681660
 		1.250000000000000e-01     0.298635763001205
		6.250000000000000e-02     0.155722185389541
		3.125000000000000e-02     0.078705115280097
		1.562500000000000e-02     0.039458967855574
		7.812500000000000e-03     0.019742811532197
                 3.906250000000000e-03     0.009873071433921
	};
	\addlegendentry{\footnotesize P1}
	%
	\addplot [line width = 1, color0, dashed, mark=*, mark size=2, mark options={solid}]
	table {%
		2.500000000000000e-01     0.185860873793322
 		1.250000000000000e-01     0.060169689093632
		6.250000000000000e-02     0.016162029828644
		3.125000000000000e-02     0.004122142414938
		1.562500000000000e-02     0.001035989191916
		7.812500000000000e-03     2.593273795343448e-04
                 3.906250000000000e-03     6.473727098480643e-05
	};
	\addlegendentry{\footnotesize P2, $\kappa = 8$}
	%
	\addplot [line width = 1, color1, mark=*, mark size=2, mark options={solid}]
	table {%
		2.500000000000000e-01     0.836292330814837
 		1.250000000000000e-01     0.586390351299707
		6.250000000000000e-02     0.333381689825550
		3.125000000000000e-02     0.172286329916609
		1.562500000000000e-02     0.086842928527487
		7.812500000000000e-03     0.043508794508445
                 3.906250000000000e-03     0.021765297335069
	};
	\addlegendentry{\footnotesize P1}
	%
	\addplot [line width = 1, color1, dashed, mark=*, mark size=2, mark options={solid}]
	table {%
		2.500000000000000e-01     0.551990681311794
 		1.250000000000000e-01     0.210682070110334
		6.250000000000000e-02     0.061798552490580
		3.125000000000000e-02     0.016266053109485
		1.562500000000000e-02     0.004126023485311
		7.812500000000000e-03     0.001035410062727
                 3.906250000000000e-03     2.586507666649759e-04
	};
	\addlegendentry{\footnotesize P2, $\kappa=16$}
	%
	\addplot [line width = 1, color2, mark=*, mark size=2, mark options={solid}]
	table {%
		2.500000000000000e-01     0.912199202598424
 		1.250000000000000e-01     0.758337331866406
		6.250000000000000e-02     0.477161463931670
		3.125000000000000e-02     0.254465518237875
		1.562500000000000e-02     0.129191190728689
		7.812500000000000e-03     0.064838125806809
                 3.906250000000000e-03     0.032449262946545
	};
	\addlegendentry{\footnotesize P1}
	%
	\addplot [line width = 1, color2, dashed, mark=*, mark size=2, mark options={solid}]
	table {%
		2.500000000000000e-01     0.778955353504647
 		1.250000000000000e-01     0.389615420258144
		6.250000000000000e-02     0.124594332750756
		3.125000000000000e-02     0.033966374398184
		1.562500000000000e-02     0.008720403913464
		7.812500000000000e-03     0.002195744748202
                 3.906250000000000e-03     5.489991410125348e-04
	};
	\addlegendentry{\footnotesize P2, $\kappa = 24$}
	%
	\addplot [line width = 1, color3, mark=*, mark size=2, mark options={solid}]
	table {%
		2.500000000000000e-01     0.921444160139615
 		1.250000000000000e-01     0.861247082659407
		6.250000000000000e-02     0.595730313664706
		3.125000000000000e-02     0.331141366977008
		1.562500000000000e-02     0.169614165278683
		7.812500000000000e-03     0.085300542085921
                 3.906250000000000e-03     0.042711460931600
	};
	\addlegendentry{\footnotesize P1}
	%
	\addplot [line width = 1, color3, dashed, mark=*, mark size=2, mark options={solid}]
	table {%
		2.500000000000000e-01     0.869000421837383
 		1.250000000000000e-01     0.560478262087220
		6.250000000000000e-02     0.199355028881317
		3.125000000000000e-02     0.056332113647795
		1.562500000000000e-02     0.014656104698015
		7.812500000000000e-03     0.003704911426015
                 3.906250000000000e-03     9.273122998924553e-04
	};
	\addlegendentry{\footnotesize P2, $\kappa = 32$}
	%
	\addplot [line width = 1, color4, mark=*, mark size=2, mark options={solid}]
	table {%
		2.500000000000000e-01     0.939074216696428
 		1.250000000000000e-01     0.909543007603891
		6.250000000000000e-02     0.693128700474052
		3.125000000000000e-02     0.407256746711859
		1.562500000000000e-02     0.211284394187172
		7.812500000000000e-03     0.106559536249405
                 3.906250000000000e-03     0.053392833567884
	};
	\addlegendentry{\footnotesize P1}
	%
	\addplot [line width = 1, color4, dashed, mark=*, mark size=2, mark options={solid}]
	table {%
		2.500000000000000e-01     0.900291790259411
 		1.250000000000000e-01     0.697624153337859
		6.250000000000000e-02     0.288346888613562
		3.125000000000000e-02     0.084935152802582
		1.562500000000000e-02     0.022454723333854
		7.812500000000000e-03     0.005704512612941
                 3.906250000000000e-03     0.001429697002490
	};
	\addlegendentry{\footnotesize P2, $\kappa = 40$}
	%
	\addplot [line width = 1, black]
	table {%
		2.500000000000000e-01     0.96
 		1.250000000000000e-01     0.48
		6.250000000000000e-02     0.24
		3.125000000000000e-02     0.12
		1.562500000000000e-02     0.06
		7.812500000000000e-03     0.03
                 3.906250000000000e-03     0.015
	};
	\addlegendentry{\footnotesize $h$}
	\addplot [line width = 1, black, dashed]
	table {%
		2.500000000000000e-01     0.57344
 		1.250000000000000e-01     0.14336
		6.250000000000000e-02     0.03584
		3.125000000000000e-02     0.00896
		1.562500000000000e-02     0.00224
		7.812500000000000e-03     0.00056
                 3.906250000000000e-03     1.4e-04
	};
	\addlegendentry{\footnotesize $h^2$}
    \end{axis}
\end{tikzpicture}
	\end{adjustbox}
	\caption{
	Comparison of the errors $\| u - P_h(u) \|_{L^2(\Omega)}$ and $\tfrac{1}{\kappa}\| \nabla u - \nabla P_h(u) \|_{L^2(\Omega)}$ for the $H^1_\kappa$-best-approximation $P_h(u)$ of $u$ given by \eqref{def-Phu-projec} for $V_{h,1}$ (P1, solid lines) and $V_{h,2}$ (P2, dashed lines). The errors include the cases $\kappa = 8,16,24,32,40$.		
	}
	\label{fig:L2-H1-errors-bestapprox}
\end{figure}

Next, we compare the actual $L^2$- and $H^1_{\kappa}$-errors for discrete minimizers $u_h$ for P1 and P2 respectively. The results are depicted in Figure \ref{fig:L2-H1-errors}. As predicted by Theorem \ref{theorem:main-result}, the resolution condition $h \kappa \lesssim 1$ is no longer sufficient to guarantee reasonable approximations and we can identify a very pronounced pre-asymptotic regime. Postulating an algebraic growth of $\rho(\kappa)$ in $\kappa$, the analytical results of Theorem \ref{theorem:main-result} imply an effective resolution condition of the form
$\kappa^m (h\kappa)^p \lesssim 1$ for some unknown power $m\ge 1$. This is fully consistent with the numerical results in Figure \ref{fig:L2-H1-errors}. On the one hand, the pre-asymptotic regime is much larger than what we observed for the energy error in Figure \ref{fig:energy-error} (where $h \kappa \lesssim 1$) and on the other hand, we can also observe that the pre-asymptotic regime is signficantly reduced for the P2-FEM approximations compared to P1-FEM. This confirms that the usage of higher order finite elements can reduce the pollution effect significantly. 

Asymptotically, we observe for the P1-FEM approximation a convergence of order $\mathcal{O}(h)$ for the $H^1$-error and $\mathcal{O}(h^2)$ for the $L^2$-error. For P2-FEM we observe the asymptotic convergence of order $\mathcal{O}(h^2)$ for the $H^1$-error and $\mathcal{O}(h^3)$ for the $L^2$-error. These are the natural expected rates.

It is also worth mentioning that the errors for $\kappa=40$ appear asymptotically slightly better than the errors for $\kappa=32$. This might indicate that the state that we found for $\kappa=40$ is only a local minimizer, but not a global minimizer. However, even extensive tests with numerous starting values and path-tracking strategies did not lead to a state of smaller energy. Hence, it remains open why the errors for $\kappa=40$ are asymptotically smaller than for $\kappa=32$.

To complete the investigation of the pre-asymptotic effect, we also study the behaviour of the $H^1_\kappa$-best-approximation $P_h(u)$ of $u$. Here, $P_h(u)$ is given by the $H^1_\kappa$-projection $P_h : H^1(\Omega) \rightarrow V_{h}$ with
\begin{align}
\label{def-Phu-projec}
\tfrac{1}{\kappa^2} (\nabla P_h(u) - \nabla u , \nabla v_h )_{L^2(\Omega)} +  ( P_h(u) - u , v_h )_{L^2(\Omega)} 
\,\,= \,\, 0
\qquad \mbox{for all } v_h \in V_{h}.
\end{align}
The corresponding results are presented in Figure \ref{fig:L2-H1-errors-bestapprox}, where we observe the same convergence regime as for the energy error. This is exactly what we expect since the best-approximation is not affected by the pollution effect and convergence is obtained as soon as $h \kappa \lesssim 1$. Indeed, after a short initial phase, both the $H^1_{\kappa}$- and $L^2$-errors (for P1 and P2) are converging with the expected rate.

In conclusion, our numerical experiments fully support our theoretical predictions and in particular the existence of a significant numerical pollution effect that is decreased by the use of higher order finite elements.

\def\cprime{$'$}

\appendix

\section{Quasi-interpolation on curved meshes}

\label{appendix:section:quasi-interpolation}

\subsection{Setting}

In this appendix, we consider a domain $\Omega \subset \mathbb R^d$, with $d=2$ or $3$.
We denote by $L$ the diameter of $\Omega$. $\GD \subset \partial \Omega$ is a relatively
open subset of the boundary that is assumed to have a (possible empty) Lipschitz boundary
$\partial \GD$.

\subsubsection{Reference simplicies}

In this section we fix a (closed) reference $d$ simplex $\widehat{K}$. We also fix a
reference $d-1$ simplex $\widehat{F}_{d-1} \subset \mathbb R^{d-1}$. For simplicity, we assume that
$|\widehat{K}| \leq 1$, which is the case for the standard choice of the convex hull
formed from the canonical basis vectors of $\mathbb R^d$. We denote by $\widehat{\mathcal{F}}$
the set of faces of $\widehat{K}$. For each face $\widehat{F} \in \widehat{\mathcal{F}}$,
there exists a finite number of bijective affine mappings $A: \widehat{F} \to \widehat{F}_{d-1}$.
A specific affine mapping is selected by specifying on which vertex of $\widehat{F}_{d-1}$
is mapped each vertex of $\widehat{F}$.

\subsubsection{Curved meshes}

\begin{definition}[Curved matching mesh]
\label{definition_curved_mesh}
A curved matching mesh $\mathcal{T}_h$ of $\Omega$ is a finite collection of elements
$K$ with associated mappings $\mathcal{F}_K: \widehat{K} \to K$ having the following properties.
We first ask that the elements are non-degenerate, meaning that for each
$K \in \mathcal{T}_h$ the mapping $\mathcal{F}_K: \widehat{K} \to K$ is bi-Lipschitz.
We further demand that $\mathcal{T}_h$ is non-overlapping, i.e., that for any two distinct
elements $K_\pm \in \mathcal{T}_h$, the intersection $K_- \cap K_+$ is either empty, a single
point, or an $\ell$-manifold with boundary for some $\ell \in \{1,\dots,d-1\}$.
We also require that the mappings are compatible, meaning that if $F := K_- \cap K_+$
is a $d-1$ manifold for some $K_\pm \in \mathcal{T}_h$ then
$\widehat{F}_\pm := \mathcal{F}_{K_\pm}^{-1}(F) \in \widehat{\mathcal{F}}$
and there exist bijective affine mappings $A_\pm: \widehat{F}_{d-1} \to \widehat{F}_\pm$
such that
\begin{equation}
\label{eq_compatible_mappings}
\mathcal{F}_{K_-} \circ A_- = \mathcal{F}_{K_+} \circ A_+
\end{equation}
as maps from $\widehat{F}_{d-1}$ to $F$. Finally, we demand that the mesh covers the domain, i.e.,
\begin{equation*}
\overline{\Omega} = \bigcup \mathcal{T}_h,
\end{equation*}
and that the mesh is compatible with the boundary condition, meaning that every boundary face
either entirely belongs to $\overline{\GD}$ or $\partial \Omega \setminus \GD$.
\end{definition}

In what follows we consider such a curved matching mesh $\mathcal{T}_h$ of $\Omega$.
Notice that Definition~\ref{definition_curved_mesh} implies that the elements
$K \in \mathcal{T}_h$ are closed sets. It is also clear that if the maps $\mathcal{F}_K$
are affine, then our definition is the standard definition of a matching mesh
(see e.g.~\cite[Definition 6.1.1]{ern_guermond_2021a}).

The notion of ``patch'' will be important in what follows. For $K \in \mathcal{T}_h$,
we finally employ the notation
\begin{equation*}
\mathcal{T}_h^K := \left \{
K' \in \mathcal{T}_h \; | \; K' \cap K \neq \emptyset
\right \}
\end{equation*}
for the patch of elements around $K$.

We now introduce a few metrics that measure the fineness and regularity of the mesh.
Classically, the fineness of the mesh is quantified with the elements diameters
\begin{equation*}
h_K := \max_{\bx,\by \in K} |\bx-\by| \qquad \forall K \in \mathcal{T}_h,
\end{equation*}
and $h := \max_{K \in \mathcal{T}_h} h_K$ is known as the mesh size.
The ``shape regularity'' of each element $K \in \mathcal{T}_h$, is
measured by the quantity
\begin{equation*}
M_K := \max(h_K \|(J\mathcal{F}_K)^{-1}\|_{L^\infty(\widehat{K})},h_K^{-1} \|J\mathcal{F}_K\|_{L^\infty(\widehat{K})}),
\end{equation*}
where $J\mathcal{F}_K$ denotes the Jacobian of $\mathcal{F}_K$ and where, for each $\widehat{\bx} \in \widehat{K}$, the matrix norm
associated with the Euclidean metric is considered. For $\mathcal{T} \subset \mathcal{T}_h$,
we also use the notation $M_{\mathcal{T}} := \max_{K \in \mathcal{T}} M_K$. These
metrics are the natural extensions of the ones used for straight elements.
The fact that $M_{\mathcal{T}_h}$ is well-defined and finite is a consequence of
Definition~\ref{definition_curved_mesh}.
We will sometimes need more regularity on the element mappings.
We will say that the element $K \in \mathcal{T}_h$ is $r$-regular
if $\mathcal{F}_K \in C^r(\widehat{K},K)$. In this case, we will
write
\begin{equation}
\label{eq_dJFK}
M_K(r)
:=
\max_{|\alpha| \leq r}
\left (\frac{h_K}{L}\right )^{|\alpha|}
\frac{\|\partial^\alpha \mathcal{F}_K\|_{L^\infty(\widehat{K})}}{L}.
\end{equation}
The mesh $\mathcal{T}_h$ is $r$-regular if all its elements are, and we then write
$M_{\mathcal{T}}(r) := \max_{K \in \mathcal{T}} M_K(r)$ for all $\mathcal{T} \subset \mathcal{T}_h$.

For later reference, we note that for all $K \in \mathcal{T}_h$,
there exists a constant $\mu_K \geq 1$ solely depending on $M_K$
such that
\begin{equation}
\label{eq_muK}
\mu_K^{-1} h_K^d \leq |(\det J\mathcal{F}_K)(\widehat{\bx})| \leq \mu_K h_K^d
\qquad
\forall \widehat{\bx} \in \widehat{K}.
\end{equation}

\subsubsection{Finite element space}

In the remainder, we fix $p \geq 1$,  and consider a finite element space
\begin{equation*}
V_h := \left \{
v_h \in H^1_{\GD}(\Omega) \; | \;
v_h|_K \circ \mathcal{F}_K \in \mathcal{P}_p(\widehat{K})
\right \}
\end{equation*}
where $\mathcal{P}_p(\widehat{K})$ is the space of (complex-valued) polynomials
of degree at most $p$ defined on the reference simplex $\widehat{K}$. The
notation
\begin{equation*}
\mathcal{P}(K) := \left \{
v \in H^1(K) \; | \; v \circ \mathcal{F}_K \in \mathcal{P}_p(\widehat{K}),
\right \}
\end{equation*}
will sometimes be useful. We also denote by
\begin{equation*}
V_h^\flat
:=
\left \{
v_h \in L^2(\Omega) \; | \; v_h|_K \in \mathcal{P}(K) \; \forall K \in \mathcal{T}_h
\right \}.
\end{equation*}
the ``broken'' counterpart of $V_h$.

\subsubsection{Sobolev spaces}

For a submesh $\mathcal{T} \subset \mathcal{T}_h$ covering a domain $D$,
if $s \geq 1$ and $1 \leq q \leq +\infty$, we introduce the broken Sobolev space
\begin{equation*}
W^{s,q}(\mathcal{T})
:=
\left \{
v \in L^q(D) \; | \; v|_K \in W^{s,q}(K) \; \forall K \in \mathcal{T}
\right \}
\end{equation*}
equipped with the semi-norm
\begin{equation*}
|v|_{W^{s,q}(\mathcal{T})}
:=
\left (
\sum_{K \in \mathcal{T}} \sum_{|\alpha| \leq s} \|\partial^\alpha v\|_{L^q(K)}^q
\right )^{1/q}
\end{equation*}
and the norm
\begin{equation*}
\|v\|_{W^{s,q}(\mathcal{T})}
=
\sum_{t=0}^s L^t |v|_{W^{t,q}(\mathcal{T})}.
\end{equation*}

As usual, we employ the notation $H^1(\mathcal{T}) = W^{1,2}(\mathcal{T})$,
for the Hilbert space and for the subscript in the associated norm and semi-norm.

\subsection{Main result}

\begin{theorem}[Quasi-interpolation]
\label{theorem_quasi_interpolant}
If $\mathcal{T}_h$ is a curved matching mesh of $\Omega$, then there exists a projection
$I_h: H^1_{\GD}(\Omega) \to V_h$ with the following properties.
If $v \in H^1_{\GD}(\Omega) \cap W^{s+1,q}(\mathcal{T}_h)$ for some
$2 \leq q \leq +\infty$ and $0 \leq s \leq p$ and if $\mathcal{T}_h$
is $s+1$ regular, then for each $K \in \mathcal{T}_h$, we have
\begin{equation}
\label{eq_error_estimate}
h_K^{-1} \|v-I_h(v)\|_{L^r(K)}+|v-I_h(v)|_{W^{1,r}(K)}
\lesssim
h_K^{d/r-d/q}
\left (\frac{h_K}{L}\right )^s
\sum_{t=0}^s L^t|v|_{W^{t+1,q}(\mathcal{T}_h^K)}.
\end{equation}
for all $1 \leq r \leq q$, with a hidden constant
depending on $p$, $r$, $q$, $M_{\mathcal{T}_h^K}$
and $M_{\mathcal{T}_h^K}(s+1)$.  It follows in particular that
\begin{equation}
\label{eq_error_estimate_global}
h^{-1}\|v-I_h(v)\|_{L^r(\Omega)} + |v-I_h(v)|_{W^{1,r}(\Omega)}
\lesssim
h^{d/r-d/q} h^s \|v\|_{W^{s+1,q}(\mathcal{T}_h)}
\end{equation}
with a hidden constant depending on $p$, $r$, $q$,
$M_{\mathcal{T}_h}$ and $M_{\mathcal{T}_h}(s+1)$.
\end{theorem}

For easy reference, we highlight the following direct consequence
of Theorem~\ref{theorem_quasi_interpolant}. The proof simply follows
by selecting $r=q=2$ and $s=0$, and applying a triangular inequality.

\begin{corollary}[Stability in energy norm]
Let $v \in H^1_{\GD}(\Omega)$. For all $K \in \mathcal{T}_h$ we have
\begin{equation*}
|I_h(v)|_{H^1(K)}
\lesssim
|v|_{H^1(\mathcal{T}_h^K)},
\end{equation*}
In particular,
\begin{equation*}
|I_h(v)|_{H^1(\Omega)}
\lesssim
|v|_{H^1(\Omega)}.
\end{equation*}
Here, $I_h$ is the projection from Theorem~\ref{theorem_quasi_interpolant}
and the hidden constants only depend on $M_{\mathcal{T}_h}$.
\end{corollary}

\subsection{Consequences of mesh regularity}

Here, we show that the number of elements in a patch is bounded
by a constant that only depends on $M_{\mathcal{T}_h}$, and
that they all have similar diameters. These results are standard
for straight elements and are extended here to curved ones.

\begin{lemma}[Element patches]
\label{lemma_patch}
There exists a constant $N_{\mathcal{T}_h}$ only depending on $M_{\mathcal{T}_h}$
such that the cardinality of $\mathcal{T}_h^K$ is less than
$N_{\mathcal{T}_h}$ for all $K \in \mathcal{T}_h$. In addition, we have
\begin{equation}
\label{eq_comparable_diameters}
h_{K'} \simeq h_K
\end{equation}
with hidden constants only depending on $M_{\mathcal{T}_h}$.
\end{lemma}

\begin{proof}
Consider an element $K \in \mathcal{T}_h$ and one of its vertices
$\ba$ (i.e. the image of a vertex of $\widehat{K}$ through $\mathcal{F}_K$).
We can then consider, for each $K' \in \mathcal{T}_h$ with $\ba \in K'$ the affine map
\begin{equation*}
F_{K'}: \widehat{K} \ni \widehat{\bx} \to \ba+(J\mathcal{F}_{K'})(\ba) \,\widehat{\bx} \in K_\dagger',
\end{equation*}
where $K_\dagger'$ is a straight tetrahedron with shape-regularity measure only depending
on $M_{K'}$.  We now observe that because the mappings are matching, the collection of all
straight elements $K_\dagger'$ obtained from $K' \in \mathcal{T}_h$ with $\ba \in K'$ is a
standard shape-regular vertex patch of straight elements, with shape-regularity measure only
controlled by $M_{\mathcal{T}_h}$. For a straight vertex patch,
it is a standard result that the number of elements is bounded by a constant that
only depends on the shape-regularity measure, which shows that the number of elements
$K$ sharing $\ba$ is bounded. Since $K$ has only $d+1$ vertices and all $K' \in \mathcal{T}_h^K$,
must have at least one of those vertices, this concludes the existence of the constant
$N_{\mathcal{T}_h}$ controlling the cardinality of $\mathcal{T}_h^K$.

To prove~\eqref{eq_comparable_diameters}, we first consider two elements
$K_\pm \in \mathcal{T}_h$ sharing a face $F$. We let
$\widehat{F}_\pm := \mathcal{F}_{K_\pm}^{-1}(F)$, and denote by
$A_\pm: \widehat{F}_{d-1} \to \widehat{F}_\pm$ the affine maps from~\eqref{eq_compatible_mappings}.
We now pick two distinct points $\widehat{\bx}_{d-1},\widehat{\by}_{d-1} \in \widehat{F}_{d-1}$.
We let $\widehat{\bx}_\pm := A_\pm(\widehat{\bx}_{d-1})$
and    $\widehat{\by}_\pm := A_\pm(\widehat{\by}_{d-1})$.
Since the $A_\pm$ is a bijective affine mapping, we have
\begin{equation*}
|\widehat{\bx}_\pm-\widehat{\by}_\pm| \simeq |\widehat{\bx}-\widehat{\by}|.
\end{equation*}
We then let $\bx := \mathcal{F}_{K_\pm}(\widehat{\bx}_\pm)$
and         $\by := \mathcal{F}_{K_\pm}(\widehat{\by}_\pm)$,
where the definition is not ambiguous due to~\eqref{eq_compatible_mappings}.
We can then write on the one hand that
\begin{equation*}
|\bx-\by|
=
|\mathcal{F}_{K_+}(\widehat{\bx}_+)-\mathcal{F}_+(\widehat{\by}_+)|
\lesssim
h_{K_+}|\widehat{\bx}_+-\widehat{\by}_+|
\lesssim
h_{K_+}|\widehat{\bx}_{d-1}-\widehat{\by}_{d-1}|,
\end{equation*}
where the hidden constant only depends on $M_{K_+}$ and on the other hand that
\begin{equation*}
|\widehat{\bx}_{d-1}-\widehat{\by}_{d-1}|
\lesssim
|\widehat{\bx}_--\widehat{\by}_-|
=
|\mathcal{F}_{K_-}^{-1}(\bx)-\mathcal{F}_{K_-}^{-1}(\by)|
\lesssim
h_{K_-}^{-1}|\bx-\by|,
\end{equation*}
with a hidden constant solely depending on $M_{K_-}$. This is eventually leads to
$h_{K_-} \lesssim h_{K_+}$, and we can conclude that
\begin{equation}
\label{tmp_comparable_diameters}
h_{K_-} \simeq h_{K_+}
\end{equation}
by flipping the roles of $K_-$ and $K_+$, where the hidden constant
depends only on $M_{\mathcal{T}_h}$.

Now, if $K \in \mathcal{T}_h$ and $K' \in \mathcal{T}_h^K$, then
we can join $K$ to $K'$ by a chain $K = K_1,\dots,K_n = K'$ of
$n \leq N_{\mathcal{T}_h}$ elements connected through a face.
And we obtain~\eqref{eq_comparable_diameters} by applying~\eqref{tmp_comparable_diameters}.
This is indeed concludes the proof since the constants in~\eqref{tmp_comparable_diameters}
are then exponentiated at most to power $N_{\mathcal{T}_h}$.
\end{proof}

\subsection{Change of variables properties}

As per usual, our error estimates will rely on scaling arguments,
which ultimately hinge on changes of variables. Such reasoning is
completely standard for straight elements, and is extended to curved
elements with an effort of using minimal assumptions.

\begin{lemma}[Change of variable for Lebesgue spaces]
Let $K \in \mathcal{T}_h$. We have
\begin{equation}
\label{eq_change_of_variable}
\mu_K^{-1/r} h_K^{d/r} \|v \circ \mathcal{F}_K\|_{L^r(\widehat{K})}
\leq
\|v\|_{L^r(K)}
\leq
\mu_K^{1/r}h_K^{d/r} \|v \circ \mathcal{F}_K\|_{L^r(\widehat{K})}
\end{equation}
for all $1 \leq r \leq +\infty$ and all $v \in L^r(K)$.
\end{lemma}

\begin{proof}
The statement is obvious for $r = +\infty$. If $1 \leq r < +\infty$,
we introduce the short-hand notation
$\widehat{v} := v \circ \mathcal{F}_K \in L^q(\widehat{K})$.
We have
\begin{equation*}
\|v\|_{L^r(K)}^r
=
\int_K |v(\bx)|^r d\bx
=
\int_{\widehat K} |\widehat{v}(\widehat{\bx})|^r
|\det J\mathcal{F}_K(\widehat{\bx})| d\widehat{\bx}
\leq
\mu_K h_K^d \|\widehat{v}\|_{L^r(\widehat{K})}^r
\end{equation*}
which is the upper bound in~\eqref{eq_change_of_variable}.
The lower bound is obtained similarly, after observing that~\eqref{eq_muK}
implies that
\begin{equation*}
|\det(J\mathcal{F}_K^{-1})(\bx)| \leq \mu_K h_K^{-d}
\end{equation*}
for all $\bx \in K$.
\end{proof}

\begin{lemma}[Change of variables for Sobolev spaces]
Let $K \in \mathcal{T}_h$ and $v \in W^{1,q}(K)$. Then, we have
\begin{equation}
\label{eq_W1q_K_hK}
h_K |v|_{W^{1,q}(K)}
\leq
M_K \mu_K^{1/q} h_K^{d/q} |v \circ \mathcal{F}_K|_{W^{1,q}(\widehat{K})}.
\end{equation}
In addition, if~\eqref{eq_dJFK} holds true with $r = s$ 
and $v \in W^{s,q}(K)$ for some $s \geq 1$, then
\begin{equation}
\label{eq_Wsq_hK_K}
h_K^{d/q} |v \circ \mathcal{F}_K|_{W^{s,q}(\widehat{K})}
\lesssim
\left (\frac{h_K}{L} \right )^s
\sum_{t=1}^s
L^t|v|_{W^{t,q}(K)}
\end{equation}
where the hidden constant depends on $M_K$, $M_K(s)$, $s$ and $q$.
\end{lemma}

\begin{proof}
Let us first establish~\eqref{eq_W1q_K_hK}. For that, we let
$g := \nabla v \in L^q(K)$, $\widehat{v} := v \circ \mathcal{F}_K \in W^{1,q}(\widehat{K})$
and $\widehat{g} := \nabla \widehat{v} \in L^q(\widehat{K})$, so that
$\|g\|_{L^q(K)} = |v|_{W^{1,q}(K)}$ and
$\|\widehat{g}\|_{L^q(\widehat{K})} = |\widehat{v}|_{W^{1,q}(\widehat{K})}$.
Due to the chain rule, we have
\begin{equation*}
\widehat{g} = (J\mathcal{F}_K)^T \cdot (g \circ \mathcal{F}_K)
\end{equation*}
and it follows that
\begin{equation*}
\|g \circ \mathcal{F}_K\|_{L^q(\widehat{K})}
=
\|J\mathcal{F}_K^{-T} \cdot \widehat{g}\|_{L^q(\widehat{K})}
\leq
M_K h_K^{-1} \|\widehat{g}\|_{L^q(\widehat{K})}.
\end{equation*}
We then conclude the proof with the estimate
\begin{equation*}
\|g\|_{L^q(K)} \leq \mu_K^{1/q} h_K^{d/q}\|g \circ \mathcal{F}_K\|_{L^q(\widehat{K})}
\end{equation*}
which we obtain by applying~\eqref{eq_change_of_variable} component wise.

The estimate in~\eqref{eq_Wsq_hK_K} is established with similar arguments,
though more technical. For shortness, we do not repeat those here, and
follow the proof of~\cite[Lemma B.3]{chaumontfrelet_galkowski_spence_2024a}
where the estimate
\begin{equation*}
|\partial^\alpha(v \circ \mathcal{F}_K)|
\lesssim
\sum_{\substack{\beta \leq \alpha \\ |\beta| \geq 1}}
L^{|\beta|} \left (\frac{h}{L}\right )^{|\alpha|}
|(\partial^{\beta} v) \circ \mathcal{F}_K|
\end{equation*}
is established a.e. in $\widehat{K}$ for all multi-indices
$\alpha$ with $|\alpha| = s$ with hidden constant depending on
$\alpha$ and $M_K(s)$. It also follows from the triangular inequality that
\begin{equation*}
\|\partial^\alpha(v \circ \mathcal{F}_K)\|_{L^q(\widehat{K})}
\lesssim
\sum_{\substack{\beta \leq \alpha \\ |\beta| \geq 1}}
L^{|\beta|} \left (\frac{h}{L}\right )^{|\alpha|}
\|(\partial^{\beta} v) \circ \mathcal{F}_K\|_{L^q(\widehat{K)}}
\end{equation*}
and~\eqref{eq_Wsq_hK_K} follows from~\eqref{eq_change_of_variable}.
\end{proof}

\begin{lemma}[Change of exponent]
Assume that $1 \leq r \leq q \leq +\infty$. Then, for all $v \in L^q(K)$, we have
\begin{equation}
\label{eq_Lr_Lq}
\|v\|_{L^r(K)} \leq \mu_K^{1/r+1/q} h_K^{d/r-d/q} \|v\|_{L^q(K)}.
\end{equation}
\end{lemma}

\begin{proof}
Here too, we introduce the short-hand notation
$\widehat{v} := v \circ \mathcal{F}_K \in L^q(\widehat{K})$.
We are first going to show that
\begin{equation}
\label{tmp_Lr_Lq_hK}
\|\widehat{v}\|_{L^r(\widehat{K})} \leq \|\widehat{v}\|_{L^q(\widehat{K})}
\end{equation}
on the reference element. Let us first observe that~\eqref{tmp_Lr_Lq_hK}
clearly holds true if $r=q$. We may therefore assume that $1 \leq r < q$. Then,
since $q/r > 1$, H\"older
inequality gives that
\begin{equation*}
\|\widehat{v}\|_{L^r(\widehat{K})}^r
=
\||\widehat{v}|^r\|_{L^1(\widehat{K)}}
\leq
\|1\|_{L^{q/(q-r)}\widehat{K}}\||\widehat{v}|^r\|_{L^{q/r}(\widehat{K})}
=
|\widehat{K}|^{(q-r)/q}\|\widehat{v}\|_{L^q(\widehat{K})}^r,
\end{equation*}
which indeed gives~\eqref{tmp_Lr_Lq_hK} since $|\widehat{K}| \leq 1$ by assumption and
$(q-r)/(qr) \geq 0$. Now,~\eqref{eq_Lr_Lq} simply follows by applying the upper bound
in~\eqref{eq_change_of_variable},~\eqref{tmp_Lr_Lq_hK} and the lower bound in~\eqref{eq_change_of_variable}.
\end{proof}

\subsection{Properties of the element-wise projection}

Our quasi-interpolation operator hinges on an element-wise
$H^1$ orthogonal projection operator. Specifically, for $K \in \mathcal{T}_h$
and $v \in H^1(K)$ we define the elementwise projection $\Pi_Kv \in \mathcal{P}(K)$
by requiring that
\begin{equation*}
(v-\Pi_Kv,1)_K = 0,
\qquad
(\nabla(v-\Pi_Kv)),\nabla w_K) = 0 \quad \forall w_K \in \mathcal{P}(K).
\end{equation*}
In this section, we established some key properties of this elementwise projection.
These results are elementary for straight elements $K$ with affine mappings
$\mathcal{F}_K$. The extension to curved elements under suitable regularity
assumptions for the mapping is straightforward, but we include it here for completeness.

\begin{lemma}[Poincar\'e inequality]
For all $K \in \mathcal{T}_h$ and $w \in H^1(K)$ with $(w,1)_K = 0$, we have
\begin{equation}
\label{eq_poincare_mean}
\|w\|_{L^2(K)} \lesssim h_K|w|_{H^1(K)}
\end{equation}
where the hidden constant only depends on $M_K$. It follows in particular that
\begin{equation}
\label{eq_poincare}
\|v-\Pi_K v\|_{L^2(K)} \lesssim h_K|v-\Pi_K v|_{H^1(K)}
\end{equation}
for all $v \in H^1(K)$.
\end{lemma}

\begin{proof}
We first observe that since $(w,1)_K = 0$, we have
\begin{equation*}
\|w\|_{L^2(K)} \leq \|w-k\|_{L^2(K)}
\end{equation*}
for any choice of constant $k \in \mathbb C$. We now apply~\eqref{eq_change_of_variable},
which gives
\begin{equation*}
\|w-k\|_{L^2(K)}
\lesssim
h_K^{d/2} \|w \circ \mathcal{F}_K-k\|_{L^2(\widehat{K})}.
\end{equation*}
We can now choose $k := (w\circ\mathcal{F}_K,1)_{\widehat{K}}$ so that
\begin{equation*}
\|w \circ \mathcal{F}_K-k\|_{L^2(\widehat{K})}
\lesssim
|w \circ \mathcal{F}_K|_{H^1(\widehat{K})},
\end{equation*}
where the hidden constant only depends on the choice of the reference
element $\widehat{K}$. Then,~\eqref{eq_poincare_mean} follows from~\eqref{eq_Wsq_hK_K}.
\end{proof}

\begin{lemma}[Stability]
Let $K \in \mathcal{T}_h$. For all $q \geq 2$, we have
\begin{equation}
\label{eq_stability_Lq}
\|\Pi_K v\|_{L^q(K)} \lesssim \|v\|_{L^q(K)} + h_K|v|_{W^{1,q}(K)}
\end{equation}
and
\begin{equation}
\label{eq_stability_W1q}
|\Pi_K v|_{W^{1,q}(K)} \lesssim |v|_{W^{1,q}(K)}
\end{equation}
for all $v \in W^{1,q}(K)$, where the hidden constants depend on $p$, $q$ and $M_K$.
\end{lemma}

\begin{proof}
We start with~\eqref{eq_stability_Lq}. 
By the change of exponent
formula in~\eqref{eq_change_of_variable}, we first observe that
\begin{equation*}
\|\Pi_K v\|_{L^q(K)}
\lesssim
h_K^{d/q} \|(\Pi_K v) \circ \mathcal{F}_K\|_{L^q(\widehat{K})}.
\end{equation*}
Since $(\Pi_K v) \circ \mathcal{F}_K \in \mathcal{P}_p(\widehat{K})$
by assumption, we have
\begin{equation*}
\|(\Pi_K v) \circ \mathcal{F}_K \|_{L^q(\widehat{K})}
\lesssim
\|(\Pi_K v) \circ \mathcal{F}_K \|_{L^2(\widehat{K})}
\end{equation*}
since $\mathcal{P}_p(\widehat{K})$ is a finite dimensional space.
We apply again the change of variable formula from~\eqref{eq_change_of_variable},
leading to
\begin{equation*}
\|(\Pi_K v) \circ \mathcal{F}_K \|_{L^2(\widehat{K})}
\lesssim
h_K^{-d/2} \|\Pi_K v\|_{L^2(K)},
\end{equation*}
and all in all
\begin{equation}
\label{tmp_Lq_L2}
\|\Pi_K v\|_{L^q(K)} \lesssim h^{d/q-d/2} \|\Pi_K v \|_{L^2(K)}.
\end{equation}
We continue with the Poincar\'e inequality in~\eqref{eq_poincare}, giving us that
\begin{equation*}
\|\Pi_K v\|_{L^2(K)}
\leq
\|v\|_{L^2(K)}+\|v-\Pi_K v\|_{L^2(K)}
\lesssim
\|v\|_{L^2(K)} + h_K|v|_{H^1(K)}.
\end{equation*}
We now apply~\eqref{eq_Lr_Lq} to $v$ and each $\partial^\alpha v$ with $|\alpha| = 1$, yielding
\begin{equation}
\label{tmp_L2_W1q}
\|\Pi_K v\|_{L^2(K)}
\leq
C_q h_K^{d/2-d/q}\left (
\|v\|_{L^q(K)}+h_K|v|_{W^{1,q}(K)}
\right ),
\end{equation}
which is possible since $q \geq 2$. Combining~\eqref{tmp_Lq_L2} and~\eqref{tmp_L2_W1q}
proves~\eqref{eq_stability_Lq}.

We now turn to~\eqref{eq_stability_W1q}. For simplicity, we introduce the notation
$\widehat{g} := \nabla( \Pi_K v\circ\mathcal{F}_K) \in L^q(\widehat{K})$.
Due to~\eqref{eq_W1q_K_hK}, we have
\begin{equation*}
|\Pi_K v|_{W^{1,q}(K)}
\lesssim
h_K^{-1} h_K^{d/q} |\Pi_K v \circ \mathcal{F}_K|_{W^{1,q}(\widehat{K})}
=
h_K^{-1}h_K^{d/q}\|\widehat{g}\|_{L^q(\widehat{K})}.
\end{equation*}
Since each component $\widehat{g}$ is a polynomial of degree $p-1$, we
can again apply a finite-dimensional argument to show that
\begin{equation*}
|\Pi_K v|_{W^{1,q}(K)}
\lesssim
h_K^{-1} h_K^{d/q}\|\widehat{g}\|_{L^2(\widehat{K})}
=
h_K^{-1} h_K^{d/q}|\Pi_K v \circ \mathcal{F}_K|_{H^1(\widehat{K})}
\lesssim
h_K^{d/q-d/2} |\Pi_K v|_{H^1(K)}
\end{equation*}
where we emloyed~\eqref{eq_Wsq_hK_K} in the last estimate.
Since $\Pi_K$ is the $H^1(K)$ orthogonal projection onto $\mathcal{P}(K)$,
we now have
\begin{equation*}
|\Pi_K v|_{H^1(K)} \leq |v|_{H^1(K)},
\end{equation*}
and we conclude by applying~\eqref{eq_Lr_Lq} to each $\partial^\alpha v$, with $|\alpha| = 1$.
\end{proof}

\begin{theorem}[Error estimate]
\label{lemma_bramble_hilbert}
Let $1 \leq r \leq +\infty$, $\max(2,r) \leq q \leq +\infty$ and $0 \leq s \leq p$.
Assume that $K \in \mathcal{T}_h$ is $(s+1)$-regular, then
\begin{equation}
\label{eq_bramble_hilbert}
h_K^{-1} \|v-\Pi_K v\|_{L^r(K)} +|v-\Pi_K v|_{W^{1,r}(K)}
\lesssim
h_K^{d/r-d/q} \left (\frac{h_K}{L}\right )^s
\sum_{t=0}^{s} L^{t} |v|_{W^{t+1,q}(K)}
\end{equation}
for all $v \in W^{s+1,q}(K)$, where the hidden constants depend on $p$, $r$, $q$, $s$,
$M_K$ and $M_K(s+1)$.
\end{theorem}

\begin{proof}
Since $\Pi_K$ is a projection onto
$\mathcal{P}(K)$, we have
\begin{equation*}
\|v-\Pi_K v\|_{L^r(K)}
=
\|v-v_h-\Pi_K(v-v_h)\|_{L^r(K)}
\leq
\|v-v_h\|_{L^r(K)}+\|\Pi_K(v-v_h)\|_{L^r(K)},
\end{equation*}
and similarly
\begin{equation*}
|v-\Pi_K v|_{W^{1,r}(K)}
\leq
|v-v_h|_{W^{1,r}(K)}+|\Pi_K(v-v_h)|_{W^{1,r}(K)},
\end{equation*}
for all $v_h \in \mathcal{P}(K)$. It then follows from~\eqref{eq_stability_Lq}
and~\eqref{eq_stability_W1q} that
\begin{equation*}
h_K^{-1} \|v-\Pi_K v\|_{L^r(K)} + |v-\Pi_K v|_{W^{1,r}(K)}
\lesssim
h_K^{-1} \|v-v_h\|_{L^r(K)}+|v-v_h|_{W^{1,r}(K)}.
\end{equation*}
We next apply~\eqref{eq_change_of_variable} and~\eqref{eq_W1q_K_hK} which give
\begin{equation*}
\|v-\Pi_K v\|_{L^r(K)}
\lesssim h^{d/r} h_K^{-1}
\left (\|\widehat{v}-v_h \circ \mathcal{F}_K\|_{L^r(\widehat{K})}
+
|\widehat{v}-v_h \circ \mathcal{F}_K|_{W^{1,r}(\widehat{K})}
\right )
\end{equation*}
where $\widehat{v} := v \circ \mathcal{F}_K$. Since $v_h \in \mathcal{P}_K$
was arbitrary, it follows that
\begin{equation*}
\|v-\Pi_K v\|_{L^r(K)}
\lesssim
h^{d/r} h_K^{-1}
\min_{\widehat{v}_h \in \mathcal P_p(\widehat{K})}
\left (\|\widehat{v}-\widehat{v}_h\|_{L^r(\widehat{K})}
+
|\widehat{v}-\widehat{v}_h|_{W^{1,r}(\widehat{K})}
\right )
\end{equation*}
and it follows from~\cite[Theorem 3.1.1]{Cia02} that
\begin{equation*}
\|v-\Pi_K v\|_{L^r(K)}
\lesssim
h^{d/r} h_K^{-1}
|\widehat{v}|_{W^{s+1,q}(\widehat{K})}.
\end{equation*}
At that point,~\eqref{eq_bramble_hilbert} follows from~\eqref{eq_Wsq_hK_K}.
\end{proof}

\subsection{Construction of the quasi-interpolation operator}

We are almost ready to introduce a quasi-interpolation operator satisfying
Theorem~\ref{theorem_quasi_interpolant}. Before that, we need to recall some
standard facts about the nodal basis of $V_h$.

There exists a finite set of nodes $\mathcal{D}_h \subset \overline{\Omega} \setminus \GD$
and an associated basis $\{\phi_{\ba}\}_{\ba \in \mathcal{D}_h}$ of $V_h$ such that
\begin{equation*}
\phi_{\ba}(\bb) = \delta_{\ba,\bb}
\end{equation*}
for all $\ba,\bb \in \mathcal{D}_h$.
These basis functions are obtained in each physical element from reference basis functions,
meaning that there exists a finite set $\widehat{\Phi} \subset \mathcal{P}_p(\widehat{K})$
such that for each $\ba \in \mathcal{D}_h$ and $K \in \mathcal{T}_h$, we have
\begin{equation}
\label{eq_phia_mapping}
\phi_{\ba}|_K = \widehat{\phi} \circ \mathcal{F}_K
\end{equation}
for some $\widehat{\phi} \in \widehat{\Phi}$. The nodes in $\mathcal{D}_h$ are associated
with vertices, edges or interiors of cells of the mesh $\mathcal{T}_h$, and the basis functions
have local support.

We denote by $\mathcal{D}_h(K) := \mathcal{D}_h \cap K$ the set of nodes $\ba \in \mathcal{D}_h$
with non-vanishing basis function $\phi_{\ba}$ on $K$. In addtion, for each $\ba \in \mathcal{D}_h$
\begin{equation*}
\mathcal{T}_h^{\ba} := \left \{
K \in \mathcal{T}_h \; | \; \ba \in \mathcal{D}_h(K)
\right \}
\end{equation*}
is the set of elements in the support of $\phi_{\ba}$, and we use the short hand notation
$\sharp \ba$ for the cardinal of $\mathcal{T}_h^{\ba}$.

We are now ready to introduce our quasi-interpolation operator.
It is an ``Oswald'' interpolation operator in the spirit of the seminal
contribution~\cite{oswald_1999a}. Our quasi-interpolation operator
$I_h: H^1_{\GD}(\Omega) \to V_h$, is defined by requiring that
\begin{equation*}
I_h(v)(\ba) = \frac{1}{\sharp \ba} \sum_{K \in \mathcal{T}_h^{\ba}} (\Pi_K v)(\ba)
\end{equation*}
for all degree of freedom nodes $\ba \in \mathcal{D}_h$. In fact,
as we already alluded to, our analysis essentially extends the
one in~\cite{ern_guermond_2017a}.

The following shorthand notation will be useful. We introduce the broken projection
$I_h^\flat: H^1_{\GD}(\Omega) \to V_h^\flat$ by setting $I_h^\flat(v)|_K = \Pi_K(v)$ for all
$K \in \mathcal{T}_h$.

\begin{lemma}[Dofs mismatch]
Let $\ba \in \mathcal{D}_h$. Then, for all pairs $K_\pm \in \mathcal{T}_h^{\ba}$ of elements
sharing a face $F = \partial K_+ \cap \partial K_-$, we have
\begin{equation}
\label{eq_mismatch_face}
|(\Pi_{K_+}v)(\ba)-(\Pi_{K_-}v)(\ba)|
\lesssim
h_{K_+}^{1-d/2} |v-\Pi_{K_+}v|_{H^1(K_+)}
+
h_{K_-}^{1-d/2} |v-\Pi_{K_-} v|_{H^1(K_-)}
\end{equation}
for all $v \in H^1_{\GD}(\Omega)$ with hidden constants depending on $M_{K_\pm}$.
Besides, if $K \in \mathcal{T}_h$ and $\ba \in \mathcal{D}_h(K)$, we have
\begin{equation}
\label{eq_mismatch}
|(\Pi_Kv)(\ba)-(\Pi_{K'}v)(\ba)|
\lesssim
h_K^{1-d/2} |v-I_h^\flat(v)|_{H^1(\mathcal{T}_h^K)}
\end{equation}
for all $K' \in \mathcal{T}_h^{\ba}$. Here, the hidden constant
only depends on $M_{\mathcal{T}_h^K}$.
\end{lemma}

\begin{proof}
We first focus on~\eqref{eq_mismatch_face}, and the following notation will be useful.
We let $\widehat{F}_\pm := \mathcal{F}_{K_\pm}^{-1}(F) \in \widehat{\mathcal{F}}$
and denote by $A_\pm$ the bijective affine maps between $\widehat{F}_{d-1}$ and $\widehat{F}_{\pm}$
used in~\eqref{eq_compatible_mappings}. We further write $v_\pm := \Pi_{K_\pm} v \in \mathcal{P}(K_\pm)$,
$\widehat{v}_\pm := v_\pm \circ \mathcal{F}_{K_\pm} \circ A_\pm \in L^2(\widehat{F}_{d-1})$
and $\widehat{\ba} = (\mathcal{F}_{K_\pm} \circ A_\pm)^{-1}(\ba) \in \widehat{F}_{d-1}$.
We further point out that $\widehat{v}_\pm \in \mathcal{P}_p(\widehat{F}_{d-1)}$ since
$A_\pm$ are affine maps, and that the definition of $\widehat{\ba}$ is not ambiguous
due to~\eqref{eq_compatible_mappings}.

We start by writing that
\begin{equation*}
|v_+(\ba)-v_-(\ba)|
=
|\widehat{v}_+(\widehat \ba)-\widehat{v}_-(\widehat \ba)|
\leq
\|\widehat{v}_+-\widehat{v}_-\|_{L^\infty(\widehat{F}_{d-1})}
\lesssim
\|\widehat{v}_+-\widehat{v}_-\|_{L^2(\widehat{F}_{d-1})}
\end{equation*}
where we used a standard inverse inequality on the face $\widehat{F}_{d-1}$,
which is possible since $\widehat{v}_\pm$ belong to the finite-dimensional space
$\mathcal{P}_p(\widehat{F}_{d-1})$. We then introduce
$\widehat{v} := v \circ \mathcal{F}_{K_\pm} \circ A_\pm \in L^2(\widehat{F}_{d-1})$
and we write that
\begin{equation*}
\|\widehat{v}_+-\widehat{v}_-\|_{L^2(\widehat{F}_{d-1})}
=
\|\widehat{v}_+-\widehat{v}+(\widehat{v}-\widehat{v}_-)\|_{L^2(\widehat{F}_{d-1})}
\leq
\|\widehat{v}_+-\widehat{v}\| _{L^2(\widehat{F}_{d-1})}
+
\|\widehat{v}-\widehat{v}_-\|_{L^2(\widehat{F}_{d-1})}.
\end{equation*}
We can now write
\begin{equation*}
\|\widehat{v}-\widehat{v}_\pm\|_{L^2(\widehat{F}_{d-1})}
\lesssim
\|(v-v_\pm) \circ \mathcal{F}_{K_\pm}\|_{L^2(\widehat{F}_\pm)}
\end{equation*}
since the map $A_\pm$ is affine and invertible and only depends on the choice
of reference simplices. A trace inequality on the reference $d$-simplex
then gives that
\begin{equation*}
\|\widehat{v}-\widehat{v}_\pm\|_{L^2(\widehat{F}_{d-1})}
\lesssim
\|(v-v_\pm) \circ \mathcal{F}_{K_\pm}\|_{L^2(\widehat{K})}
+
|(v-v_\pm) \circ \mathcal{F}_{K_\pm}|_{H^1(\widehat{K})},
\end{equation*}
and we conclude with~\eqref{eq_change_of_variable},~\eqref{eq_Wsq_hK_K} and~\eqref{eq_poincare}.

We now prove~\eqref{eq_mismatch}. If $K = K'$, there is nothing to show, and
if $K$ and $K'$ share a face, then~\eqref{eq_mismatch} is a direct consequence
of~\eqref{eq_mismatch_face}. In the general case, however, we might only have
$K \cap K' = \ba$ if $\ba$ is vertex of $\mathcal{T}_h$.
Nevertheless, in this case, there exists
a chain of $n$ elements $K = K_1, \ldots, K_n = K'$ that are all connected by a face,
and~\eqref{eq_mismatch} then follows from the triangular inequality
\begin{equation*}
|(\Pi_K v)(\ba)-(\Pi_{K'}v)(\ba)|
\leq
\sum_{j=1}^{n-1} |(\Pi_{K_{j+1}} v)(\ba)-(\Pi_{K_j}v)(\ba)|
\end{equation*}
by applying~\eqref{eq_mismatch_face} to each summand and observing that
$h_{K_j} \lesssim h_K$ for $j=1,\dots,n$ due to~\eqref{eq_comparable_diameters}
\end{proof}

\begin{lemma}[Basis function scaling]
For all $\ba \in \mathcal{D}_h$ and $K \in \mathcal{T}_h$, we have
\begin{equation}
\label{eq_phia_scaling}
\|\phi_{\ba}\|_{L^r(K)} + h_K|\phi_{\ba}|_{W^{1,r}(K)}
\lesssim
h_K^{d/r}
\end{equation}
where the hidden constant depends on $p$ and $M_K$.
\end{lemma}

\begin{proof}
It is a simple combination of~\eqref{eq_phia_mapping} and the change
of coordinates formulas~\eqref{eq_change_of_variable}
and~\eqref{eq_W1q_K_hK}.
\end{proof}

\begin{theorem}[Error in conformity enforcement]
Let $v \in H^1_{\GD}(\Omega)$. Then, for all $K \in \mathcal{T}_h$, we have
\begin{equation}
\label{eq_Ihb_Ih}
h_K^{-1}\|I_h^\flat(v)-I_h(v)\|_{L^r(K)}+|I_h^\flat(v)-I_h(v)|_{W^{1,r}(K)}
\lesssim
h_K^{d/r-d/2} |v-I_h^\flat(v)|_{H^1(\mathcal{T}_h^K)}
\end{equation}
with a hidden constant solely depending on $M_{\mathcal{T}_h^K}$.
\end{theorem}

\begin{proof}
Let us first obersve that for a fixed $K \in \mathcal{T}_h$, we have
\begin{align*}
\left .\left ( I^\flat(v)-I_h(v) \right )\right |_K
&=
\sum_{\ba \in \mathcal{D}_h(K)}
\left \{
(\Pi_K v)(\ba)
-
\frac{1}{\sharp \ba}
\sum_{K' \in \mathcal{T}_h^{\ba}}
(\Pi_{K'}v)(\ba)
\right \}
\phi_{\ba}|_K
\\
&=
\sum_{\ba \in \mathcal{D}_h(K)}
\frac{1}{\sharp \ba}
\sum_{K' \in \mathcal{T}_h^{\ba}}
\left \{
(\Pi_K v)(\ba)
-
(\Pi_{K'}v)(\ba)
\right \}
\phi_{\ba}|_K,
\end{align*}
so that
\begin{equation*}
\|I_h^\flat(v)-I_h(v)\|_{L^r(K)}
\leq
\sum_{\ba \in \mathcal{D}_h(K)}
\sum_{K' \in \mathcal{T}_h^{\ba}}
|
(\Pi_K v)(\ba)
-
(\Pi_{K'}v)(\ba)
|
\|\phi_{\ba}\|_{L^r(K)}
\end{equation*}
and
\begin{equation*}
|I_h^\flat(v)-I_h(v)|_{W^{1,r}(K)}
\leq
\sum_{\ba \in \mathcal{D}_h(K)}
\sum_{K' \in \mathcal{T}_h^{\ba}}
|
(\Pi_K v)(\ba)
-
(\Pi_{K'}v)(\ba)
|
|\phi_{\ba}|_{W^{1,r}(K)}.
\end{equation*}
It then follows from~\eqref{eq_phia_scaling} that
\begin{multline*}
h_K^{-1}\|I_h^\flat(v)-I_h(v)\|_{L^r(K)} + |I_h^\flat(v)-I_h(v)|_{W^{1,r}(K)}
\lesssim
\\
h_K^{d/r-1}
\sum_{\ba \in \mathcal{D}_h(K)}
\sum_{K' \in \mathcal{T}_h^{\ba}}
|
(\Pi_K v)(\ba)
-
(\Pi_{K'}v)(\ba)
|,
\end{multline*}
and~\eqref{eq_Ihb_Ih} follows from~\eqref{eq_mismatch} and Lemma~\ref{lemma_patch}.
\end{proof}

We are now ready to establish our main result.

\begin{proof}[Proof of Theorem~\ref{theorem_quasi_interpolant}]
We first observe that~\eqref{eq_error_estimate_global} follows from~\eqref{eq_error_estimate}
due to Lemma~\ref{lemma_patch} by observing that
\begin{multline*}
\sum_{K \in \mathcal{T}_h}
|\phi|_{W^{t+1,q}(\mathcal{T}_h^K)}^q
=
\sum_{K \in \mathcal{T}_h}
\sum_{\substack{K' \in \mathcal{T}_h \\ K' \cap K \neq \emptyset}}
|\phi|_{W^{t+1,q}(K')}^q
\\
=
\sum_{K' \in \mathcal{T}_h}
\sum_{\substack{K \in \mathcal{T}_h \\ K \cap K' \neq \emptyset}}
|\phi|_{W^{t+1,q}(K')}^q
\leq
N |\phi|_{W^{t+1,q}(\mathcal{T}_h)}^q
\lesssim
|\phi|_{W^{t+1,q}(\mathcal{T}_h)}^q
\end{multline*}
for all $t \geq 0$, $q \geq 2$ and $\phi \in W^{t+1,q}(\mathcal{T}_h)$.

We therefore fix $K \in \mathcal{T}_h$ and establish~\eqref{eq_error_estimate}.
We start with the triangular inequality
\begin{align*}
h_K^{-1}\|v-I_h(v)\|_{L^r(K)}
+
|v-I_h(v)|_{W^{1,r}(K)}
&\leq
h_K^{-1}\|v-I_h^\flat(v)\|_{L^r(K)}
+
|v-I_h^\flat(v)|_{W^{1,r}(K)}
\\
&+
h_K^{-1}\|I_h^\flat(v)-I_h(v)\|_{L^r(K)}
+
|I_h^\flat(v)-I_h(v)|_{W^{1,r}(K)}.
\end{align*}
Since $v-I_h^\flat(v)|_K = v-\Pi_K v$, we can immediatly
apply~\eqref{eq_bramble_hilbert} for the corresonding terms.
For the other terms, we involve~\eqref{eq_Ihb_Ih}, to write
\begin{equation*}
h_K^{-1}\|I_h^\flat(v)-I_h(v)\|_{L^r(K)}
+
|I_h^\flat(v)-I_h(v)|_{W^{1,r}(K)}
\lesssim
h_K^{d/r-d/2}|v-I_h^\flat(v)|_{H^1(\mathcal{T}_h^K)}.
\end{equation*}
At that point,~\eqref{eq_error_estimate} follows from~\eqref{eq_bramble_hilbert}.
\end{proof}


\begin{thebibliography}{10}

\bibitem{Abr04}
A.~A. Abrikosov.
\newblock Nobel lecture: Type-{II} superconductors and the vortex lattice.
\newblock {\em Rev. Mod. Phys.}, 76(3,1):975--979, 2004.

\bibitem{agmon_douglis_nirenberg_1959a}
S.~Agmon, A.~Douglis, and L.~Nirenberg.
\newblock Estimates near the boundary for solutions of elliptic partial
  differential equations satisfying general boundary conditions. i.
\newblock {\em Comm. Pure Appl. Math.}, 12(4):623--727, 1959.

\bibitem{AHYY26}
Y.~Ai, P.~Henning, M.~Yadav, and S.~Yuan.
\newblock Riemannian conjugate {S}obolev gradients and their application to
  compute ground states of {BEC}s.
\newblock {\em J. Comput. Appl. Math.}, 473:Paper No. 116866, 17, 2026.

\bibitem{APS22}
R.~Altmann, D.~Peterseim, and T.~Stykel.
\newblock Energy-adaptive {R}iemannian optimization on the {S}tiefel manifold.
\newblock {\em ESAIM Math. Model. Numer. Anal.}, 56(5):1629--1653, 2022.

\bibitem{BaS97}
I.~M. Babu{\v s}ka and S.~A. Sauter.
\newblock Is the pollution effect of the {FEM} avoidable for the {H}elmholtz
  equation considering high wave numbers?
\newblock {\em SIAM J. Numer. Anal.}, 34(6):2392--2423, 1997.

\bibitem{bernardi_1989a}
C.~Bernardi.
\newblock Optimal finite-element interpolation on curved domains.
\newblock {\em SIAM J. Numer. Anal.}, 26(5):1212--1240, 1989.

\bibitem{bernkopf_chaumontfrelet_melenk_2025a}
M.~Bernkopf, T.~Chaumont-Frelet, and J.~Melenk.
\newblock Wavenumber-explicit stability and convergence analysis of $hp$ finite
  element discretizations of helmholtz problems in piecewise smooth media.
\newblock {\em Math. Comp.}, 94(351):73--122, 2025.

\bibitem{BDH25}
M.~Blum, C.~D\"oding, and P.~Henning.
\newblock Vortex-capturing multiscale spaces for the {G}inzburg-{L}andau
  equation.
\newblock {\em Multiscale Model. Simul.}, 23(1):339--373, 2025.

\bibitem{BrezisMironescu18}
H.~Brezis and P.~Mironescu.
\newblock Gagliardo-{N}irenberg inequalities and non-inequalities: the full
  story.
\newblock {\em Ann. Inst. H. Poincar\'e{} C Anal. Non Lin\'eaire},
  35(5):1355--1376, 2018.

\bibitem{CantorMatovsky85}
M.~R. Cantor and J.~C. Matovsky.
\newblock Helmholtz decomposition of {$W^{p,s}$} vector fields.
\newblock In {\em Differential geometry, calculus of variations, and their
  applications}, volume 100 of {\em Lecture Notes in Pure and Appl. Math.},
  pages 139--147. Dekker, New York, 1985.

\bibitem{chaumontfrelet_galkowski_spence_2024a}
T.~Chaumont-Frelet, J.~Galkowski, and E.~Spence.
\newblock Sharp error bounds for edge-element discretisations of the
  high-frequency {M}axwell equations.
\newblock arXiv:2408.04507, 2024.

\bibitem{CFN20}
T.~Chaumont-Frelet and S.~Nicaise.
\newblock Wavenumber explicit convergence analysis for finite element
  discretizations of general wave propagation problems.
\newblock {\em IMA J. Numer. Anal.}, 40(2):1503--1543, 2020.

\bibitem{chaumontfrelet_spence_2024a}
T.~Chaumont-Frelet and E.~Spence.
\newblock The geometric error is less than the pollution error when solving the
  high-frequency {H}elmholtz equation with high-order {FEM} on curved domains.
\newblock arXiv:2401.16413, 2024.

\bibitem{Chen97}
Z.~Chen.
\newblock Mixed finite element methods for a dynamical {G}inzburg-{L}andau
  model in superconductivity.
\newblock {\em Numer. Math.}, 76(3):323--353, 1997.

\bibitem{CLLZ24}
Z.~Chen, J.~Lu, Y.~Lu, and X.~Zhang.
\newblock On the convergence of {S}obolev gradient flow for the
  {G}ross-{P}itaevskii eigenvalue problem.
\newblock {\em SIAM J. Numer. Anal.}, 62(2):667--691, 2024.

\bibitem{Cia02}
P.~G. Ciarlet.
\newblock {\em The finite element method for elliptic problems.}, volume~40 of
  {\em Classics Appl. Math.}
\newblock SIAM, Philadelphia, 2002.

\bibitem{DaK10}
I.~Danaila and P.~Kazemi.
\newblock A new {S}obolev gradient method for direct minimization of the
  {G}ross-{P}itaevskii energy with rotation.
\newblock {\em SIAM J. Sci. Comput.}, 32(5):2447--2467, 2010.

\bibitem{DaP17}
I.~Danaila and B.~Protas.
\newblock Computation of ground states of the {G}ross-{P}itaevskii functional
  via {R}iemannian optimization.
\newblock {\em SIAM J. Sci. Comput.}, 39(6):B1102--B1129, 2017.

\bibitem{DDH24}
C.~D\"oding, B.~D\"orich, and P.~Henning.
\newblock A multiscale approach to the stationary ginzburg-landau equations of
  superconductivity.
\newblock ArXiv e-print 2409.12023, 2024.

\bibitem{DoeHe24}
B.~D\"orich and P.~Henning.
\newblock Error bounds for discrete minimizers of the {G}inzburg-{L}andau
  energy in the high-{$\kappa$} regime.
\newblock {\em SIAM J. Numer. Anal.}, 62(3):1313--1343, 2024.

\bibitem{Du94b}
Q.~Du.
\newblock Finite element methods for the time-dependent {G}inzburg-{L}andau
  model of superconductivity.
\newblock {\em Comput. Math. Appl.}, 27(12):119--133, 1994.

\bibitem{Du97}
Q.~Du.
\newblock Discrete gauge invariant approximations of a time dependent
  {G}inzburg-{L}andau model of superconductivity.
\newblock {\em Math. Comp.}, 67(223):965--986, 1998.

\bibitem{DuGuPe}
Q.~Du, M.~D. Gunzburger, and J.~S. Peterson.
\newblock Analysis and approximation of the {G}inzburg-{L}andau model of
  superconductivity.
\newblock {\em SIAM Rev.}, 34(1):54--81, 1992.

\bibitem{ern_guermond_2017a}
A.~Ern and J.~Guermond.
\newblock Finite element quasi-interpolation and best approximation.
\newblock {\em ESAIM Math. Model. Numer. Anal.}, 51:1367--1385, 2017.

\bibitem{ern_guermond_2021a}
A.~Ern and J.~Guermond.
\newblock {\em Finite elements {I}: basic theory and practice}.
\newblock Springer, 2021.

\bibitem{galkowski2025sharp}
J.~Galkowski and E.~A. Spence.
\newblock Sharp preasymptotic error bounds for the {H}elmholtz {$h$}-{FEM}.
\newblock {\em SIAM J. Numer. Anal.}, 63(1):1--22, 2025.

\bibitem{GJX19}
H.~Gao, L.~Ju, and W.~Xie.
\newblock A stabilized semi-implicit {E}uler gauge-invariant method for the
  time-dependent {G}inzburg-{L}andau equations.
\newblock {\em J. Sci. Comput.}, 80(2):1083--1115, 2019.

\bibitem{GaoS18}
H.~Gao and W.~Sun.
\newblock Analysis of linearized {G}alerkin-mixed {FEM}s for the time-dependent
  {G}inzburg-{L}andau equations of superconductivity.
\newblock {\em Adv. Comput. Math.}, 44(3):923--949, 2018.

\bibitem{GaPe01}
J.~J. Garc\'ia-Ripoll and V.~M. P\'erez-Garc\'ia.
\newblock Optimizing {S}chr\"odinger functionals using {S}obolev gradients:
  applications to quantum mechanics and nonlinear optics.
\newblock {\em SIAM J. Sci. Comput.}, 23(4):1316--1334, 2001.

\bibitem{Ginzburg1955}
V.~L. Ginzburg.
\newblock On the {T}heory of {S}uperconductivity.
\newblock {\em Nuovo Cimento}, 2(6):1235--1250, 1955.

\bibitem{HeP20}
P.~Henning and D.~Peterseim.
\newblock Sobolev gradient flow for the {G}ross-{P}itaevskii eigenvalue
  problem: global convergence and computational efficiency.
\newblock {\em SIAM J. Numer. Anal.}, 58(3):1744--1772, 2020.

\bibitem{lafontaine2022wavenumber}
D.~Lafontaine, E.~Spence, and J.~Wunsch.
\newblock Wavenumber-explicit convergence of the {$hp$-FEM} for the full-space
  heterogeneous helmholtz equation with smooth coefficients.
\newblock {\em Comput. Math. Appl.}, 113:59--69, 2022.

\bibitem{Landau1965}
L.~Landau.
\newblock On the {T}heory of {S}uperconductivity.
\newblock In D.~{ter Haar}, editor, {\em Collected Papers of L.D. Landau},
  pages 217--225. Pergamon, 1965.

\bibitem{lenoir_1986a}
M.~Lenoir.
\newblock Optimal isoparametric finite elements and error estimates for domains
  involving curved boundaries.
\newblock {\em SIAM J. Numer. Anal.}, 23(3):562--580, 1986.

\bibitem{Li17}
B.~Li.
\newblock Convergence of a decoupled mixed {FEM} for the dynamic
  {G}inzburg-{L}andau equations in nonsmooth domains with incompatible initial
  data.
\newblock {\em Calcolo}, 54(4):1441--1480, 2017.

\bibitem{LiZ15}
B.~Li and Z.~Zhang.
\newblock A new approach for numerical simulation of the time-dependent
  {G}inzburg-{L}andau equations.
\newblock {\em J. Comput. Phys.}, 303:238--250, 2015.

\bibitem{LiZ17}
B.~Li and Z.~Zhang.
\newblock Mathematical and numerical analysis of the time-dependent
  {G}inzburg-{L}andau equations in nonconvex polygons based on {H}odge
  decomposition.
\newblock {\em Math. Comp.}, 86(306):1579--1608, 2017.

\bibitem{li2025higher}
Y.~Li and H.~Wu.
\newblock Higher-order {FEM} and {CIP}-{FEM} for {H}elmholtz equation with high
  wave number and perfectly matched layer truncation.
\newblock {\em J. Sci. Comput.}, 104(2):Paper No. 47, 25, 2025.

\bibitem{MaQ23}
L.~Ma and Z.~Qiao.
\newblock An energy stable and maximum bound principle preserving scheme for
  the dynamic {G}inzburg-{L}andau equations under the temporal gauge.
\newblock {\em SIAM J. Numer. Anal.}, 61(6):2695--2717, 2023.

\bibitem{mclean_2000a}
W.~C.~H. McLean.
\newblock {\em Strongly elliptic systems and boundary integral equations}.
\newblock Cambridge university press, 2000.

\bibitem{MeS10}
J.~M. Melenk and S.~Sauter.
\newblock Convergence analysis for finite element discretizations of the
  {H}elmholtz equation with {D}irichlet-to-{N}eumann boundary conditions.
\newblock {\em Math. Comp.}, 79(272):1871--1914, 2010.

\bibitem{Mikhailov}
V.~P. Mikha{\u i}lov.
\newblock {\em Partial differential equations}.
\newblock ``Mir'', Moscow; distributed by Imported Publications, Chicago, IL,
  1978.
\newblock Translated from the Russian by P. C. Sinha.

\bibitem{Neu97}
J.~W. Neuberger.
\newblock {\em Sobolev gradients and differential equations}, volume 1670 of
  {\em Lecture Notes in Mathematics}.
\newblock Springer-Verlag, Berlin, 1997.

\bibitem{nezza_palatucci_valdinoci_2012a}
E.~Nezza, G.~Palatucci, and E.~Valdinoci.
\newblock Hitchhiker's guide to the fractional {S}obolev spaces.
\newblock {\em Bull. Sci. math.}, 136:521--573, 2012.

\bibitem{NoPr23}
A.~Novruzi and B.~Protas.
\newblock An accelerated {S}obolev gradient method for unconstrained
  optimization problems based on variable inner products.
\newblock {\em J. Comput. Appl. Math.}, 420:Paper No. 114833, 18, 2023.

\bibitem{oswald_1999a}
P.~Oswald.
\newblock On a {BPX}-preconditioner for {P1} elements.
\newblock {\em Computing}, 51(2):125--133, 1993.

\bibitem{Pet17}
D.~Peterseim.
\newblock Eliminating the pollution effect in {H}elmholtz problems by local
  subscale correction.
\newblock {\em Math. Comp.}, 86(305):1005--1036, 2017.

\bibitem{polak1969}
E.~Polak and G.~Ribi\`ere.
\newblock Note sur la convergence de m\'ethodes de directions conjugu\'ees.
\newblock {\em Rev. Fran\c caise Informat. Recherche Op\'erationnelle},
  3(16):35--43, 1969.

\bibitem{PoRa94}
J.~Pousin and J.~Rappaz.
\newblock Consistency, stability, a priori and a posteriori errors for
  {P}etrov-{G}alerkin methods applied to nonlinear problems.
\newblock {\em Numer. Math.}, 69(2):213--231, 1994.

\bibitem{Sch74}
A.~H. Schatz.
\newblock An observation concerning {R}itz-{G}alerkin methods with indefinite
  bilinear forms.
\newblock {\em Math. Comp.}, 28:959--962, 1974.

\bibitem{tartar_2007a}
T.~Tartar.
\newblock {\em An introduction to Sobolev spaces and interpolation spaces},
  volume~3.
\newblock Springer Science \& Business Media, 2007.

\end{thebibliography}
\end{document}